\documentclass[11pt,a4paper]{amsart}
\pdfoutput=1
\usepackage[centering, margin=1in]{geometry}
\usepackage{accents,amsmath,amsfonts,amsthm,amssymb,enumitem,graphicx,mathrsfs,mathtools}
\usepackage[stretch=10]{microtype}
\usepackage[unicode]{hyperref}
\usepackage{tikz}
\usepackage{tikz-cd}
\usepackage{xcolor}
\definecolor{dark-red}{rgb}{0.4,0.15,0.15}
\definecolor{dark-blue}{rgb}{0.15,0.15,0.4}
\definecolor{medium-blue}{rgb}{0,0,0.5}
\hypersetup{
	pdftitle={Sparse Equidistribution of Geometric Invariants of Real Quadratic Fields},
	pdfauthor={Peter Humphries and Asbj\o{}rn Christian Nordentoft},
	pdfnewwindow=true,
	colorlinks, linkcolor={dark-red},
	citecolor={dark-blue}, urlcolor={medium-blue}
}
\allowdisplaybreaks
\newcommand{\A}{\mathbb{A}}
\renewcommand{\aa}{\mathfrak{a}}
\newcommand{\bb}{\mathfrak{b}}
\newcommand{\BB}{\mathcal{B}}
\newcommand{\C}{\mathbb{C}}
\newcommand{\CC}{\mathcal{C}}
\newcommand{\Cl}{\mathrm{Cl}}
\newcommand{\Cscr}{\mathscr{C}}
\newcommand{\dee}{\partial}
\newcommand{\df}{\mathfrak{d}}
\newcommand{\EE}{\mathcal{E}}
\newcommand{\e}{\varepsilon}
\newcommand{\FF}{\mathcal{F}}
\newcommand{\GG}{\mathcal{G}}
\newcommand{\Hb}{\mathbb{H}}
\newcommand{\hol}{\mathrm{hol}}
\newcommand{\LL}{\mathcal{L}}
\newcommand{\lf}{\mathfrak{l}}

\newcommand{\N}{\mathbb{N}}
\newcommand{\NN}{\mathcal{N}}
\newcommand{\OO}{\mathcal{O}}
\newcommand{\Pb}{\mathbb{P}}

\newcommand{\PP}{\mathcal{P}}
\newcommand{\Pscr}{\mathscr{P}}
\newcommand{\Q}{\mathbb{Q}}
\newcommand{\QQ}{\mathcal{Q}}
\newcommand{\R}{\mathbb{R}}

\newcommand{\T}{\mathbb{T}}
\newcommand{\WW}{\mathcal{W}}
\newcommand{\Z}{\mathbb{Z}}
\DeclareMathOperator{\ad}{ad}
\DeclareMathOperator{\arc}{arc}
\DeclareMathOperator{\den}{den}
\DeclareMathOperator{\Dgp}{D}
\DeclareMathOperator{\Gen}{Gen}
\DeclareMathOperator{\GL}{GL}
\DeclareMathOperator{\Ht}{ht}
\DeclareMathOperator{\Mat}{Mat}
\DeclareMathOperator{\Norm}{N}
\DeclareMathOperator{\Ogp}{O}
\DeclareMathOperator{\PGL}{PGL}
\DeclareMathOperator{\PSL}{PSL}
\DeclareMathOperator{\Ram}{Ram}
\DeclareMathOperator{\sgn}{sgn}
\DeclareMathOperator{\SL}{SL}
\DeclareMathOperator{\SO}{SO}
\DeclareMathOperator{\St}{St}
\DeclareMathOperator{\Tr}{Tr}
\DeclareMathOperator{\vol}{vol}
\DeclareMathOperator{\wl}{wl}
\DeclareMathOperator{\Zgp}{Z}
\numberwithin{equation}{section}
\newtheorem{theorem}[equation]{Theorem}

\newtheorem{corollary}[equation]{Corollary}
\newtheorem{lemma}[equation]{Lemma}
\newtheorem{proposition}[equation]{Proposition}
\newtheorem{question}[equation]{Question}
\theoremstyle{remark}
\newtheorem{remark}[equation]{Remark}
\newtheorem{example}[equation]{Example}
\theoremstyle{definition}
\newtheorem{definition}[equation]{Definition}
\begin{document}

\title[Sparse Equidistribution of Geometric Invariants of Real Quadratic Fields]{Sparse Equidistribution of Geometric Invariants of Real Quadratic Fields}

\author{Peter Humphries}

\address{Department of Mathematics, University of Virginia, Charlottesville, VA 22904, USA}

\email{\href{mailto:pclhumphries@gmail.com}{pclhumphries@gmail.com}}

\urladdr{\href{https://sites.google.com/view/peterhumphries/}{https://sites.google.com/view/peterhumphries/}}

\author{Asbj\o{}rn Christian Nordentoft}

\address{Department of Mathematical Sciences, University of Copenhagen, Universitetsparken 5, 2100 Copenhagen \O{}, Denmark}

\email{\href{mailto:acnordentoft@outlook.com}{acnordentoft@outlook.com}, \href{mailto:nordentoft@math.ku.dk}{nordentoft@math.ku.dk}}

\urladdr{\href{https://sites.google.com/view/asbjornnordentoft/}{https://sites.google.com/view/asbjornnordentoft/}}



\thanks{The first author was supported by the National Science Foundation grant DMS-2302079 and the Simons Foundation (award 965056). The second author was supported by a grant from the Independent Research Fund Denmark DFF (1025-00020B) and by a public grant from Fondation Math\'{e}matique Jacques Hadamard.}

\begin{abstract}
Duke, Imamo\={g}lu, and T\'{o}th have recently constructed a new geometric invariant, a hyperbolic orbifold, associated to each narrow ideal class of a real quadratic field. Furthermore, they have shown that the projection of these hyperbolic orbifolds onto the modular surface $\Gamma \backslash \Hb$ equidistributes on average over a genus of the narrow class group as the fundamental discriminant $D$ of the real quadratic field tends to infinity.

We extend this construction of hyperbolic orbifolds to allow for a level structure, akin to Heegner points and closed geodesics of level $q$. Additionally, we refine this equidistribution result in several directions. First, we investigate sparse equidistribution in the level aspect, where we prove the equidistribution of level $q$ hyperbolic orbifolds when restricted to a translate of $\Gamma \backslash \Hb$ in $\Gamma_0(q) \backslash \Hb$, which presents some new interesting features. Second, we explore sparse equidistribution in the subgroup aspect, namely equidistribution on average over small subgroups of the narrow class group. Third, we prove small scale equidistribution and give upper bounds for the discrepancy.

Behind these refinements is a new interpretation of the Weyl sums arising in these equidistribution problems in terms of ad\`{e}lic period integrals, which in turn are related to Rankin--Selberg $L$-functions via Waldspurger's formula. The key remaining inputs are hybrid subconvex bounds for these $L$-functions and a certain homological version of the sup-norm problem.
\end{abstract}

\maketitle

\section{Introduction}

\subsection{Equidistribution of Hyperbolic Orbifolds}

Let $E \coloneqq \Q(\sqrt{D})$ be a real quadratic number field, where $D > 1$ is a positive fundamental discriminant. Associated to each narrow ideal class $A$ of the narrow class group $\Cl_D^+$ of $E$ is a closed geodesic $\CC_A$ on the modular surface $\Gamma \backslash \Hb$, where $\Gamma \coloneqq \SL_2(\Z)$ denotes the modular group. In \cite{DIT16}, Duke, Imamo\={g}lu, and T\'{o}th introduced a new geometric invariant associated to each narrow ideal class $A$, a hyperbolic orbifold $\Gamma_A \backslash \NN_A$ with boundary given by the closed geodesic $\CC_A$. The group $\Gamma_A \subset \PSL_2(\Z)$ is a Fuchsian group of the second kind whose construction is given in terms of certain invariants of $A$, while $\NN_A \subset \Hb$ is the Nielsen region of $\Gamma_A$, namely the smallest nonempty $\Gamma_A$-invariant open convex subset of $\Hb$.

Duke, Imamo\={g}lu, and T\'{o}th additionally showed that these hyperbolic orbifolds equidistribute as $D$ tends to infinity when projected onto the modular surface. For each positive fundamental discriminant $D$, one chooses a genus $G_D$ in the group of genera $\Gen_D \coloneqq \Cl_D^+ / (\Cl_D^+)^2$, so that $G_D$ is a coset $C (\Cl_D^+)^2$ of narrow ideal classes for some $C \in \Cl_D^+$; then for every continuity set $B \subset \Gamma \backslash \Hb$,
\begin{equation}
\label{eqn:equidistribution}
\frac{\sum_{A \in G_D} \vol(\FF_A \cap \Gamma B)}{\sum_{A \in G_D} \vol(\FF_A)} = \frac{\vol(B)}{\vol(\Gamma \backslash \Hb)} + o_B(1)
\end{equation}
as $D$ tends to infinity through fundamental discriminants \cite[Theorem 2]{DIT16}. Here $\FF_A$ denotes a canonical fundamental domain for $\Gamma_A \backslash \NN_A$, while the volume measure on the upper half-plane $\Hb \ni z = x + iy$ is $d\mu(z) = y^{-2} \, dx \, dy$, so that $\vol(\Gamma \backslash \Hb) = \pi/3$. This equidistribution theorem can be viewed as an analogue of Duke's celebrated result on the equidistribution of closed geodesics and of Heegner points on the modular surface and additionally of lattice points on the sphere \cite{Duk88}.

\subsection{Sparse Equidistribution in the Level Aspect}

We generalise Duke, Imamo\={g}lu, and T\'{o}th's construction of hyperbolic orbifolds in the level aspect. Let $q$ be an odd prime that splits in $E$. In \hyperref[sect:hypq]{Section \ref*{sect:hypq}}, we construct hyperbolic orbifolds of level $q$, denoted by $\Gamma_A(q) \backslash \NN_A(q)$, and a canonical fundamental domain $\FF_A(q)\subset \Hb$. This extends the level $1$ construction of the hyperbolic orbifolds $\Gamma_A \backslash \NN_A$ and their canonical fundamental domains $\FF_A$ introduced in \cite{DIT16}. These level $q$ hyperbolic orbifolds are analogous to Heegner points of level $q$ \cite[p.~499]{GKZ87} and closed geodesics of level $q$ \cite[p.~500]{GKZ87} (cf.\ \cite[Section 1]{Dar94}), but with several key new features. We refer to \hyperref[fig:F_A1]{Figures \ref*{fig:F_A1}}, \ref{fig:F_A2}, and \ref{fig:1} in \hyperref[sect:hypq]{Section \ref*{sect:hypq}} for some illustrative examples.

The level $q$ modular surface $\Gamma_0(q) \backslash \Hb$ may be written as
\[\Gamma_0(q) \backslash \Hb = \bigcup_{\omega_q \in \Gamma / \Gamma_0(q)} \omega_q^{-1} \Gamma \backslash \Hb.\]
This partitions $\Gamma_0(q) \backslash \Hb$ into $q + 1$ translates of $\Gamma \backslash \Hb$. We are interested in the \emph{hybrid} problem concerning the asymptotic behaviour of the volume of $\FF_A(q) \cap \Gamma_0(q) \omega_q^{-1} \Gamma \backslash \Hb$ on average over a genus $G_D$ as \emph{both} $q$ and $D$ tend to infinity; this is motivated by the work of Liu, Masri, and Young on the analogous problem for Heegner points \cite[Theorem 1.4]{LMY13}. Since the volume of $\omega_q^{-1} \Gamma \backslash \Hb$ is equal to that of $\Gamma \backslash \Hb$, these translates are small in comparison to the total volume of $\Gamma_0(q) \backslash \Hb$, so that as $q$ grows, we are studying the equidistribution of hyperbolic orbifolds in sets of shrinking volume.

\begin{theorem}
\label{thm:level}
Fix $\delta \in [0,\frac{1}{36})$. For each positive squarefree fundamental discriminant $D$, choose a genus $G_D$ in the group of genera $\Gen_D$. For each odd prime $q$ that splits in $E$ with $q \leq D^{\delta}$, choose $\omega_q \in \Gamma / \Gamma_0(q)$. Then
\[\frac{\vol(\Gamma_0(q) \backslash \Hb)}{\vol(\Gamma \backslash \Hb)} \frac{\sum_{A \in G_D} \vol(\FF_A(q) \cap \Gamma_0(q) \omega_q^{-1} \Gamma \backslash \Hb)}{\sum_{A \in G_D} \vol(\FF_A(q))} = 1 + o_{\delta}(1)\]
as $qD$ tends to infinity. Assuming the generalised Lindel\"{o}f hypothesis, the same result holds for $q \leq D^{\delta}$ for some fixed $\delta \in [0,\frac{1}{12})$.
\end{theorem}

The second author \cite{Nor23} recently studied the related problem of determining the distribution of the homology classes of closed geodesics
\[[\CC_A(q)]\in H_1(Y_0(q),\Z),\]
where $Y_0(q) \coloneqq \Gamma_0(q)\backslash \Hb$ is the noncompact modular surface of level $q$ and $\CC_A(q)$ is the oriented closed geodesic associated to $A\in \Cl_D^+$. As discussed in \cite{LMY13}, these results can be viewed as an analogue of Linnik's theorem. Linnik's theorem, in its effective form, states that there exists an absolute constant $L > 1$ such that given a positive integer $q$ and a reduced residue class $a$ modulo $q$, there exists the expected number of primes, $(1 + o(1)) \frac{q^L}{\varphi(q)}$, that are less than $q^L$ and are congruent to $a$ modulo $q$. \hyperref[thm:level]{Theorem \ref*{thm:level}} can be viewed in a similar light: there exists an absolute constant $\delta > 0$ such that given a prime $q$ and a translate $\omega_q^{-1} \Gamma \backslash \Hb$ of $\Gamma \backslash \Hb$ in $\Gamma_0(q) \backslash \Hb$, there exists the expected proportion of mass of projections of hyperbolic orbifolds of discriminant $D \leq q^{\frac{1}{\delta}}$ onto $\Gamma_0(q) \backslash \Hb$ that lie in the translate $\omega_q^{-1} \Gamma \backslash \Hb$.

\subsection{Sparse Equidistribution in the Subgroup Aspect}

Our second refinement is to study the equidistribution of hyperbolic orbifolds averaged over \emph{sparse} subsets of $\Cl_D^+$. Previously, we averaged over a genus $G_D$, which has cardinality $2^{1 - \omega(D)} h_D^+$, where $h_D^+ \coloneqq |\Cl_D^+|$ denotes the narrow class number and $\omega(D)$ denotes the number of distinct prime divisors of $D$. We instead consider an \emph{arbitrary} subgroup $H = H_D$ of the narrow class group $\Cl_D^+$ in place of the subgroup $(\Cl_D^+)^2$ and an \emph{arbitrary} coset $CH$ in place of a genus $G_D = C (\Cl_D^+)^2$. Whereas a genus $G_D$ satisfies $|G_D| \gg_{\e} D^{-\e} h_D^+$ for every $\e > 0$, we allow for the possibility that the cardinality of a coset $CH$ may be significantly smaller than $h_D^+$.

\begin{theorem}
\label{thm:subgroup}
Fix $\delta \in [0,\frac{1}{2826})$ and fix either $q = 1$ or $q$ an odd prime. For each positive fundamental discriminant $D$ for which $q$ splits in $E$ if $q > 1$, choose a coset $CH$ with $C \in \Cl_D^+$ and $H = H_D$ a subgroup of $\Cl_D^+$ satisfying $|H| \gg D^{-\delta} h_D^+$. Then for each fixed continuity set $B \subset \Gamma_0(q) \backslash \Hb$,
\begin{equation}
\label{eqn:subgroupequidistribution}
\frac{\sum_{A \in CH} \vol(\FF_A(q) \cap \Gamma_0(q) B)}{\sum_{A \in CH} \vol(\FF_A(q))} = \frac{\vol(B)}{\vol(\Gamma_0(q) \backslash \Hb)} + o_{q,B,\delta}(1)
\end{equation}
as $D$ tends to infinity. Assuming the generalised Lindel\"{o}f hypothesis, the same result holds for $|H| \gg D^{-\delta} h_D^+$ for some fixed $\delta \in [0,\frac{1}{4})$.
\end{theorem}

Note that by taking $H$ to be the trivial subgroup, \hyperref[thm:subgroup]{Theorem \ref*{thm:subgroup}} implies the equidistribution of \emph{individual} hyperbolic orbifolds as $D$ tends to infinity along fundamental discriminants for which $h_D^+ \ll D^{\delta}$ for some fixed $\delta < \frac{1}{2826} \approx 0.00035$ (cf.\ \cite[Theorem 6.5.1]{Pop06}).

At the other extreme, we may take $H$ to be $\Cl_D^+$, so that \eqref{eqn:subgroupequidistribution} gives the equidistribution of hyperbolic orbifolds averaged over the whole narrow class group. For $q = 1$, such a result is \emph{trivial}, as observed by Duke, Imamo\={g}lu, and T\'{o}th \cite[Section 4]{DIT16}; as we discuss in \hyperref[sect:sparseproof]{Section \ref*{sect:sparseproof}}, however, this result is no longer trivial for $q$ an odd prime.

\hyperref[thm:subgroup]{Theorem \ref*{thm:subgroup}} is motivated by a conjecture of Michel and Venkatesh \cite[Conjecture 1]{MV06}, where the analogous statement for Heegner points is conjectured to hold for any fixed $\delta \in [0,\frac{1}{2})$; it is noted that the generalised Lindel\"{o}f hypothesis implies such a conjecture in the range $\delta \in [0,\frac{1}{4})$. Harcos and Michel have proven this conjecture in the range $\delta \in [0,\frac{1}{2826})$ \cite[Theorem 1.2]{HM06} (see additionally \cite[Corollary 1.4]{Har11}), while Venkatesh has studied this problem in more general settings \cite[Theorem 7.2]{Ven10}.

More recently, a toy model of this problem was resolved by the first author \cite[Theorem 1.5]{Hum22}, namely the sparse equidistribution as $q$ tends to infinity of the points
\[\left\{\left(\frac{d}{q},\frac{d'}{q}\right) \in \T^2 : d \in CH, \ dd' \equiv 1 \hspace{-.25cm} \pmod{q}\right\}\]
in the torus $\T^2 = (\R/\Z)^2$ indexed by a coset $CH$ of the group $(\Z/q\Z)^{\times}$ with $q$ a prime and $|H| \gg q^{\delta}$ for some fixed $\delta > 0$. In this setting, it is shown that this sparse equidistribution result is a simple consequence of a deep result of Bourgain on cancellation in certain exponential sums \cite{Bou05}.

\subsection{Small Scale Equidistribution and Discrepancy Bounds}

The first author investigated a refinement of Duke, Imamo\={g}lu, and T\'{o}th's equidistribution result in \cite{Hum18}, namely small scale equidistribution, in which the continuity set $B \subset \Gamma \backslash \Hb$ in \eqref{eqn:equidistribution} is chosen to be a ball $B_R(w)$ whose radius $R$ shrinks as $D$ grows. One can think of the fastest rate at which this radius can shrink with respect to the growth of $D$ for which equidistribution still holds as being the smallest \emph{scale} of equidistribution. We prove the following unconditional result in this regard.

\begin{theorem}
\label{thm:smallscale}
Fix $\delta \in [0,\frac{1}{2})$ and $w \in \Gamma \backslash \Hb$. For each positive squarefree fundamental discriminant $D$, choose a genus $G_D$ in the group of genera $\Gen_D$. Then for all $R \in [D^{-\delta},1]$,
we have that
\begin{equation}
\label{eqn:smallscale}
\frac{\vol(\Gamma \backslash \Hb)}{\vol(B_R(w))} \frac{\sum_{A \in G_D} \vol(\FF_A \cap \Gamma B_R(w))}{\sum_{A \in G_D} \vol(\FF_A)} = 1 + o_{\delta}(1)
\end{equation}
as $D$ tends to infinity.
\end{theorem}

This improves upon \cite[Theorem 1.24]{Hum18}, where this result was proven under the assumption of the generalised Lindel\"{o}f hypothesis\footnote{As stated, \cite[Theorem 1.24]{Hum18} claims that the asymptotic formula \eqref{eqn:smallscale} holds for $R \asymp D^{-\delta}$ for some fixed $\delta \in [0,\frac{1}{12})$ unconditionally and for some fixed $\delta \in [0,\frac{1}{4})$ under the assumption of the generalised Lindel\"{o}f hypothesis, though these conditions should in fact be the stronger conditions $\delta \in [0,\frac{1}{6})$ and $\delta \in [0,\frac{1}{2})$ respectively.}. We use \hyperref[thm:smallscale]{Theorem \ref*{thm:smallscale}} to bound the discrepancy associated to this equidistribution result, which may be thought of as a quantitative way of measuring \emph{uniformly} the \emph{rate} of equidistribution.

\begin{theorem}
\label{thm:discrepancy}
For each positive squarefree fundamental discriminant $D$, choose a genus $G_D$ in the group of genera $\Gen_D$. Then as $D$ tends to infinity, we have that
\[\sup_{B_R(w) \subset \Gamma \backslash \Hb} \left|\frac{\sum_{A \in G_D} \vol(\FF_A \cap \Gamma B_R(w))}{\sum_{A \in G_D} \vol(\FF_A)} - \frac{\vol(B_R(w))}{\vol(\Gamma \backslash \Hb)}\right| \ll_{\e} D^{-\frac{1}{12} + \e},\]
where the supremum is over all injective geodesic balls in $\Gamma \backslash \Hb$. Assuming the generalised Lindel\"{o}f hypothesis, the stronger bound $O_{\e}(D^{-1/4 + \e})$ holds for the discrepancy.
\end{theorem}

\begin{remark}
Recently, the second named author and Ser Peow Tan \cite{NT25} gave a geometric construction of the \emph{partial coverings} studied in this paper and in \cite{DIT16} (i.e.\ the image of the projection $\FF_A(q) \to \Gamma_0(q)\backslash \Hb$) applicable to any Fuchsian group of the first kind. This construction circumvents any reference to thin subgroups. Furthermore, \cite[Theorem 1.1]{NT25} shows that equidistribution of a sequence of collections of closed geodesics implies equidistribution of the associated partial coverings, where in both settings equidistribution is meant in the sense of smooth compactly supported functions. The methods are quite  different to the present paper and in particular do not yield effective error terms (outside of the co-compact case, which is not treated in the present paper) and do not cover the case where the Fuchsian group is varying, as in \hyperref[thm:level]{Theorem \ref*{thm:level}}.
\end{remark}

\subsection{A Discussion on the Proofs}

\subsubsection{Sparse Equidistribution in the Level Aspect}

\hyperref[thm:level]{Theorem \ref*{thm:level}} is proven via approximating the indicator function of $\omega_q^{-1} \Gamma \backslash \Hb$ by a smooth function and spectrally expanding this on $\Gamma_0(q) \backslash \Hb$. One then integrates this spectral expansion over the projection of $\FF_A(q) $ onto $\Gamma_0(q) \backslash \Hb$ and sums over $A \in G_D$. The main contribution comes from the constant function. It remains to bound the contributions from the cuspidal and continuous spectrum, which involve Weyl sums. As is standard for equidistribution results of this form, these Weyl sums can be expressed in terms of the twisted $L$-functions $L\left(\frac{1}{2},f \otimes \chi_{D_1}\right)$, where $D_1 \mid D$ is a fundamental discriminant and $\chi_{D_1}$ denotes the primitive quadratic Dirichlet character modulo $|D_1|$. Unusually, there is additionally a new \emph{topological contribution} involving 
\begin{enumerate}[leftmargin=*,label=\textup{(\arabic*)}]
\item the twisted $L$-functions $L\left(\frac{1}{2},h \otimes \chi_{D_1}\right)$ for $h$ a holomorphic Hecke cusp form of weight $2$ and level $q$;
\item line integrals of $f$ along sides of a fundamental polygon $\PP(q)$ of $\Gamma_0(q)$ (which are independent of $D$); and
\item certain cap product pairings between homology and cohomology of the compactification $X_0(q)$ of the modular surface $\Gamma_0(q) \backslash \Hb$ (which are also independent of $D$).
\end{enumerate}

A key step is to obtain strong bounds for certain moments of $L$-functions of the form 
\[\sum_{D_1 D_2 = D} \sum_{\substack{f \in \BB_0(\Gamma_0(q)) \\ T \leq t_f \leq 2T}} \frac{L\left(\frac{1}{2},f \otimes \chi_{D_1}\right) L\left(\frac{1}{2},f \otimes \chi_{D_2}\right)}{L(1,\ad f)},\]
where $\BB_0(\Gamma_0(q))$ denotes an orthonormal basis of Hecke--Maa\ss{} cusp forms in $L^2(\Gamma_0(q) \backslash \Hb)$. Due the hybrid nature of \hyperref[thm:level]{Theorem \ref*{thm:level}}, we require bounds for this moment that are uniform in $q$ and $D$, and similarly hybrid bounds for related moments of $L$-functions associated to \emph{holomorphic} Hecke cusp forms. Here we proceed via H\"{o}lder's inequality together with hybrid bounds for the third moments of $L(\frac{1}{2},f \otimes \chi_{D_1})$ and $L(\frac{1}{2},f \otimes \chi_{D_2})$ due to Petrow and Young \cite{PY19}.

What remains is to bound the contribution from the topological term. By a detailed analysis of the structure of $\partial\FF_A(q)$, the line integrals of $f$ can be expressed in terms of Vorono\u{\i} $L$-series (i.e.\ $L$-series of additive twists of $f$), which can in turn be bounded on average via the spectral large sieve. The bounding of the cap product pairing can be seen as a homological instance of the sup-norm problem and has been bounded recently by the second author \cite{Nor23} using a theta correspondence approach combined with techniques from geometric coding to treat the counting problem that arises.

\subsubsection{Sparse Equidistribution in the Subgroup Aspect}

We reduce the proof of \hyperref[thm:subgroup]{Theorem \ref*{thm:subgroup}} to showing that for each fixed Hecke--Maa\ss{} cusp form $f \in \BB_0(\Gamma_0(q))$, the quantity
\[\frac{|H|}{h_D^+} \frac{1}{\sum_{A \in CH} \vol(\FF_A(q))} \sum_{\chi \in H^{\perp}} \overline{\chi}(C) W_{\chi,f}\]
tends to zero as $D$ tends to infinity, as well as an analogous result for $f$ replaced by an Eisenstein series. Here $H^{\perp}$ denotes the annihilator of $H$, namely the set of characters $\chi$ of $\Cl_D^+$ that satisfy $\chi(A) = 1$ for all $A \in H$, while the Weyl sum $W_{\chi,f}$ is
\begin{equation}
\label{eqn:Weylchif}
W_{\chi,f} \coloneqq \sum_{A \in \Cl_D^+} \chi(A) \int_{\FF_A(q)} f(z) \, d\mu(z).
\end{equation}

Following the earlier work of Duke, Imamo\={g}lu, and T\'{o}th \cite[Proposition 1]{DIT16}, we prove a lower bound for the sums of volumes of the hyperbolic orbifolds of the form
\[\sum_{A \in CH} \vol(\FF_A(q)) \gg_{q,\e} \frac{|H|}{h_D^+} D^{\frac{1}{2} - \e}.\]
Since $|H^{\perp}| = \frac{h_D^+}{|H|}$, equidistribution follows provided that we can show that there exists an absolute constant $\alpha > 0$ such that
\[W_{\chi,f} \ll_{q,f} \frac{|H|}{h_D^+} D^{\frac{1}{2} - \alpha}.\]
In this way, the proof of \hyperref[thm:subgroup]{Theorem \ref*{thm:subgroup}} is reduced to proving such bounds for Weyl sums.

To bound these Weyl sums, we make explicit a version of Waldspurger's formula \cite{Wal85} in order to show that the square of the absolute value of a Weyl sum is essentially equal to a special value of a Rankin--Selberg $L$-function. The proof is then completed upon invoking subconvex bounds for Rankin--Selberg $L$-functions due to Harcos and Michel \cite{HM06}.

\subsubsection{Small Scale Equidistribution and Discrepancy Bounds}

The proof of \hyperref[thm:smallscale]{Theorem \ref*{thm:smallscale}} concerning small scale equidistribution follows in a similar manner to the previous work of Young \cite[Theorem 2.1]{You17} on small scale equidistribution of Heegner points on $\Gamma \backslash \Hb$ and of the first author and Radziwi\l{}\l{} \cite[Theorem 1.5]{HR22} on small scale equidistribution of lattice points on the sphere. The chief idea is to approximate the indicator function of a ball via a smooth function and spectrally expand this on $\Gamma \backslash \Hb$. One then integrates this spectral expansion over the projection of $\Gamma_A \backslash \NN_A$ onto $\Gamma \backslash \Hb$ and sums over $A \in G_D$. The main contribution comes from the constant function, and our goal becomes adequately bounding the ensuing sum over the cuspidal spectrum and integral over the continuous spectrum.

We approach this via a dyadic subdivision together with an application of pre-existing identities of Duke, Imamo\={g}lu, and T\'{o}th relating Weyl sums to $L$-functions \cite[Theorems 3 and 4]{DIT16}. In this way, the problem is reduced to deducing strong bounds for the quantities
\begin{gather*}
\sum_{D_1 D_2 = D} \sum_{\substack{f \in \BB_0(\Gamma) \\ T \leq t_f \leq 2T}} \frac{L\left(\frac{1}{2},f \otimes \chi_{D_1}\right) L\left(\frac{1}{2},f \otimes \chi_{D_2}\right)}{L(1,\ad f)},	\\
\sum_{D_1 D_2 = D} \int\limits_{T \leq |t| \leq 2T} \left|\frac{L\left(\frac{1}{2} + it,\chi_{D_1}\right)^2 L\left(\frac{1}{2} + it,\chi_{D_2}\right)^2}{\zeta(1 + 2it)}\right|^2 \, dt,
\end{gather*}
that are uniform in both $T$ and $D$. For this, we can appeal to the recent work of the first author and Radziwi\l{}\l{} \cite[Proposition 2.14]{HR22}. The bounds for the discrepancy in \hyperref[thm:discrepancy]{Theorem \ref*{thm:discrepancy}} are then readily deduced from \hyperref[thm:smallscale]{Theorem \ref*{thm:smallscale}}.

\subsection{Complications}

The above discussion of the proofs of \hyperref[thm:level]{Theorems \ref*{thm:level}}, \ref{thm:subgroup}, \ref{thm:smallscale}, and \ref{thm:discrepancy} glosses over two key obstacles that we must overcome.

\subsubsection{Hyperbolic Orbifolds of Level $q$}

The first obstacle is defining the hyperbolic orbifolds of level $q$. For $q = 1$, Duke, Imamo\={g}lu, and T\'{o}th \cite{DIT16} associate to a narrow ideal class $A$ of a real quadratic field a Fuchsian group of the second kind $\Gamma_A \leq \PSL_2(\Z)$ whose unique boundary component projects to the closed geodesic $\CC_A(1) \subset X_0(1)$ associated to $A$. In the higher level setting, however, it is not possible to define a subgroup of $\Gamma_0(q)$ that is a Fuchsian group of the second kind whose boundary component projects to $\CC_A(q) \subset X_0(q)$. This is caused by the possibility of a topological obstruction in the case of nontrivial genus. One can observe this intuitively by considering a loop around one of the holes of a torus, as in \hyperref[fig:torus]{Figure \ref*{fig:torus}}. This loop is not the boundary of a submanifold of the torus, which is precisely due to the fact that the loop is nontrivial in the homology of the torus. In other words, we must compensate for the nontrivial homology of $\CC_A(q)$ in a systematic way. 

\begin{figure}
\begin{tikzpicture}[scale=2] 
 \begin{scope}\draw[yscale=cos(70), double distance=2*5mm] (0:1) arc (0:180:1);
    \draw[yscale=cos(70),  double distance=2*5mm] (180:1) arc (180:360:1);
     \fill[white](-12.5 mm,-0.07mm) rectangle  (-7 mm ,0.09mm);
     \fill[white](12.5 mm,-0.07mm) rectangle  (7 mm ,0.09mm);
    \draw [ xscale=cos(70), yshift=-3.4 mm, densely dashed]  circle(0.255);

    \end{scope}

\end{tikzpicture}
\caption{A loop around one of the holes of the torus} \label{fig:torus}
\end{figure}
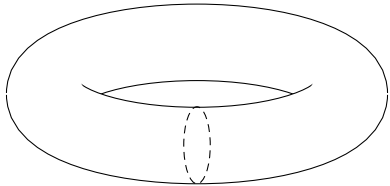

In \hyperref[sect:hypq]{Section \ref*{sect:hypq}}, we define a Fuchsian group of the second kind $\Gamma_A(q) \leq \PSL_2(\Z)$ (which in general is not contained in $\Gamma_0(q)$) as well as an associated canonical fundamental domain $\FF_A(q) \subset \Hb$ of $\Gamma_A(q) \backslash \NN_A(q)$, where $\NN_A(q)$ is the associated Nielsen region. The equidistribution problem that we study concerns the image under the map $\FF_A(q) \to X_0(q)$ given by restricting the projection $\Hb \twoheadrightarrow X_0(q)$. The starting point for this construction is a special fundamental polygon $\PP(q)$ of $\Gamma_0(q)$ introduced by Kulkarni \cite{Kul91} and studied recently by Doan, Kim, Lang, and Tan \cite{DKLT25}. Pictorially, the boundary of the fundamental polygon $\partial\PP(q)$ defines a curve inside $X_0(q)$, and we compensate for the nontrivial homology of $\CC_A(q)$ by subsegments of $\partial\PP(q)$. The boundary $\partial\PP(q)$ contributes to the topological terms in the Weyl sums alluded to above, which has not previously been observed before in the context of Duke's theorem. Notably, bounding these topological contributions requires several new ideas that go beyond the work of Duke, Imamo\={g}lu, and T\'{o}th \cite{DIT16}, namely averaged bounds for Vorono\u{\i} $L$-series and a homological version of the sup-norm problem.

\subsubsection{Weyl Sums}

The second obstacle is relating the Weyl sum \eqref{eqn:Weylchif} to a special value of a Rankin--Selberg $L$-function. Let us first recall how this is done by Duke, Imamo\={g}lu, and T\'{o}th when $q = 1$ and $H = (\Cl_D^+)^2$. In this setting, a character $\chi \in H^{\perp}$ is a genus character associated to a pair of primitive quadratic Dirichlet characters $\chi_{D_1},\chi_{D_2}$ of conductors $|D_1|,|D_2|$, where $D_1,D_2 \in \Z$ are fundamental discriminants for which $D_1 D_2 = D$. In \cite[Theorem 4]{DIT16}, it is shown that the Weyl sum $W_{\chi,f}$ for this genus character $\chi$ is essentially equal to $b(D_1) b(D_2)$, where $b(n)$ denotes the $n$-th Fourier coefficient of the Shintani lift of the Hecke--Maa\ss{} form $f$ of weight zero, namely the Maa\ss{} form of weight $\frac{1}{2}$ associated to $f$ via the Shimura--Shintani correspondence, as elucidated by Katok and Sarnak \cite{KS93}. One then uses Waldspurger's formula \cite[Th\'{e}or\`{e}me 1]{Wal81}, in an explicit form due to Baruch and Mao \cite[Theorem 1.4]{BM10}, to relate $|b(D_1)|^2$ and $|b(D_2)|^2$ to the values at $s = \frac{1}{2}$ of the $L$-functions $L(s,f \otimes \chi_{D_1})$ and $L(s, f \otimes \chi_{D_2})$.

This method breaks down, however, when $\chi$ is not a genus character, for then there is no explicit relation between the Weyl sum and the Fourier coefficients of the Shintani lift. We show that nonetheless there is a relation between the square of the absolute value of the Weyl sum and the special value of an $L$-function. The key observation is that \cite[Lemma 1]{DIT16} identifies the Weyl sum with a weighted sum of cycle integrals, namely integrals of an automorphic form along a closed geodesic. With a careful classical-to-ad\`{e}lic correspondence, we show that this weighted sum of cycle integrals can in turn be viewed as an ad\`{e}lic period integral $\Pscr_{\Omega}(\phi)$, where the integrand is the product of a distinguished automorphic form $\phi : \GL_2(\A_{\Q}) \to \C$ lying in the automorphic representation $\pi_f$ associated to $f$ and the Hecke character $\Omega : E^{\times} \backslash \A_E^1 \to \C^{\times}$ associated to $\chi$.

It is known, going back to the work of Waldspurger \cite{Wal85}, that $|\Pscr_{\Omega}(\phi)|^2$ is essentially equal to the value at $s = \frac{1}{2}$ of the Rankin--Selberg $L$-function $L(s, \pi_f \otimes \pi_{\Omega})$, where $\pi_{\Omega}$ denotes the automorphic induction of $\Omega$ to an automorphic representation of $\GL_2(\A_{\Q})$. When $\Omega$ is the id\`{e}lic lift of a narrow class character $\chi$, this $L$-function may in turn be written as the Rankin--Selberg $L$-function $L(s,f \otimes \Theta_{\chi})$, where $\Theta_{\chi}$ denotes the theta series on $\Gamma_0(D) \backslash \Hb$ associated to $\chi$.

It is instructive to consider the case when the Hecke character $\Omega : E^{\times} \backslash \A_E^1 \to \C^{\times}$ factors through the norm map from $\A_E^{\times}$ to $\A_{\Q}^{\times}$. When this occurs, the automorphic representation $\pi_{\Omega}$ is noncuspidal; it is the isobaric sum of a pair of Hecke characters $\omega_1,\omega_2 : \Q^{\times} \backslash \A_{\Q}^1 \to \C^{\times}$ and the vector space of automorphic forms associated to $\pi_{\Omega}$ consists of Eisenstein series, so that the Rankin--Selberg $L$-function $L(s, \pi_f \otimes \pi_{\Omega})$ factorises as $L(s, \pi_f \otimes \omega_1) L(s, \pi_f \otimes \omega_2)$. When $\Omega$ is additionally the ad\`{e}lic lift of a narrow class character $\chi$, then necessarily $\chi$ is a genus character associated to a pair of quadratic characters $\chi_{D_1},\chi_{D_2}$ whose id\`{e}lic lifts are $\omega_1,\omega_2$, so that $L(s,\pi_f \otimes \omega_1) = L(s,f \otimes \chi_{D_1})$ and $L(s,\pi_f \otimes \omega_2) = L(s,f \otimes \chi_{D_2})$. In this way, our result encompasses that of Duke, Imamo\={g}lu, and T\'{o}th.

There is a caveat to our claim that $|\Pscr_{\Omega}(\phi)|^2$ is essentially equal to $L(\frac{1}{2}, \pi_f \otimes \pi_{\Omega})$; this is true only up to multiplication by a constant that depends sensitively on the automorphic form $\phi$. Thankfully, this constant has been explicitly evaluated for certain choices of automorphic forms $\phi$ by Martin and Whitehouse \cite{MW09}, and while the cases of interest required in the proofs of our two main theorems do not quite fall under the umbrella of the cases considered in \cite{MW09}, only mild modifications are needed.

\section{Oriented Closed Geodesics of Level \texorpdfstring{$q$}{q}}
\label{sect:levelq}

We begin by recording several details relating oriented ideals and narrow ideal classes of real quadratic fields, integral binary quadratic forms, embeddings of real quadratic fields into spaces of matrices, and closed geodesics on the level $q$ modular surface $\Gamma_0(q) \backslash \Hb$. Useful references for this material include \cite[Section 1]{GKZ87}, \cite[Section 1]{Dar94}, and \cite[Section 6]{Pop06}. We work throughout with a positive fundamental discriminant $D > 1$ and a squarefree integer $q$ for which every prime dividing $q$ splits in $E \coloneqq \Q(\sqrt{D})$. We also fix once and for all a residue class $r$ modulo $2q$ for which $r^2 \equiv D \pmod{4q}$; the dependence on the choice of $r$ is briefly discussed in \hyperref[sect:AtkinLehner]{Section \ref*{sect:AtkinLehner}}.

\subsection{Oriented Ideals}

Let $\aa$ be a nonzero fractional ideal of $E$. Let $(\alpha_1,\alpha_2) \in E^2$ be generators over $\Z$ of $\aa$, so that $\aa = \Z \alpha_1 + \Z \alpha_2$. The (absolute) norm of $\aa$ is
\[\Norm(\aa) = \frac{|\alpha_1 \sigma(\alpha_2) - \alpha_2 \sigma(\alpha_1)|}{\sqrt{D}},\]
where $\sigma$ denotes the nontrivial Galois automorphism of $E$. The ideal $\aa$ is said to be oriented with respect to the ordered pair of generators $(\alpha_1,\alpha_2) \in E^2$ if $\alpha_1 \sigma(\alpha_2) - \alpha_2 \sigma(\alpha_1) > 0$ and to be of level $q$ if
\[\frac{\Norm(\alpha_1)}{\Norm(\aa)} \equiv 0 \hspace{-.2cm} \pmod{q} \quad \text{and} \quad \frac{\Tr(\alpha_1 \sigma(\alpha_2))}{\Norm(\aa)} \equiv r \hspace{-.2cm} \pmod{2q}.\]
We denote by $[\aa;\alpha_1,\alpha_2]$ the oriented ideal $\aa$ with respect to the generators $(\alpha_1,\alpha_2)$.

The congruence subgroup $\Gamma_0(q) \ni \gamma$ acts on the set of such triples $[\aa;\alpha_1,\alpha_2]$ by acting trivially on the ideal $\aa$ and mapping the ordered pair of generators $(\alpha_1,\alpha_2) \in E^2$, viewed as a row vector, to $(\alpha_1,\alpha_2) \gamma$. This action preserves oriented ideals of level $q$.

Let $\OO_E$ denote the ring of integers of $E$. The set
\[\PP_E^+ \coloneqq \{(\alpha) = \alpha \OO_E \subset E : \alpha \in E, \ \alpha, \sigma(\alpha) > 0\}\]
of totally positive principal fractional ideals --- equivalently, the identity $I$ in the narrow class group $\Cl_D^+$ --- acts on an oriented ideal $[\aa; \alpha_1,\alpha_2]$ of level $q$ via the map $(\alpha) \cdot [\aa;\alpha_1,\alpha_2] \coloneqq [(\alpha)\aa; \alpha \alpha_1,\alpha \alpha_2]$, and this action commutes with the action of $\Gamma_0(q)$.

In this way, each narrow ideal class of $\Cl_D^+$ may be bijectively identified with equivalence classes of oriented ideals of level $q$ modulo the action of $\Gamma_0(q)$ and $\PP_E^+$, where we identify the identity narrow ideal class $I$ with the equivalence class of oriented ideals containing $[\Z q + \Z \frac{r - \sqrt{D}}{2}; q, \frac{r - \sqrt{D}}{2}]$ for each $r \in \Z$ for which $r^2 \equiv D \pmod{4q}$. The narrow ideal class $J \in \Cl_D^+$ containing the different $\df$ then corresponds to the equivalence class of oriented ideals containing $[\Z q\sqrt{D} + \Z \frac{D - r\sqrt{D}}{2}; q\sqrt{D}, \frac{D - r\sqrt{D}}{2}]$.

\subsection{Binary Quadratic Forms}
\label{sect:heegnerform}

For each trio of integers $(a,b,c) \in \Z^3$ having greatest common divisor equal to $1$ and satisfying $b^2 - 4ac = D$, $a \equiv 0 \pmod{q}$, and $b \equiv r \pmod{2q}$, we define the integral binary quadratic form
\begin{equation}
\label{eqn:Qxy}
Q(x,y) = \begin{pmatrix} x & y \end{pmatrix} \begin{pmatrix} a & \frac{b}{2} \\ \frac{b}{2} & c \end{pmatrix} \begin{pmatrix} x \\ y \end{pmatrix} = ax^2 + bxy + cy^2,
\end{equation}
which is a primitive form of level $q$ and discriminant $b^2 - 4ac = D$. We write $Q = [a,b,c]$ to denote this form and call such a form a Heegner form, following \cite{Dar94}; we denote the set of such forms by $\QQ_D(q)$.

The congruence subgroup $\Gamma_0(q) \ni \gamma$ acts on Heegner forms via
\[(\gamma \cdot Q)(x,y) \coloneqq Q((x,y) \prescript{t}{}{\gamma}) = \begin{pmatrix} x & y \end{pmatrix} \prescript{t}{}{\gamma} \begin{pmatrix} a & \frac{b}{2} \\ \frac{b}{2} & c \end{pmatrix} \gamma \begin{pmatrix} x \\ y \end{pmatrix},\]
where $\prescript{t}{}{\gamma}$ denotes the transpose of $\gamma$ and we view $(x,y)$ as a row vector. Moreover, this action preserves $\QQ_D(q)$.

To each equivalence class $\PP_E^+ \cdot [\aa;\alpha_1,\alpha_2]$ of oriented ideals of level $q$, we associate the Heegner form $Q = Q_{\PP_E^+ \cdot [\aa;\alpha_1,\alpha_2]} \in \QQ_D(q)$ given by
\[Q(x,y) \coloneqq \frac{\Norm(\alpha_1 x + \alpha_2 y)}{\Norm(\aa)}.\]
Conversely, associated to each Heegner form $Q = [a,b,c] \in \QQ_D(q)$ as in \eqref{eqn:Qxy} is the equivalence class of oriented ideals of level $q$ given by
\begin{equation}
\label{eqn:Qtoideals}
\begin{dcases*}
\PP_E^+ \cdot \left[\Z a + \Z \frac{b - \sqrt{D}}{2}; a,\frac{b - \sqrt{D}}{2}\right] & if $a > 0$,	\\
\PP_E^+ \cdot \left[\Z (-a\sqrt{D}) + \Z \frac{D - b\sqrt{D}}{2}; -a\sqrt{D},\frac{D - b\sqrt{D}}{2},\right] & if $a < 0$.
\end{dcases*}
\end{equation}
This map is a bijection between $\QQ_D(q)$ and equivalence classes of oriented ideals of level $q$. This association descends to a bijection between narrow ideal classes of the narrow class group $\Cl_D^+$ and equivalence classes of primitive integral binary quadratic forms of level $q$ and discriminant $D$ modulo the action of $\Gamma_0(q)$. In particular, for each $r \in \Z$ for which $r^2 \equiv D \pmod{4q}$, the equivalence class of elements of $\QQ_D(q)$ modulo $\Gamma_0(q)$ that contains $[q,r,\frac{r^2 - D}{4q}]$ corresponds to the principal narrow ideal class $I \in \Cl_D^+$, while the equivalence class containing $[-q,r,\frac{D - r^2}{4q}]$ corresponds to the narrow ideal class $J$ containing the different $\df$.

\subsection{Oriented Embeddings}
\label{sect:embeddings}

Again let $(a,b,c) \in \Z^3$ have greatest common divisor equal to $1$ and satisfy $b^2 - 4ac = D$, $a \equiv 0 \pmod{q}$, and $b \equiv r \pmod{2q}$. We define an embedding $\Psi : E \hookrightarrow \Mat_{2 \times 2}(\Q)$ by
\begin{equation}
\label{eqn:Psi}
\Psi(x + \sqrt{D} y) \coloneqq \begin{pmatrix} x + by & 2cy \\ -2ay & x - by \end{pmatrix}
\end{equation}
for $x,y \in \Q$. This satisfies
\[\Psi(E) \cap \left\{g \in \Mat_{2 \times 2}(\Z) : g_{2,1} \equiv 0 \hspace{-.25cm} \pmod{q}\right\} = \Psi(\OO_E);\]
that is, $\Psi$ is an oriented optimal embedding of level $q$. Conversely, every oriented optimal embedding of level $q$ arises from such a trio of integers $(a,b,c) \in \Z^3$.

The congruence subgroup $\Gamma_0(q)$ acts on the set of optimal oriented embeddings of level $q$ by conjugation, namely
\[(\gamma \cdot \Psi)(x + \sqrt{D} y) \coloneqq \gamma^{-1} \Psi(x + \sqrt{D} y) \gamma\]
for $\gamma \in \Gamma_0(q)$, and this action preserves optimal oriented embeddings of level $q$.

There is a natural bijection between oriented optimal embeddings $\Psi$ of level $q$ as in \eqref{eqn:Psi} and Heegner forms $Q = [a,b,c]$ as in \eqref{eqn:Qxy}, since these are both completed determined by $(a,b,c) \in \Z^3$; in turn, there is a bijection with equivalence classes of oriented ideals of level $q$ as in \eqref{eqn:Qtoideals}. Again, this descends to a bijection between narrow ideal classes $A$ of the narrow class group $\Cl_D^+$ and equivalence classes of oriented optimal embeddings of level $q$ modulo the action of $\Gamma_0(q)$. In particular, the equivalence classes of oriented ideals of level $q$ corresponding to $I$ and $J$ respectively contain the optimal embeddings of level $q$ given respectively by
\[\Psi_I(x + \sqrt{D} y) \coloneqq \begin{pmatrix} x + ry & \frac{r^2 - D}{2q} y \\ -2qy & x - ry \end{pmatrix}, \qquad \Psi_J(x + \sqrt{D} y) \coloneqq \begin{pmatrix} x + ry & \frac{D - r^2}{2q} y \\ 2qy & x - ry \end{pmatrix}.\]

\subsection{Closed Geodesics}
\label{sect:closedgeodesic}

Associated to a Heegner form $Q = [a,b,c] \in \QQ_D(q)$ as in \eqref{eqn:Qxy} is an oriented geodesic $S_Q$ in the upper half-plane connecting the two points $\frac{-b - \sqrt{D}}{2a}$ and $\frac{-b + \sqrt{D}}{2a}$, namely the Euclidean semicircle
\begin{equation}
\label{eqn:semicircle}
S_Q \coloneqq \left\{z \in \Hb : a|z|^2 + b\Re(z) + c = 0\right\}
\end{equation}
oriented anticlockwise if $a > 0$ and clockwise if $a < 0$.

Let $\epsilon_D > 1$ be the least unit with positive norm in $\OO_E$, so that $\epsilon_D = u + \sqrt{D} v$ with $u,v$ positive half-integers that satisfy Pell's equation $u^2 - D v^2 = 1$ and minimise $v$ among all such positive half-integral solutions. For the oriented optimal embedding $\Psi$ of level $q$ as in \eqref{eqn:Psi} associated to $Q$, define the matrix
\begin{equation}
\label{eqn:gammaQ}
\gamma_Q \coloneqq \Psi(\epsilon_D) = \begin{pmatrix} u + bv & 2cv \\ -2av & u - bv \end{pmatrix} \in \Gamma_0(q).
\end{equation}
Together with $\begin{psmallmatrix} -1 & 0 \\ 0 & -1 \end{psmallmatrix}$, this generates the group of automorphs of $Q$,
\[\Gamma_0(q)_Q \coloneqq \left\{\gamma \in \Gamma_0(q) : \gamma \cdot Q = Q\right\}.\]
Note that $\gamma_Q$ is a hyperbolic matrix, as $\Tr(\gamma_Q) = 2u = 2\sqrt{1 + Dv^2} > 2$.

Let $A$ be the narrow ideal class associated to the equivalence class of Heegner forms modulo $\Gamma_0(q)$ that contains $Q$. We let $\CC_A(q) \coloneqq \Gamma_0(q)_Q \backslash S_Q$ denote the oriented closed geodesic in $\Gamma_0(q) \backslash \Hb$ corresponding to $A$, which we may view explicitly as the reduction modulo $\Gamma_0(q)$ of the oriented geodesic segment from $z_Q$ to $\gamma_Q z_Q$, where
\begin{equation}
\label{eqn:closedgeodesic}
z_Q \coloneqq \begin{dcases*}
\frac{-b + i\sqrt{D}}{2a} & if $a > 0$,	\\
\frac{b + i\sqrt{D}}{-2a} & if $a < 0$,
\end{dcases*} \qquad \gamma_Q z_Q = \begin{dcases*}
\frac{-b(u^2 + D v^2) - 2Duv + i\sqrt{D}}{2a(u^2 + D v^2)} & if $a > 0$,	\\
\frac{b(u^2 + D v^2) + 2Duv + i\sqrt{D}}{-2a(u^2 + D v^2)} & if $a < 0$.
\end{dcases*}
\end{equation}

\subsection{Atkin--Lehner Operators}
\label{sect:AtkinLehner}

We return to our choice of $r$ modulo $2q$ for which $r^2 \equiv D \pmod{4q}$. There are $2^{\omega(q)}$ such choices, and for each choice, there is an associated collection of $h_D^+$ distinct oriented closed geodesics in $\Gamma_0(q) \backslash \Hb$. Each of these collections of oriented closed geodesics is permuted by the $2^{\omega(q)}$ Atkin--Lehner operators on $\Gamma_0(q)$. More precisely, for each divisor $q_1$ of $q$, so that $q = q_1 q_2$, we let
\begin{equation}
\label{eqn:ALop}
W_{q_1} \coloneqq \begin{pmatrix} a \sqrt{q_1} & \frac{b}{\sqrt{q_1}} \\ c q_2 \sqrt{q_1} & d\sqrt{q_1} \end{pmatrix} \in \SL_2(\R)
\end{equation}
be an Atkin--Lehner operator on $\Gamma_0(q)$ associated to $q_1$, where $a,b,c,d \in \Z$ are such that $adq_1 - bcq_2 = 1$. Any two such Atkin--Lehner operators associated to $q_1$ are equivalent modulo $\Gamma_0(q)$. These operators are elements of the normaliser of $\Gamma_0(q)$ and permute the cusps of $\Gamma_0(q) \backslash \Hb$. They act on oriented ideals, Heegner forms, oriented embeddings, and oriented closed geodesics; an Atkin--Lehner operator $W_{q_1}$, where $q_1 q_2 = q$, has the effect of replacing $r \in \Z/2q\Z$ with the element $r' \in \Z/2q\Z$ that satisfies $r' \equiv -r \pmod{2q_1}$ and $r' \equiv r \pmod{2q_2}$.

\section{Hyperbolic Orbifolds of Level \texorpdfstring{$q$}{q}}
\label{sect:hypq}

Let $D > 1$ be a positive fundamental discriminant and let $q$ be an odd prime\footnote{Much of what we prove below generalises in a straightforward manner to squarefree level $q$ (at least when there is no $3$-torsion), but we stick to the case of $q$ prime for simplicity.} that splits in $E = \Q(\sqrt{D})$; as in \hyperref[sect:closedgeodesic]{Section \ref*{sect:closedgeodesic}}, we fix once and for all a residue class $r$ modulo $2q$ for which $r^2 \equiv D \pmod{4q}$. Our goal for this section is to construct a hyperbolic orbifold whose boundary is the closed geodesic on $X_0(q)$ associated to an element of the narrow class group $\Cl_D^+$. Due to the possibility of nontrivial genus, these orbifolds have some new and interesting features that do not show up in the level $1$ case.

\subsection{Special Fundamental Domains}

By the Kurosh subgroup theorem, any subgroup of
\[\PSL_2(\Z) \cong \Z/2\Z\ast \Z/3\Z\]
is isomorphic to a free product of a number of copies of $\Z/2\Z$, $\Z/3\Z$, and $\Z$. In particular, if $\Gamma_0(q)$ is a torsion-free Hecke congruence subgroup, then it is a free group on $\mathrm{rank}_{\Z}\, \Gamma_0(q)^\mathrm{ab}$ generators. We now describe an explicit geometric way, following \cite{Kul91}, for constructing an independent generating set using so-called \emph{special fundamental polygons} of $\Gamma_0(q)$.

\subsubsection{Farey Symbols of Level \texorpdfstring{$q$}{q}}

Here and below, we denote by $g$ the genus of $X_0(q)$ (suppressing $q$ from the notation), and by $e_2\in \{0,2\}$ and $e_3\in \{0,2\}$ the number of conjugacy classes of subgroups in $\Gamma_0(q)$ of order $2$ and of order $3$ respectively (see \cite[(2.12) and (2.13)]{Iwa02}). We write
\[T \coloneqq \begin{pmatrix}1 & 1 \\ 0 & 1 \end{pmatrix}, \qquad S \coloneqq \begin{pmatrix}0& -1 \\ 1 & 0 \end{pmatrix}\]
for the standard generators of $\PSL_2(\Z)$.

\begin{definition}[\cite{Kul91}]
A \emph{Farey symbol of level $q$} is a sequence of reduced fractions,
\[\frac{0}{1} = \frac{a_{0}}{b_{0}}<\frac{a_1}{b_1}<\ldots<\frac{a_{n-1}}{b_{n-1}}< \frac{a_{n}}{b_{n}}=\frac{1}{1},\]
where $n \coloneqq 4g + e_2 + e_3$, such that
\begin{itemize}
\item $a_{i+1}b_{i}-a_{i}b_{i+1}=1$ for all $i \in \{0,\ldots,n\}$,
\item there are $e_2$ \emph{even indices} $i$ such that
\[b_i^2+b_{i+1}^2 \equiv 0 \hspace{-.2cm} \pmod{q},\]
\item there are $e_3$ \emph{odd indices} $i$ such that 
\[b_i^2+b_ib_{i+1}+b_{i+1}^2 \equiv 0 \hspace{-.2cm} \pmod{q},\]
\item for the remaining $4g$ \emph{free indices}, there is a pairing $i\leftrightarrow i^\ast$ satisfying 
\[b_ib_{i^\ast}+b_{i+1}b_{i^\ast+1}\equiv 0 \hspace{-.2cm} \pmod{q}.\]
\end{itemize}
Here we consider the indices $i$ above modulo $n + 2$ by setting $a_{-1} = -1$, $b_{-1} = b_{n + 1} = 0$, and $a_{n + 1} = 1$.
\end{definition}

Such a Farey symbol of level $q$ always exists; moreover, one can even find one that is symmetric around $\frac{1}{2}$, so that $a_{n - i} = 1 - a_i$ and $b_{n - i} = b_i$ for each index $i$ \cite[Section 13]{Kul91}. Dooms, Jespers, and Konovalov \cite{DJK10} have described an algorithm for determining Farey symbols of arbitrary level.

It is easily seen that \emph{any} Farey symbol of level $q$ satisfies the bound $b_i\ll e^{O(q)}$ for the denominators; on the other hand, we have the lower bound $\max_i |b_i| \gg \sqrt{q}$. It has recently been shown by Doan, Kim, Lang, and Tan that furthermore there exist Farey symbols of prime level that are minimal in the sense that this lower bound is essentially sharp.

\begin{theorem}[{\cite[Theorem 1.1]{DKLT25}}]
\label{thm:boundentries}
Let $q$ be an odd prime. There exists a Farey symbol of level $q$ that is symmetric about $\frac{1}{2}$ with $b_i \leq \lfloor \sqrt{\frac{4q}{3}}\rfloor$ for each index $i \in \{0,\ldots,n\}$, and such that
\begin{align*}
b_i^2 + b_{i + 1}^2 & = q \quad \text{for all even indices $i$,}	\\
b_i^2 + b_i b_{i + 1} + b_{i + 1}^2 & = q \quad \text{for all odd indices $i$, and}	\\
b_i b_{i^\ast} + b_{i + 1} b_{i^\ast + 1} & = q \quad \text{for all free indices $i$.}
\end{align*}
\end{theorem}

From here onward, we let $q$ be an odd prime and we fix a Farey symbol of level $q$ as in \hyperref[thm:boundentries]{Theorem \ref*{thm:boundentries}}. Following \cite[Section 2]{Kul91}, we denote by $\PP(q)$ the hyperbolic polygon in $\Hb$ with the following $ e_2+e_3+n+2$ vertices:
\begin{itemize}
\item $\infty$;
\item the $n + 1$ fractions $\frac{a_i}{b_i}$ of the Farey symbol for each index $i \in \{0,\ldots,n\}$;
\item the $e_2$ $\PGL_2(\Z)$-translates of $\sqrt{-1}$ on the infinite geodesic  connecting $\frac{a_i}{b_i}$ and $\frac{a_{i + 1}}{b_{i + 1}}$ for each even index $i \in \{1,\ldots,n - 1\}$; and
\item the $e_3$ $\PGL_2(\Z)$-translates of $\frac{1 + i\sqrt{3}}{2}$ lying just below the infinite geodesic connecting $\frac{a_i}{b_i}$ and $\frac{a_{i + 1}}{b_{i + 1}}$ for each odd index $i \in \{1,\ldots,n - 1\}$.
\end{itemize}
As in \cite[Section 3]{Kul91}, we define a side pairing on $\PP(q)$ by pairing the following edges:
\begin{itemize}
\item the two vertical geodesics sharing the cusp $\infty$;
\item the two geodesic segments connecting $\frac{a_i}{b_i}$ and $\frac{a_{i + 1}}{b_{i + 1}}$ with the $\PGL_2(\Z)$-translate of $\sqrt{-1}$ lying between them for each of the $e_2$ even indices $i$;
\item the two geodesic segments connecting $\frac{a_i}{b_i}$ and $\frac{a_{i + 1}}{b_{i + 1}}$ with the $\PGL_2(\Z)$-translate of $\frac{1+\sqrt{-3}}{2}$ lying between them for each of the $e_3$ odd indices $i$; and
\item the infinite geodesic between the cusps $\frac{a_{i}}{b_{i}}$ and $\frac{a_{i+1}}{b_{i+1}}$ with the infinite geodesic between the cusps $\frac{a_{i^\ast+1}}{b_{i^\ast+1}}$ and $\frac{a_{i^\ast}}{b_{i^\ast}}$ for each of the $4g$ free indices $i$.
\end{itemize}
\begin{definition}
We define a \emph{side} of $\PP(q)$ to be a maximal geodesic segment on the boundary of $\PP(q)$ (note that there is only one side for each even index). Given a side $\LL$ of $\PP(q)$ that is paired with a corresponding side $\LL'$ (which might be $\LL$ itself), we define the \emph{label of $\LL$} as the element $\alpha$ of $\PSL_2(\Z)$ that maps $\LL'$ to $\LL$ as defined by Kulkarni (see below). The set of all labels is called the set of \emph{side pairing elements} of $\PP(q)$.
\end{definition}

By Poincar\'{e}'s polygon theorem \cite{Mas71}, the side pairing elements define a finite index subgroup of $\PSL_2(\Z)$ that has $\PP(q)$ as a fundamental domain; by \cite[Theorem 13.2]{Kul91}, this subgroup is the congruence subgroup $\Gamma_0(q)$. Furthermore, since $\PP(q)$ has a minimal number of sides (which follows from \cite[Proposition 2.6]{Iwa02}), the labels define an independent set of generators (upon forgetting inverses) of $\Gamma_0(q)$.

We denote by
\[\{iy: y>0\} \eqqcolon \LL_1,\quad \LL_2, \quad \ldots \quad \LL_{N - 1}, \quad \LL_{N} \coloneqq \{ 1+iy: y>0\},\]
the $N \coloneqq 4g+2+e_2+2e_3$ oriented sides of $\PP(q)$ taken in the positive orientation starting from the side containing $0$ and $\infty$. We refer to the vertex of $\LL_j$ shared with $\LL_{j - 1}$ as the \emph{left-most} or \emph{first} vertex of $\LL_j$ (which is $\infty$ for $j = 1$) and the vertex of $\LL_j$ shared with $\LL_{j + 1}$ as the \emph{right-most} or \emph{second} vertex (which is $\infty$ for $j = N$).

We denote by $\alpha_j$ the label in $\Gamma_0(q)$ associated to the oriented side $\LL_j$ of $\PP(q)$, so that $\alpha_1 = T^{-1}$ and $\alpha_N = T$. From \cite[Theorem 6.1]{Kul91}, these labels are explicitly as follows:
\begin{itemize}
\item for a side $\LL_j$ associated to an even index $i$,
\[\alpha_j = \begin{pmatrix} a_{i+1}b_{i+1}+a_{i}b_i & -a_i^2-a_{i+1}^2 \\ b_i^2+b_{i+1}^2 & -a_{i+1}b_{i+1}-a_{i}b_{i}\end{pmatrix};\]
\item for a side $\LL_j$ associated to an odd index $i$,
\[\alpha_j = \begin{pmatrix} a_{i+1}b_{i+1}+a_{i}b_{i+1}+a_{i}b_i & -a_i^2-a_ia_{i+1}-a_{i+1}^2 \\ b_i^2+b_ib_{i+1}+b_{i+1}^2 & -a_{i+1}b_{i+1}-a_{i+1}b_i-a_{i}b_{i}\end{pmatrix}^{\pm 1},\]
where $\pm = +$ if $\frac{a_i}{b_i}$ is to the left of the associated $\PGL_2(\Z)$-translate of $\frac{1 + \sqrt{-3}}{2}$, while $\pm = -$ if $\frac{a_i}{b_i}$ is to the right;
\item for a side $\LL_j$ associated to a free index $i$,
\[\alpha_j = \begin{pmatrix} a_{i^\ast+1}b_{i+1}+a_{i^\ast}b_i & -a_ia_{i^\ast}-a_{i+1}a_{i^\ast+1} \\ b_ib_{i^\ast}+b_{i+1}b_{i^\ast+1} & -a_{i+1}b_{i^\ast+1}-a_{i}b_{i^\ast}\end{pmatrix}.\]
\end{itemize}
We say that a side $\LL_j$ is \emph{hyperbolic, elliptic}, or \emph{parabolic} if the associated label $\alpha_j$ is a hyperbolic, elliptic, or parabolic matrix, respectively; in particular, sides associated to even or odd indices are elliptic. We note that since $0 \leq a_i \leq b_i$ and $b_i\ll \sqrt{q}$, the Frobenius norm $\sqrt{\Tr(\alpha_j \prescript{t}{}{\alpha_j})}$ of any label $\alpha_j$ of $\PP(q)$ is bounded by $O(q)$.

Here and below, we consider the indices $j$ of $\LL_j$ and $\alpha_j$ modulo $N$. By a slight abuse of notation, we denote by $\ast:\Z/N\Z\to \Z/N\Z$ the side pairing, so that for $j \in \Z/ N\Z$,
\[\alpha_{j}\LL_{j^\ast} = \LL_j, \qquad \alpha_{j^\ast} \LL_j = \LL_{j^{\ast}}.\]

\subsubsection{Orbits of Consecutive Polygons}

It is key for us to understand the orbits of the map $t:\Z/N\Z\to \Z/N\Z$ given by
\[t(j) \coloneqq j^\ast - 1.\]
To see why this map is relevant, let $\PP_0, \PP_1, \PP_2$ be three consecutive $\Gamma_0(q)$-translates of $\PP(q)$ that all contain some cusp $\bb \in \Pb^1(\Q)$ (ordered in positive orientation). Let $\alpha_{j_1}$ be the label of the side of $\PP_0$ shared with $\PP_1$ and let $\alpha_{j_2}$ be the label of the side of $\PP_1$ shared with $\PP_2$. Then we have the equality
\[j_2 = j_1^\ast - 1 = t(j_1).\]
In other words, the map $t$ encodes labels of consecutive translates of $\PP(q)$.

\begin{lemma}
\label{lem:ast}
Let $\EE\subset \Z/N\Z$ denote the subset of indices $j$ for which $\alpha_j$ is either elliptic of order $3$ or parabolic and additionally $\LL_j$ is to the left of $\LL_{j^{\ast}}$. The map $t:\Z/N\Z\to \Z/N\Z$ has $e_3+2$ orbits explicitly given by
\[\{j\} \text{ for each $j \in \EE$ and }(\Z/N\Z)\backslash \EE.\]
Furthermore, for any $t$-orbit $\mathcal{S}\subset \Z/N\Z$ and for all $j\in \mathcal{S}$, the element of $\Gamma_0(q)$ given by
\[\alpha_{j}\alpha_{t(j)}\cdots \alpha_{t^{(|\mathcal{S}|-1)}(j)}\]
is either elliptic or parabolic.
\end{lemma}

\begin{proof}
First of all, for $j\in \EE$ we have that $t(j) =j+1-1= j$. Next, consider the orbit $\mathcal{S}$ of $1 \in \Z/N\Z$. Observe that for $j\in S$, the second vertex of the side $\LL_j$ of $\PP(q)$ is cuspidal and not equal to $\infty$. Pick an element $j \in \mathcal{S}$ and let $\bb \in \Pb^1(\Q)$ be the second vertex of the side $\LL_{j}$. Then the element
\[\alpha \coloneqq \alpha_{j} \alpha_{t(j)}\alpha_{t(t(j))}\cdots \alpha_{t^{(|\mathcal{S}|-1)}(j)}\in \Gamma_0(q)\]
lies in the stabiliser of $\bb$ as it maps the side $\LL_{j}$ of $\PP(q)$ containing $\bb$ to the side $\alpha\LL_{j}$ of $\alpha \PP(q)$ containing $\bb$ (using here that $t^{(|\mathcal{S}|)}(j)=j$). This implies that the width of the cusp between $\LL_{j}$ and $\alpha\LL_{j}$ is at least $q$ (as this is the width of the cusp $0$ in $\Gamma_0(q)$). On the other hand, all of the cuspidal zones near $\bb$ of
\[\alpha_j\PP(q),\quad  \alpha_j\alpha_{t(j)}\PP(q),\quad \ldots\quad \alpha_{j} \alpha_{t(j)}\cdots \alpha_{t^{(|\mathcal{S}|-2)}(j)}\PP(q),\quad \alpha \PP(q), \]
are $\Gamma_0(q)$-inequivalent. Thus we conclude that $\alpha$ is in fact a generator of the stabiliser of $\bb$ and furthermore that $\mathcal{S}$ must contain all indices not in $\EE$, since these are precisely the indices whose associated side has second vertex equal to a cusp in the $\Gamma_0(q)$-orbit of $0$.
\end{proof}

We note that the size of the orbit containing $1\in \Z/N\Z$ is equal to $4g+e_2+e_3+1$. In particular, it follows from the above that
\[\alpha_{1} \alpha_{t(1)}\cdots \alpha_{t^{(4g+e_2+e_3)}(1)} = \begin{pmatrix}1 & 0\\-q & 1 \end{pmatrix},\]
as the left-hand side is equal to the generator of the stabiliser in $\Gamma_0(q)$ of $0$ and maps $\infty$ to a negative rational. Here we recall that $\alpha_{t^{(4g+e_2+e_3+1)}(1)}=\alpha_1=T^{-1}$.

\begin{lemma}
\label{lem:boundfrobnorm}
For each $j \in \{0,\ldots,4g+e_2+e_3 \}$, the Frobenius norm of the matrix
\[\alpha_{1} \alpha_{t(1)}\cdots \alpha_{t^{(j)}(1)}\in\Gamma_0(q),\]
is bounded by $O(q^{\frac{3}{2}})$.
\end{lemma}

\begin{proof}
Put
\begin{equation}\label{eqn:gammaj2}\gamma_j \coloneqq \alpha_{1} \alpha_{t(1)}\cdots \alpha_{t^{(j)}(1)}.\end{equation}
Then the fundamental domain $\gamma_j\PP(q)$ contains the cusp $0$. This means that 
\begin{equation}
\label{eqn:gammaj}
\gamma_j \frac{a_i}{b_i} = 0
\end{equation}
for some vertex $\frac{a_i}{b_i}$ of $\PP(q)$. Since all the cuspidal zones near $0$ of 
\[\alpha_{1}\PP(q),\qquad \alpha_{1}\alpha_{t(1)}\PP(q),\qquad \ldots\quad \gamma_j\PP(q),\]
are $\Gamma_0(q)$-inequivalent, we see that $\gamma_j\in \Gamma_0(q)$ is characterised as the element of $\Gamma_0(q)$ satisfying \eqref{eqn:gammaj} for which $\gamma_j\infty$ is negative such that $|\gamma_j\infty|$ is maximised. This means explicitly that
\[\gamma_j=\begin{pmatrix} b_i & -a_i \\ -qx_0 & y_0 \end{pmatrix},\]
where $x_0 \in \N$ is the minimal positive solution to the congruence $a_i q x_0\equiv -1 \pmod{b_i}$ and $y_0 \coloneqq \frac{1 + a_i q x_0}{b_i}$. This gives the desired result since $b_i\ll \sqrt{q}$.
\end{proof}

\subsection{Orbifolds Associated to Real Quadratic Fields}

Let $\gamma = \begin{psmallmatrix} a & b \\ c & d \end{psmallmatrix} \in \Gamma_0(q)$ be a primitive hyperbolic matrix with axis
\begin{equation}
\label{eqn:semicircle2}
S_{\gamma} \coloneqq \left\{z \in \Hb : c|z|^2 + (d - a)\Re(z) - b = 0\right\}
\end{equation}
oriented anticlockwise if $c < 0$ and clockwise if $c > 0$ and intersecting the fundamental polygon $\PP(q)$ (which can always be arranged by conjugation). We denote by $\CC_{\gamma}(q)$ the corresponding oriented closed geodesic in $\Gamma_0(q) \backslash \Hb$, namely $\CC_{\gamma}(q) \coloneqq \Gamma_0(q)_{\gamma} \backslash S_{\gamma}$ with $\Gamma_0(q)_{\gamma} \coloneqq \{\gamma' \in \Gamma_0(q) : \gamma^{\prime -1} \gamma \gamma' = \gamma\}$ the associated group of automorphs.

Let $\alpha_{i(0)}$ be the label of the side containing the second intersection between $S_\gamma$ and $\PP(q)$, $\alpha_{i(1)}$ the label of the side containing the second intersection between $S_\gamma$ and $\alpha_{i(0)}\PP(q)$, and so on. As the projection of $S_{\gamma}$ onto $\PP(q)$ is a \emph{closed} geodesic, this defines a \emph{periodic} sequence of side pairing elements (see e.g.\ \cite{Kat96})
\[\alpha_{i(0)}, \alpha_{i(1)}, \alpha_{i(2)},\ldots.\]
Let $m_{\gamma}+1$ denote the period of this sequence. We then have the following  representation of $\gamma$ in terms of the labels of $\PP(q)$, namely the \emph{Morse code} of $\gamma$:
\begin{equation}
\label{eq:morsecode}
\gamma=\alpha_{i(0)} \alpha_{i(1)}\cdots \alpha_{i(m_{\gamma})}.
\end{equation}  
The representation of $\gamma$ in terms of the side pairing elements of $\PP(q)$ is unique up to the relation $\alpha_j^2=\alpha_{j^\ast}$ for the elliptic labels $\alpha_j$ of order $3$. We consider the index as a map $i:\Z/(m_{\gamma}+1)\Z\to \Z/N \Z$. For $i_1,i_2 \in \Z/N\Z$, we denote by $\arc(i_1,i_2)$ the indices that lie on the positively oriented open arc strictly between $i_1$ and $i_2$ when $\Z/N\Z$ is embedded into the unit circle $S^1$ via the homomorphism defined by
\[1 \hspace{-.2cm} \pmod{N} \mapsto e^{\frac{2\pi i}{N}}.\]
Thus if we view $i_1,i_2 \in \Z/N\Z$ as positive integers in $\{1,\ldots,N\}$, then
\[\arc(i_1,i_2) = \begin{dcases*}
\{i_1 + 1,\ldots,i_2 - 1\} & if $i_1 < i_2$,	\\
\emptyset & if $i_1 = i_2$,	\\
\{1,\ldots,i_2 - 1,i_1 + 1,\ldots,N\} & if $i_1 > i_2$.
\end{dcases*}\]

Given an integer $k \in \Z/N\Z$, we define  $\sigma_k \in \PSL_2(\Z)$ by
\[\sigma_k \coloneqq \begin{dcases*}
\alpha_k & if $\alpha_k$ is elliptic,	\\
\begin{pmatrix} a_{i+1} & a_{i}\\ b_{i+1}& b_i \end{pmatrix} S \begin{pmatrix} a_{i+1} & a_{i}\\ b_{i+1}& b_i \end{pmatrix}^{-1} & otherwise,
\end{dcases*}\]
where the index $i$ is such that the side $\LL_k$ associated to $\alpha_k$ is contained in the strip between $\frac{a_i}{b_i}$ and $\frac{a_{i + 1}}{b_{i + 1}}$. Note that if $i$ is a free index, then
\[\sigma_k = \begin{pmatrix} a_{i+1}b_{i+1}+a_{i}b_i & -a_i^2-a_{i+1}^2 \\ b_i^2+b_{i+1}^2 & -a_{i+1}b_{i+1}-a_{i}b_{i}\end{pmatrix},\]
which is not an element of $\Gamma_0(q)$ since by definition $i$ is not an even index. Note also that if $i$ is a free or even index, then $\alpha_k$ is a $\pi$-rotation around $\LL_k$.

We now construct a subgroup of $\PSL_2(\Z)$ from the Morse code of $\gamma$. Special care has to be given to labels of order $3$ since the geometry is different at the corresponding sides. We define
\begin{align*}
\delta_+(j) & \coloneqq \begin{dcases*}
1 & if $\alpha_{i(j)}^3 = 1$ and $i(j)^\ast\in\arc(i(j), i(j-1)^\ast)$, \\
0 & otherwise,
\end{dcases*}	\\
\delta_-(j) & \coloneqq \begin{dcases*}
1 & if $\alpha_{i(j-1)}^3 = 1$ and $i(j-1) \in\arc(i(j), i(j-1)^\ast)$, \\
0 & otherwise.
\end{dcases*}
\end{align*}
Let $\mathcal{I}_{3}\subset\{0,\ldots, m_\gamma\}$ denote the set of $j$ such that $\alpha_{i(j-1)}^3=1$, $i(j-1)\in\arc(i(j), i(j-1)^\ast)$ and $i(j-1)\neq i(j-2)$. For $j\notin \mathcal{I}_3$, we define
\begin{equation}
\label{eqn:Omega}
\Omega_{\gamma,j}(q) \coloneqq \left\{ \alpha_{i(0)}\cdots \alpha_{i(j-1)} \sigma_{k} \alpha_{i(j-1)}^{-1}\cdots \alpha_{i(0)}^{-1} :   k\in \arc(i(j)+\delta_+(j), i(j-1)^\ast-\delta_-(j))\right\},
\end{equation}
and for $j\in \mathcal{I}_3$ (in which case $\delta_-(j)=1$), we define
\begin{multline}
\label{eqn:Omega1}
\Omega_{\gamma,j}(q) \coloneqq \left\{ \alpha_{i(0)}\cdots \alpha_{i(j-1)} \sigma_{k} \alpha_{i(j-1)}^{-1}\cdots \alpha_{i(0)}^{-1} :   k\in \arc(i(j)+\delta_+(j), i(j-1)^\ast-\delta_-(j))\right\}\\
\cup \left\{ \alpha_{i(0)}\cdots \alpha_{i(j-2)}\alpha_{i(j-1)^\ast} \sigma_{k} \alpha_{i(j-1)^\ast}^{-1}\alpha_{i(j-2)}^{-1}\cdots \alpha_{i(0)}^{-1} :   k\in \arc(i(j-1)^\ast, i(j-1))\right\}.
\end{multline}
As we shall see in the next section, the additional set of matrices in the case $j\in \mathcal{I}_3$ corresponds geometrically to ``adding an extra copy of $\PP(q)$'' necessary to ensuring that the fundamental domain is convex.  

We put
$$\Omega_\gamma(q)\coloneqq \bigcup_{j=0}^{m_\gamma} \Omega_{\gamma,j}(q).$$
and define the following subgroup of $\PSL_2(\Z)$:
\[\Gamma_\gamma(q) \coloneqq \left\langle \Omega_\gamma(q), \gamma\right\rangle.\]
We shall see in the next section that (outside of one exceptional case) $\Gamma_\gamma(q)$ is a \emph{thin subgroup} of $\PSL_2(\Z)$, meaning Zariski dense and of infinite index. Note in particular that $\Omega_{\gamma}(q)$ is nonempty since $\gamma$ is hyperbolic. We denote by $\NN_\gamma(q)$ the \emph{Nielsen region} (or \emph{convex core}) for $\Gamma_\gamma(q)$, namely the smallest nonempty $\Gamma_{\gamma}(q)$-invariant open convex subset of $\Hb$. Notice that we have a well-defined map $\Gamma_\gamma(q)\backslash \NN_\gamma(q) \to X_0(1)$ arising from the inclusion $\Gamma_\gamma(q)\subset \PSL_2(\Z)$. We do not, however, have a canonical map from $\Gamma_\gamma(q) \backslash \NN_\gamma(q)$ to $X_0(q)$. In order to define such a map, we construct a canonical fundamental domain for $\Gamma_\gamma(q) \backslash \NN_\gamma(q)$ inside $\Hb$ and then consider the projection of this to $X_0(q)$.

\subsubsection{A Fundamental Domain for $\Gamma_\gamma(q)\backslash \NN_{\gamma}(q)$} 

We begin by constructing an explicit fundamental domain for $\Gamma_{\gamma}(q)\backslash \Hb$. This allows us to define a map $\Gamma_{\gamma}(q)\backslash \NN_{\gamma}(q)\to X_0(q)$ and furthermore to study the geometry of $\Gamma_{\gamma}(q)\backslash \NN_{\gamma}(q)$. There are geometrical obstacles if $\alpha_{i(m_\gamma)}^3=1$ and $i(m_\gamma)\in \arc(i(0),i(m_\gamma)^\ast)$. By conjugation, we can always arrange so that this is \emph{not} the case unless there is $3$-torsion and $\gamma\in\{\alpha_{k_1}\alpha_{k_2},\alpha_{k_2}\alpha_{k_1}\}$ with $1< k_1<k_2< N$ and $(\LL_{k_1},\LL_{k_1+1}),(\LL_{k_2},\LL_{k_2+1})$ the two pairs of sides of $\PP(q)$ with order $3$ labels (this follows by a quick case study recalling that $e_3\in \{0,2\}$). If the discriminant $D>0$ of $\alpha_{k_1}\alpha_{k_2}$ is fundamental, we can exclude it from our considerations by changing the sign of the orientation $r$ modulo $2q$ as in \hyperref[sect:levelq]{Section \ref*{sect:levelq}} (i.e.\  $r^2 \equiv D \pmod{4q}$). Thus we will henceforth assume that $\gamma\in \Gamma_0(q)$ is primitive hyperbolic and \emph{nonexceptional}, meaning that either $\alpha_{i(m_\gamma)}^3\neq 1$ or $i(m_\gamma)\notin \arc(i(0),i(m_\gamma)^\ast)$.

With this in mind, let $\tilde{\FF}_{\gamma}(q)$ be the hyperbolic polygon bounded by the following geodesic segments:
\begin{itemize}
\item the complete geodesics containing
\[\LL_{i(m_{\gamma})^\ast}, \quad \gamma\LL_{i(m_{\gamma})^\ast},\]
\item for each $j \in \{0,\ldots,m_{\gamma}\}$ and each $k\in \arc(i(j), i(j-1)^\ast)$, the geodesic segment
\[\alpha_{i(0)}\cdots \alpha_{i(j-1)}\LL_k,\]
\item for each $j \in \mathcal{I}_3$, as defined above equation \eqref{eqn:Omega}, and each $k\in \arc(i(j-1)^\ast, i(j-1))$, the geodesic segment
\[\alpha_{i(0)}\cdots \alpha_{i(j-2)}\alpha_{i(j-1)^\ast}\LL_k.\]
\end{itemize}
Note that all of the geodesic segments are disjoint by the assumption that $\gamma$ is nonexceptional. Moreover, $\tilde{\FF}_{\gamma}(q)$ is convex by construction. If $\alpha_k$ is hyperbolic, parabolic, or elliptic of order $2$, the geodesic segment $\alpha_{i(0)}\cdots \alpha_{i(j-1)}\LL_k$ is mapped to itself by the corresponding element of $\Omega_{\gamma,j}(q)$, whereas if $\alpha_k$ is of order $3$, the corresponding element of $\Omega_{\gamma,j}(q)$ maps this side to
\[\alpha_{i(0)}\cdots \alpha_{i(j-1)}\LL_{k\pm 1},\]
where $\pm$ depends on whether $\LL_k$ is the left-most or right-most side in the elliptic pair of sides. 

From the above, we define the following hyperbolic polygon
\begin{equation}
\label{eqn:surface}
\FF_{\gamma}(q) \coloneqq \tilde{\FF}_{\gamma}(q) \cap \NN_{\gamma}(q)\subset \Hb,
\end{equation}
which we shall soon see is a fundamental domain for $\Gamma_\gamma(q) \backslash \NN_\gamma(q)$, and in fact that $\FF_{\gamma}(q) = \tilde{\FF}_{\gamma}(q) \setminus \mathrm{int}(S_\gamma)$, where $\mathrm{int}(S_\gamma)\subset \Hb$ denotes the interior (relative to the orientation) of the axis $S_\gamma$ of $\gamma$ (see \hyperref[fig:F_A1]{Figures \ref*{fig:F_A1}} and \ref{fig:F_A2}).

\begin{example}
It is instructive to consider the case of $q=11$. The genus of $X_0(11)$ is $g = 1$, while there are no conjugacy classes in $\Gamma_0(11)$ of order $2$ or $3$, so that $e_2 = e_3 = 0$, and consequently $n = 4g + e_2 + e_3 = 4$ and $N = n + e_3  + 2 = 6$. \hyperref[fig:1]{Figure \ref*{fig:1}} shows a special fundamental domain $\PP(11)$ given by the hyperbolic polygon with vertices $\infty, 0, \tfrac{1}{3},\tfrac{1}{2},\tfrac{2}{3},1$ with corresponding labels
\begin{align*}
\alpha_1=T^{-1} & = \begin{pmatrix} 1 & -1 \\ 0 & 1 \end{pmatrix}, & \alpha_2 &= \begin{pmatrix} -3 & 2 \\ -11 & 7 \end{pmatrix}, & \alpha_3 & = \begin{pmatrix}4 & -3 \\ 11 & -8 \end{pmatrix},	\\
\alpha_4 = \alpha_2^{-1} & = \begin{pmatrix} 7 & -2 \\ 11 & -3 \end{pmatrix}, & \alpha_5 = \alpha_3^{-1} & = \begin{pmatrix} -8 & 3 \\ -11 & 4 \end{pmatrix}, & \alpha_6 = T & = \begin{pmatrix} 1 & 1 \\ 0 & 1 \end{pmatrix}.
\end{align*}

\hyperref[fig:F_A1]{Figure \ref*{fig:F_A1}} gives an example of a surface $F_{\gamma}(11)$ with $\gamma$ of reduced word length $1$ in the free generators $\{T, \alpha_2,\alpha_3\}$ (cf.\ \cite[Figures 2 and 4]{DIT16}). In this case, $\gamma = \alpha_2^{-1} = \begin{psmallmatrix} 7 & -2 \\ 11 & -3 \end{psmallmatrix}$ is equal to the matrix $\gamma_Q$ as in \eqref{eqn:gammaQ} associated to the Heegner form $Q = [-11,10,-2] \in \QQ_{12}(11)$. With $r = 10$, $Q$ corresponds to the narrow ideal class $J \in \Cl_{12}^+$.

\hyperref[fig:F_A2]{Figure \ref*{fig:F_A2}} gives an example of a surface $\FF_\gamma(11)$ with $\gamma$ of reduced word length $3$ in the free generators $\{T, \alpha_2,\alpha_3\}$. In this case, $\gamma = \alpha_3^{-1} \alpha_2^{-1} \alpha_3^{-1} = \begin{psmallmatrix} 107 & -41 \\ 154 & -59 \end{psmallmatrix}$. Note that the discriminant $D = 2300 = 5^2 \cdot 92$ of this hyperbolic matrix is not a fundamental discriminant. The boundary consists of $9$ semicircles (two of which are very small) as well as two incomplete geodesic segments connected by the closed geodesic associated to $\gamma$.

We are interested in the projections of surfaces $\FF_\gamma(11)$ to $X_0(11)$, where $\gamma$ corresponds to an ideal class in a real quadratic number field $E$ such that $11$ splits in $E$. Consider $E=\Q(\sqrt{92})$, which has narrow class number $2$ and wide class number $1$. \hyperref[fig:1]{Figure \ref*{fig:1}} shows the projection of $\FF_{\gamma}(11)$ to $\PP(11)$, where $\gamma = \alpha_3 \alpha_6 \alpha_2 \alpha_3^{-1} \alpha_6 = \begin{psmallmatrix} 19 & 10 \\ 55 & 29 \end{psmallmatrix}$ is equal to the matrix $\gamma_Q$ as in \eqref{eqn:gammaQ} associated to the Heegner form $Q = [-11,-2,2] \in \QQ_{92}(11)$. With $r = -2$, this corresponds to the narrow ideal class $J \in \Cl_{92}^+$ (cf.\ \cite[Figure 6]{DIT16}).
\end{example}

\begin{figure}[ht]
\centering
\begin{tikzpicture}[scale=14]
    \begin{scope}
      \clip (-0.6,-0.1) rectangle (0.6,0.4);
       \fill[black, opacity= 0.15]  (-0.6,-0.1) rectangle (0.6,0.4);

         \draw[ultra thin] (-0.083333,0) circle(0.08333333);
      \draw[ultra thin] (0.0833333,0) circle(0.083333333);
        \draw[ultra thin] (0.33333333,0) circle(0.16666666);
         \draw[ultra thin]  (-0.3333333,0) circle(0.166666666);
    
\draw[thick] (0.29709*0.5+0.61200*0.5-0.5,0)circle(-0.29709*0.5+0.61200*0.5);      
           \fill[white](0.29709*0.5+0.61200*0.5-0.5,0)circle(-0.29709*0.5+0.61200*0.5);

          \fill[white](-0.083333,0) circle(0.08333333);
      \fill[white](0.0833333,0) circle(0.083333333);
         \fill[white] (0.33333333,0) circle(0.16666666);
         \fill[white] (-0.3333333,0) circle(0.166666666);

                 \draw (-0.5,0) -- (-0.5,0.6);
        \draw  (0.5,0) -- (0.5,0.6);

     \fill[white](-0.5,0) rectangle  (-0.6,0.6);
        \fill[white] (0.5,0) rectangle  (0.6,0.6);
         \fill[white] (-0.6,-0.3) rectangle  (0.6,0);

        \node at (0,0) [below] {{\tiny $ \tfrac{1}{2}$}};
         \node at (0.166666,0) [below] {{\tiny $ \tfrac{2}{3}$}};
          \node at (-0.166666,0) [below] {{\tiny $ \tfrac{1}{3}$}};
           \node at (0.5,0) [below] {{\tiny $1$}};
           \node at (-0.5,0) [below] {{\tiny $0$}};
    \end{scope}
    
      \end{tikzpicture}
\caption{The fundamental polygon $\FF_\gamma(11)$ associated to the hyperbolic matrix $\gamma=\begin{psmallmatrix} 7 & -2 \\ 11 & -3 \end{psmallmatrix}$. The boundary consists of 4 semicircles as well as two incomplete arcs connected by the closed geodesic (thick) associated to $\gamma$.}
\label{fig:F_A1}
\end{figure}
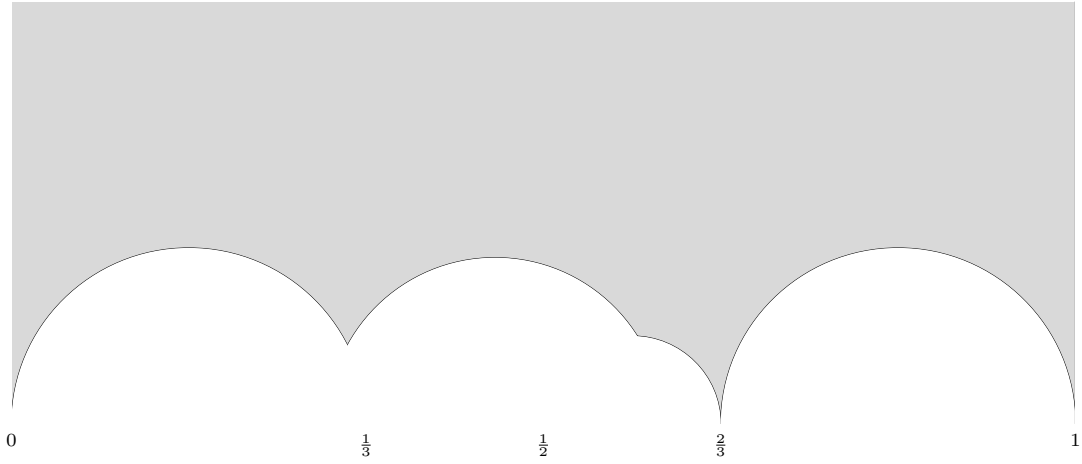
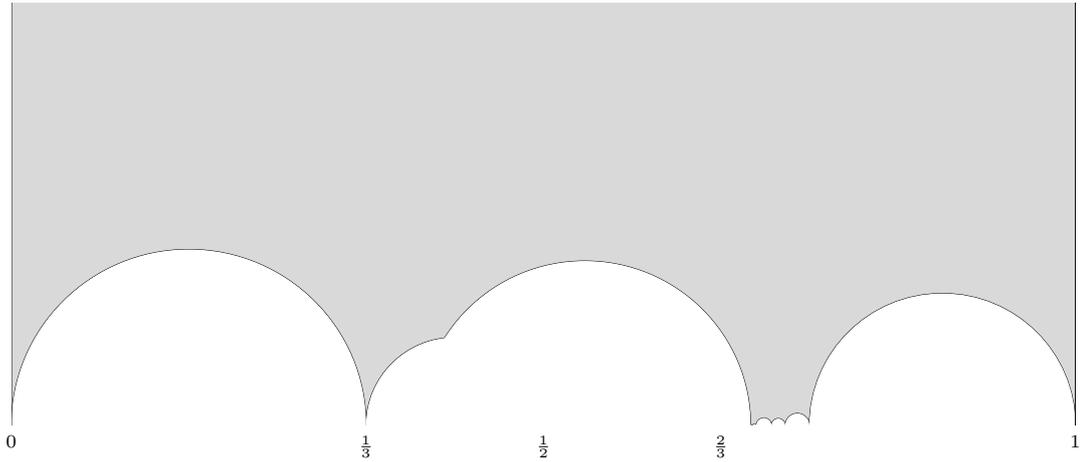
\begin{figure}[ht]
\centering
\begin{tikzpicture}[scale=14]
    \begin{scope}
      \clip (-0.6,-0.1) rectangle (0.6,0.4);
       \fill[black, opacity= 0.15]  (-0.6,-0.1) rectangle (0.6,0.4);

         \draw[ultra thin] (-0.083333,0) circle(0.08333333);
      \draw[ultra thin] (0.0833333,0) circle(0.083333333);
         \draw[ultra thin]  (-0.3333333,0) circle(0.166666666);
    
\draw[ultra thin] (0.6666666*0.5+0.6923076923*0.5-0.5,0) circle(0.6923076923*0.5-0.66666*0.5);
\draw[ultra thin] (0.6944444444*0.5+0.6923076923*0.5-0.5,0) circle(-0.6923076923*0.5+0.6944444444*0.5);
\draw[ultra thin] (0.6944444444*0.5+0.6956521739*0.5-0.5,0) circle(0.6956521739*0.5-0.6944444444*0.5);
\draw[ultra thin] (0.696969697*0.5+0.6956521739*0.5-0.5,0) circle(-0.6956521739*0.5+0.696969697*0.5);
\draw[ultra thin] (0.696969697*0.5+0.7*0.5-0.5,0) circle(0.7*0.5-0.696969697*0.5);

\draw[ultra thin] (0.7*0.5+0.7142857143*0.5-0.5,0)circle(-0.7*0.5+0.7142857143*0.5);
\draw[ultra thin] (0.7272727273*0.5+0.7142857143*0.5-0.5,0)circle(0.7272727273*0.5-0.7142857143*0.5);
\draw[ultra thin] (0.7272727273*0.5+0.75*0.5-0.5,0)circle(-0.7272727273*0.5+0.75*0.5);
\draw[ultra thin] (0.75*0.5+1*0.5-0.5,0) circle(0.5-0.75*0.5);

\draw[thick] (0.383252223269068*0.5+0.694669854653010*0.5-0.5,0)circle(-0.383252223269068*0.5+0.694669854653010*0.5);      
           \fill[white](0.383252223269068*0.5+0.694669854653010*0.5-0.5,0)circle(-0.383252223269068*0.5+0.694669854653010*0.5);

          \fill[white](-0.083333,0) circle(0.08333333);
      \fill[white](0.0833333,0) circle(0.083333333);
         \fill[white] (-0.3333333,0) circle(0.166666666);
\fill[white] (0.6666666*0.5+0.6923076923*0.5-0.5,0) circle(0.6923076923*0.5-0.66666*0.5);
\fill[white] (0.6944444444*0.5+0.6923076923*0.5-0.5,0) circle(-0.6923076923*0.5+0.6944444444*0.5);
\fill[white] (0.6944444444*0.5+0.6956521739*0.5-0.5,0) circle(0.6956521739*0.5-0.6944444444*0.5);
\fill[white] (0.696969697*0.5+0.6956521739*0.5-0.5,0) circle(-0.6956521739*0.5+0.696969697*0.5);
\fill[white] (0.696969697*0.5+0.7*0.5-0.5,0) circle(0.7*0.5-0.696969697*0.5);

\fill[white](0.7*0.5+0.7142857143*0.5-0.5,0)circle(-0.7*0.5+0.7142857143*0.5);
\fill[white] (0.7272727273*0.5+0.7142857143*0.5-0.5,0)circle(0.7272727273*0.5-0.7142857143*0.5);
\fill[white](0.7272727273*0.5+0.75*0.5-0.5,0)circle(-0.7272727273*0.5+0.75*0.5);
\fill[white] (0.75*0.5+1*0.5-0.5,0) circle(0.5-0.75*0.5);

                 \draw (-0.5,0) -- (-0.5,0.6);
        \draw  (0.5,0) -- (0.5,0.6);

     \fill[white](-0.5,0) rectangle  (-0.6,0.6);
        \fill[white] (0.5,0) rectangle  (0.6,0.6);
         \fill[white] (-0.6,-0.3) rectangle  (0.6,0);

        \node at (0,0) [below] {{\tiny $ \tfrac{1}{2}$}};
         \node at (0.166666,0) [below] {{\tiny $ \tfrac{2}{3}$}};
          \node at (-0.166666,0) [below] {{\tiny $ \tfrac{1}{3}$}};
           \node at (0.5,0) [below] {{\tiny $1$}};
           \node at (-0.5,0) [below] {{\tiny $0$}};
    \end{scope}
    
      \end{tikzpicture}
\caption{The fundamental polygon $\FF_\gamma(11)$ associated to the hyperbolic matrix $\gamma=\begin{psmallmatrix} 107 & -41 \\ 154 & -59 \end{psmallmatrix}$. The boundary consists of 9 semicircles (two of which are very small) as well as two incomplete geodesic segments connected by the closed geodesic (thick) associated to $\gamma$.}
\label{fig:F_A2}
\end{figure}
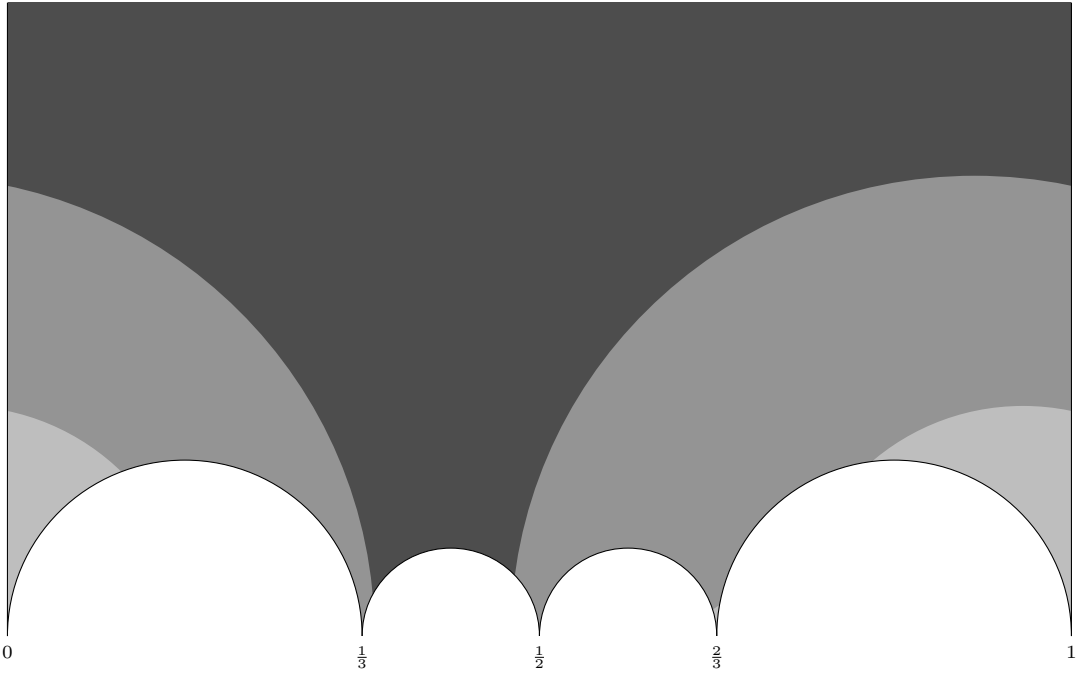
\begin{figure}[ht]
\centering
\begin{tikzpicture}[scale=14]
    \begin{scope}
      \clip (-0.6,-0.1) rectangle (0.6,0.6);
     \fill[black, opacity= 0.7]  (-0.59,-0.1) rectangle (0.59,0.6);
          \fill[white,opacity=0.4] (0.40909,0) circle(0.435984);
          
          \fill[white,opacity=0.4] (0.40909-1,0) circle(0.435984); 
          \fill[white,opacity=0.4] (1.17252139-0.5-0.2179841352,0) circle(0.2179841352); 
          \fill[white,opacity=0.4] (1.17252139-0.5-1-0.2179841352,0) circle(0.2179841352); 
          \fill[white,opacity=0.4] (0.718838463-0.5-0.03353053045,0) circle(0.03353053045);

          \fill[white](-0.083333,0) circle(0.08333333);
      \fill[white](0.0833333,0) circle(0.083333333);
         \fill[white] (0.33333333,0) circle(0.16666666);
         \fill[white] (-0.3333333,0) circle(0.166666666);
         \draw (-0.083333,0) circle(0.08333333);
      \draw (0.0833333,0) circle(0.083333333);
         \draw (0.33333333,0) circle(0.16666666);
         \draw  (-0.3333333,0) circle(0.166666666);
     \fill[white](-0.5,0) rectangle  (-0.6,0.6);
        \fill[white] (0.5,0) rectangle  (0.6,0.6);
         \fill[white] (-0.6,-0.3) rectangle  (0.6,0);
      \draw (-0.5,0) -- (-0.5,0.6);
        \draw  (0.5,0) -- (0.5,0.6);
        \node at (0,0) [below] {{\tiny $ \tfrac{1}{2}$}};
         \node at (0.166666,0) [below] {{\tiny $ \tfrac{2}{3}$}};
          \node at (-0.166666,0) [below] {{\tiny $ \tfrac{1}{3}$}};
           \node at (0.5,0) [below] {{\tiny $1$}};
           \node at (-0.5,0) [below] {{\tiny $0$}};
    
    \end{scope}
    
      \end{tikzpicture}
\caption{The projection onto $X_0(11)$ of the fundamental polygon $\FF_{\gamma}(11)$ associated to the hyperbolic matrix $\gamma = \begin{psmallmatrix} 19 & 10 \\ 55 & 29 \end{psmallmatrix}$, which corresponds to the narrow ideal class $J \in \Cl_{92}^+$.}\label{fig:1}
\end{figure}

We now show how to use these explicit fundamental domains to extract geometric information about the hyperbolic orbifold $\Gamma_{\gamma}(q)\backslash \NN_{\gamma}(q)$.
 
\begin{proposition}
\label{prop:lowerboundvol}
Let $\gamma \in \Gamma_0(q)$ be a (nonexceptional) primitive hyperbolic matrix such that the axis $S_{\gamma}$ given by \eqref{eqn:semicircle2} intersects the fundamental polygon $\PP(q)$.
\begin{enumerate}[leftmargin=*,label=\textup{(\arabic*)}]
\item The polygon $\tilde{\FF}_{\gamma}(q)$ is a fundamental domain for $\Gamma_{\gamma}(q)\backslash \Hb$.
\item The subgroup $\Gamma_\gamma(q)\leq \PSL_2(\Z)$ is thin.
\item The orbifold $\Gamma_{\gamma}(q)\backslash \NN_{\gamma}(q)$ has genus $0$, a unique cusp, and a unique boundary component, which projects to the oriented closed geodesic $\CC_\gamma(1)\subset X_0(1)$. Furthermore, the corresponding geodesic segment on the boundary of $\FF_{\gamma}(q)$ projects to $\CC_{\gamma}(q)\subset X_0(q)$.
\item We have that
\begin{equation}
\label{eqn:orbifoldvolume}
\vol (\Gamma_{\gamma}(q)\backslash \NN_{\gamma}(q))= \pi |\Omega_{\gamma}(q)|-\frac{2\pi}{3}e_{3,\gamma},
\end{equation}
where $e_{3,\gamma}$ denotes the number of conjugacy classes of order $3$ subgroups in $\Gamma_{\gamma}(q)$.
\end{enumerate}
\end{proposition}

By comparing the formula \eqref{eqn:orbifoldvolume} to \cite[(3.3)]{DIT16}, we see that the cardinality of $\Omega_{\gamma}(q)$ plays the role of the length $\ell_A$ of the minus continuous fraction in the level $1$ case in \cite{DIT16}.

\begin{proof}[Proof of {\hyperref[prop:lowerboundvol]{Proposition \ref*{prop:lowerboundvol}}}]
It is a simple consequence of Poincar\'{e}'s polygon theorem \cite{Mas71} that $\tilde{\FF}_{\gamma}(q)$ is a fundamental domain for $\Gamma_{\gamma}(q)\backslash \Hb$ with side pairing elements $\Omega_\gamma(q)\cup\{\gamma,\gamma^{-1}\}$. Since the volume of $\tilde{\FF}_\gamma(q)$ is infinite and $\Gamma_\gamma(q)$ contains two hyperbolic elements with different axes (e.g. $\gamma$ and $\sigma\gamma\sigma^{-1}$ for any $\sigma\in\Omega_{\gamma}(q)$), we conclude that it is a thin subgroup of $\PSL_2(\Z)$. 

It follows from the explicit side pairing on $\tilde{\FF}_{\gamma}(q)$ that there is a unique cusp and that the genus is $0$. Furthermore, $\Gamma_{\gamma}(q)\backslash \Hb$ has exactly one boundary circle, namely corresponding on $\tilde{\FF}_{\gamma}(q)$ to the interval of $\Pb^1(\R)$ between the right-most vertex $v$ of (the complete geodesic containing) $\LL_{i(m_{\gamma})^\ast}$ and $\gamma v$ (which is the left-most vertex of the complete geodesic containing $\gamma\LL_{i(m_{\gamma})^\ast}$). Notice here that $v\neq \gamma v$ since $\gamma$ is hyperbolic. Consider the oriented geodesic $\CC$ connecting the intersections between the axis $S_\gamma\subset \Hb$ of $\gamma$ with $\LL_{i(m_{\gamma})^\ast}$ and $\gamma\LL_{i(m_{\gamma})^\ast}$ respectively. By construction, this is freely homotopic to the boundary circle of $\Gamma_{\gamma}(q)\backslash \Hb$. Since there is a unique closed geodesic with this property, we conclude that (the projection of) $\CC$ is exactly the boundary component of $\Gamma_\gamma(q)\backslash \NN_\gamma(q)$. In particular, the boundary component of the fundamental domain $\FF_{\gamma}(q)$ equals the geodesic $\CC$ which projects to $\CC_{\gamma}(q)\subset X_0(q)$ by construction. From this, it also follows that the boundary component of $\Gamma_{\gamma}(q)\backslash \NN_{\gamma}(q)$ projects to $\CC_{\gamma}(1)\subset X_0(1)$. 

As above, denote by $e_{2,\gamma}$ the number of elliptic points of order $2$ of $\Gamma_{\gamma}(q)\backslash \NN_{\gamma}(q)$ and by $e_{3,\gamma}$ the number of elliptic points of order $3$. Then clearly we have that
\[|\Omega_{\gamma}(q)|= e_{2,\gamma} + 2 e_{3,\gamma}.\]
Thus the Gau{\ss}--Bonnet theorem \cite[(2.8)]{DIT16} gives that
\[\vol (\Gamma_{\gamma}(q)\backslash \NN_{\gamma}(q))=2\pi\left(\frac{1}{2}e_{2,\gamma}+\frac{2}{3}e_{3,\gamma}\right)= \pi |\Omega_{\gamma}(q)|-\frac{2\pi}{3}e_{3,\gamma},\]
as desired.
\end{proof}

For $\gamma = \gamma_Q$ associated to a Heegner form $Q \in \QQ_D(q)$, we write
\[\Gamma_{Q}(q) \coloneqq \Gamma_{\gamma_Q}(q), \qquad \FF_{Q}(q) \coloneqq \FF_{\gamma_Q}(q), \qquad \Omega_{Q,j}(q) \coloneqq \Omega_{\gamma_Q,j}(q), \qquad \Omega_{Q}(q) \coloneqq \Omega_{\gamma_Q}(q).\]
Notice that different choices of $\Gamma_0(q)$-equivalent Heegner forms $Q \in \QQ_D(q)$ (all corresponding to the same narrow ideal class $A \in \Cl_D^+$) give rise to subgroups $\Gamma_{Q}(q)$ of $\PSL_2(\Z)$ that are $\Gamma_0(q)$-conjugates; furthermore, the projection of $\FF_{Q}(q)$ to $X_0(q)$ is independent of this choice. For $A \in \Cl_D^+$, we denote by $\Gamma_A(q) \backslash \NN_A(q)$ any such orbifold $\Gamma_{Q}(q)\backslash \NN_{Q}(q)$ and by $\Omega_A(q)$ and $\FF_A(q)$ the corresponding set of matrices and fundamental domain. As we are interested in the projection of $\FF_A(q)$ to $X_0(q)$, we may use this abuse of notation with impunity.

\begin{remark}
Recall that $Q = [a,b,c] \in \QQ_D(q)$ is such that $(a,b,c) = 1$, $a \equiv 0 \pmod{q}$, and $b \equiv r \pmod{2q}$, where $r$ is a fixed choice of residue class modulo $2q$ for which $r^2 \equiv D \pmod{4q}$. As $q$ is prime, there are two such possible choices of $r \pmod{2q}$. Thus there are two closed geodesics, and hence two hyperbolic orbifolds, associated to each narrow ideal class $A \in \Cl_D^+$: one to each choice of $r$. By fixing $r$, we are thereby fixing a choice of one of these two closed geodesics and one of these two hyperbolic orbifolds associated to $A$. Since we regard $r$ as being fixed, we suppress from the notation the dependence of $\Gamma_A(q) \backslash \NN_A(q)$ and of $\FF_A(q)$ on $r$.
\end{remark}

\subsubsection{A Lower Bound for the Volume}

In this section, we use the formula \eqref{eqn:orbifoldvolume} to obtain a lower bound for the hyperbolic volume of $\FF_A(q)$. In the level $1$ setting, this was obtained in \cite{DIT16} by relating the volume to minus continued fractions and using some explicit relations to the regulator $\log \epsilon_D$, while in \cite[Section 5.1]{NT25}, such a lower bound was obtained using geometric considerations relying on equidistribution of the closed geodesics. In the level $q$ setting, on the other hand, such routes do not seem to be readily available. We instead use a group-theoretic approach. 

For $\gamma \in \Gamma_0(q)$, let $\wl(\gamma)$ denote the word length of $\gamma$ in the independent generators defined from $\PP(q)$, namely the labels $\alpha_1,\ldots, \alpha_N$ (forgetting inverses). Given a Heegner form $Q\in \QQ_D(q)$ with associated hyperbolic matrix $\gamma_Q$, we first observe that the entries of $\gamma_Q$ are bounded by $q^{O(\wl(\gamma_Q))}$ since the Frobenius norms of the labels $\alpha_j$ are $O(q)$. In particular, we have that
\[\Tr(\gamma_Q) \ll q^{O( \wl(\gamma_Q))},\]
which implies that 
\[\wl(\gamma_Q)\gg \frac{\log \Tr(\gamma_Q)}{\log q} \gg \frac{\log \epsilon_D}{\log q}.\]
This is close to what we want. Alas, a lower bound for the word length does not quite give a lower bound for $|\Omega_Q(q)|$ in general. For a \emph{generic} hyperbolic matrix $\gamma\in \Gamma_0(q)$ of word length tending to infinity, it is not hard to see that in fact $|\Omega_\gamma(q)|\asymp q \wl(\gamma)$. Thus one might hope that for subgroups $H\subset \Cl_D^+$ of sufficiently small index (say, $[\Cl_D^+:H] \ll D^{\delta}$ for some small $\delta>0$), one has that
\[\sum_{A\in CH} \vol (\FF_A(q)) \gg q |H| \log \epsilon_D\]
for any coset $CH \subset \Cl_D^+$ (cf.\ \cite[Proposition 1]{DIT16}). We make progress towards by showing the following.

\begin{proposition}
For $A\in \Cl_D^+$, we have that
\begin{equation}
\label{eqn:volbound}
\vol(\FF_A(q)) \gg \frac{\log \epsilon_D}{\log qD}.
\end{equation}
\end{proposition}

\begin{proof}
Let $Q \in \QQ_{D}(q)$ be a Heegner form corresponding to the narrow ideal class $A\in \Cl_D^+$ with corresponding hyperbolic matrix $\gamma_Q \in \Gamma_0(q)$ such that its axis $S_Q \subset \Hb$ (i.e.\ the infinite geodesic fixed by $\gamma_Q$) intersects $\PP(q)$. We start by bounding the word length $\mathrm{wl}(\gamma_Q)$ in terms of $D$. Recall from \cite{Kul91} that the special fundamental polygon $\PP(q)$ is tiled by 
\begin{equation}
\label{eqn:fund}
\left\{z\in \Hb: |z|\geq 1, \ 0 \leq \Re(z) < \frac{1}{2}\right\}.
\end{equation} 
The word length $\mathrm{wl}(\gamma_Q)$ of $\gamma_Q$ with respect to the independent generators of $\Gamma_0(q)$ defined from $\PP(q)$ (forgetting inverses) is precisely the number of intersections between the geodesic $\CC_A(q)$ from $z_Q$ to $\gamma_Q z_Q$ and the $\Gamma_0(q)$-translates of $\PP(q)$ and thus is bounded by the number of intersections with \eqref{eqn:fund}. We shall now relate this intersection number to continued fractions.

Consider the tessellation of $\Hb$ by the \emph{Farey triangle} $\triangle_\mathrm{Far}$ with vertices $\{0,1,\infty\}$. Note that $\triangle_\mathrm{Far}$ is tiled by six copies of the triangle \eqref{eqn:fund}. As explained in \cite[Theorem A and Corollary 3.3.5]{Ser85}, the \emph{left-right cutting sequence} of the axis $S_Q$ of $\gamma_Q$ with respect to this tessellation is periodic with period $a_1 + \cdots + a_{\ell'}$, where $[b_0,\ldots, b_r,\overline{a_1,\ldots, a_{\ell'}}]$ is the continued fraction expansion of the attracting endpoint of $S_Q$, which is eventually periodic with period $\ell'$. Consider the fundamental domain $\FF_0$ for $\PSL_2(\Z)$ with vertices $\{0,\frac{1 + \sqrt{-3}}{2},\infty\}$ and side pairing elements $\{S,TS, ST^{-1}\}$. Note that $\triangle_\mathrm{Far}$ is tiled by three copies of $\FF_0$ such that all the sides of $\triangle_\mathrm{Far}$ correspond to the side of $\FF_0$ containing $0$ and $\infty$ with label $S$. Furthermore, a ``left'' (respectively ``right'') in the cutting sequence corresponds, in terms of $\FF_0$, to the label $S$ followed by either the label $TS$ or twice its inverse (respectively $ST^{-1}$ or twice its inverse). Since $(ST^{-1})^2 = TS$, we conclude that the cutting sequence uniquely determines the representation of $\gamma_Q \in \PSL_2(\Z)$ in terms of the independent generators $S,TS$ of $\PSL_2(\Z)$ of order two and three respectively. In particular, the minimal period of the cutting sequence of $S_Q$ is at most a third of the length of the Morse code of $\gamma_Q$ with respect to $\FF_0$. Since $\FF_0$ is tiled by two copies of the triangle \eqref{eqn:fund}, we conclude that
\begin{equation}
\label{eq:wlbound}
\mathrm{wl}(\gamma_Q)\leq 6(a_1 + \cdots + a_{\ell'}) \ll D^{3/2}
\end{equation}
using the elementary bounds $a_i\ll\sqrt{D}$ and $\ell'\ll D$, which follow from the classical algorithm for continued fractions as described in \cite{Hic73}.

Let $0< j(1)<\ldots< j(\ell)< m_\gamma+1$ be the nonzero indices in the Morse code \eqref{eq:morsecode} of $\gamma_Q$ for which $\Omega_{Q,j(k)}(q)=\emptyset$ with $\Omega_{Q,j}(q)=\Omega_{\gamma_Q,j}(q)$ as defined in \eqref{eqn:Omega} and \eqref{eqn:Omega1}. Put $j(0)=0$ and $j(\ell+1)=m_\gamma+1$. We now argue that for all $0\leq k\leq \ell$, the product
\[\beta_k \coloneqq \alpha_{i(j(k))}\cdots \alpha_{i(j(k+1)-1)}\]
can be written as $\alpha_{h}\alpha_{t(h)}\cdots \alpha_{t^{(m)}(h)}$ for some $h\in\Z/N\Z$ and $m\geq0$. Assume first that $e_3=0$. Then  $\Omega_{Q,j}(q)=\emptyset$ means exactly that $i(j)=i(j-1)^\ast-1=t(i(j-1))$ and the claim follows. In the presence of $3$-torsion, $\Omega_{Q,j}(q)=\emptyset$ implies that $\arc(i(j),i(j-1)^\ast)\subset \{i(j-1),i(j)^\ast\}$. We need to consider the four possible geometric configurations corresponding to the nonempty subsets of $\{i(j-1),i(j)^\ast\}$. If $\arc(i(j),i(j-1)^\ast)=\emptyset$, we get as above that $i(j)=t(i(j-1))$. Next, assume that $\arc(i(j),i(j-1)^\ast)=\{i(j-1)\}$. Then since $\Omega_{Q,j}(q)=\emptyset$, we conclude that $\alpha_{i(j-1)}^3=1$, $i(j-1)<i(j-1)^\ast$ and $i(j)=i(j-1)-1$. Furthermore, since $j\notin \mathcal{I}_3$, we must have that $i(j-2)=i(j-1)$, implying that $i(j-1)=t(i(j-2))$. Since we have $\alpha_{i(j-2)}\alpha_{i(j-1)}=\alpha_{i(j-1)^\ast}$, $i(j-2)^\ast\in \arc(i(j-2),i(j-3)^\ast)$, and $i(j)=(i(j-1)^\ast)^\ast-1=t(i(j-1)^\ast)$, we see that we can alter the Morse code to get the desired shape. The two final cases can be treated similarly, which yields the claim.

We conclude by \hyperref[lem:ast]{Lemma \ref*{lem:ast}} that we may write
\[ \gamma_Q= \beta_0\beta_1\cdots \beta_\ell\]
where for all $0\leq k\leq \ell$ either
\begin{enumerate}
\item $\beta_k=\alpha_j$ is a label of $\PP(q)$, in which case the Frobenius norm is bounded by $O(q)$; 
\item $\beta_k=T^n$ for some $n\in \Z$, in which case the Frobenius norm is bounded by $O(|n|)$; or
\item we have that
\[\beta_k=\gamma_{j_1}^{-1} \begin{pmatrix}1 & 0 \\q & 1 \end{pmatrix}^n \gamma_{j_2},\]
for some $j_1,j_2 \in \{1,\ldots,4g+e_2+e_3+1\}$ (with $\gamma_{j_i}$ defined as in \eqref{eqn:gammaj2}) and $n\in \Z $, in which case the Frobenius norm is bounded by $O(q^{\frac{3}{2}}\cdot q|n| \cdot q^{\frac{3}{2}})$ via \hyperref[lem:boundfrobnorm]{Lemma \ref*{lem:boundfrobnorm}}.
\end{enumerate} 

Now clearly the power $n$ of the parabolic elements appearing in $\beta_k$ satisfies
\[|n|\leq \wl (\beta_k)\leq \wl(\gamma_Q) \ll D^{3/2},\]
using the bound \eqref{eq:wlbound}. Thus the entries of $\gamma_Q$ are bounded by
\[q^{O(\ell)} D^{O(\ell)}.\]
By computing the trace, we see that
\[\ell\gg \frac{\log \epsilon_D}{\log qD}.\]
Since trivially $|\Omega_Q(q)|\geq \ell$ and $e_{3,\gamma_Q} \leq \frac{|\Omega_Q(q)|}{2}$, \hyperref[prop:lowerboundvol]{Proposition \ref*{prop:lowerboundvol}} gives the desired result.
\end{proof}

\begin{remark}
If we consider an \emph{arbitrary} special polygon $\PP(q)$ instead of the one constructed based on \hyperref[thm:boundentries]{Theorem \ref*{thm:boundentries}}, then in place of the bound $b_i\ll \sqrt{q}$ for the denominators in the Farey sequence of level $q$, we would only have that $b_i\ll e^{O(q)}$. The above argument yields in this case the weaker lower bound
\[\vol(\Gamma_Q(q)\backslash \NN_Q(q)) \gg \frac{\log \epsilon_D}{q + \log D}.\]
\end{remark}

\begin{remark}
To prove a polynomial bound for $\wl(\gamma_Q)$, one might instead use the bound $\wl(\gamma_Q) \leq \mathfrak{m}_A + \ell_A$, where $\mathfrak{m}_A$ is as in \cite[(2.1)]{DIT16}. Via the identity $\mathfrak{m}_A = 2\ell_A + \ell_{AJ}$ \cite[(46) and (49)]{DIT18}, this reduces the problem to bounding $\ell_A $ in terms of $D$. It is claimed in \cite[p.~968]{DIT16} that one can apply a general argument of Eichler \cite{Eic65} to show that $\ell_A \ll \log \epsilon_D$ uniformly for $A \in \Cl_D^+$, which would yield the desired result. Unfortunately, however, Eichler's argument does not apply directly: Eichler's result gives a method for determining upper bounds for the word length of the \emph{geometric code} of a closed geodesic (in the sense of \cite{Kat96}), whereas $\ell_A$ is the word length of the \emph{arithmetic code} of a closed geodesic (in the sense of \cite{Kat96}), and in general these need not coincide \cite[Theorem 1]{Kat96}.
\end{remark}

\section{Ad\`{e}lisation of Maa\ss{} Cusp Forms}
\label{sect:adelisation}

We review some standard notions about Maa\ss{} cusp forms of weight $k$, level $q$, and principal nebentypus, with an emphasis on forms of weight $0$ and the action of raising and lowering operators on such forms. We then describe the relation between such classical automorphic forms and ad\`{e}lic automorphic forms, highlighting the correspondence between Whittaker functions of representations of $\GL_2(\R)$ and $\GL_2(\Q_p)$, the Whittaker expansion of an ad\`{e}lic automorphic form, and the Fourier expansion of a classical automorphic form. This explicit correspondence is invaluable in \hyperref[sect:Weylsums]{Section \ref*{sect:Weylsums}}, where we prove an identity between integrals of Maa\ss{} forms over hyperbolic orbifolds and period integrals of ad\`{e}lic automorphic forms. Useful references for this material include \cite[Section 4]{DFI02}, \cite[Chapters 3 and 4]{GH11}, \cite{Sch02}, and \cite{Pop08}.

\subsection{Maa\ss{} Cusp Forms}

Let $k$ be an integer, $q$ be a positive integer, and denote by $\Cscr_k(\Gamma_0(q))$ the vector subspace of $L^2(\Gamma_0(q) \backslash \Hb)$ spanned by Maa\ss{} cusp forms of weight $k$, level $q$, and principal nebentypus, in the sense of \cite[Section 4]{DFI02}. Such a Maa\ss{} cusp form is a real-analytic function $f : \Hb \to \C$ for which
\begin{itemize}
\item $f$ is an eigenfunction of the weight $k$ Laplacian
\[\Delta_k = -y^2 \left(\frac{\dee^2}{\dee x^2} + \frac{\dee^2}{\dee y^2}\right) + iky \frac{\dee}{\dee x},\]
so that $\Delta_k f(z) = \lambda_f f(z)$ for some $\lambda_f \in \C$ (and necessarily $\lambda_f \in [\frac{1}{4} - \left(\frac{7}{64}\right)^2,\infty)$),
\item $f$ is automorphic, so that $j_{\gamma}(z)^{-k} f(\gamma z) = f(z)$ for all $z \in \Hb$ and $\gamma \in \Gamma_0(q)$, where for $g = \begin{psmallmatrix} a & b \\ c & d \end{psmallmatrix} \in \GL_2^+(\R)$, the space of $2 \times 2$ matrices with real entries and positive determinant,
\begin{equation}
\label{eqn:gzjgz}
gz \coloneqq \frac{az + b}{cz + d}, \qquad j_g(z) \coloneqq \frac{cz + d}{|cz + d|},
\end{equation}
\item $f$ is of moderate growth, and
\item $f$ is cuspidal, so that for each cusp $\bb$ of $\Gamma_0(q) \backslash \Hb$,
\[\int_{0}^{1} j_{\sigma_{\bb}}(z)^{-k} f(\sigma_{\bb} z) \, dx = 0\]
for all $y > 0$, where $\sigma_{\bb} \in \SL_2(\R)$ is a scaling matrix for $\bb$.
\end{itemize}

\subsubsection{The Fourier Expansion and Hecke Eigenvalues of a Maa\ss{} Cusp Form}

The Fourier expansion at the cusp at infinity of a weight $0$ Maa\ss{} cusp form $f \in \Cscr_0(\Gamma_0(q))$ is
\begin{equation}
\label{eqn:Fourier}
f(z) = \sum_{\substack{n = -\infty \\ n \neq 0}}^{\infty} \rho_f(n) W_{0,it_f}(4\pi|n|y) e(nx),
\end{equation}
where $W_{\alpha,\beta}$ denotes the classical Whittaker function and $t_f \in \R \cup i[-\frac{7}{64},\frac{7}{64}]$ is the spectral parameter of $f$, so that $\lambda_f = \frac{1}{4} + t_f^2$. If $f$ is additionally a Hecke--Maa\ss{} \emph{newform}, namely an eigenfunction of the $n$-th Hecke operator $T_n$ for all $n \in \N$ as well as the reflection operator $X : \Cscr_0(\Gamma_0(q)) \to \Cscr_0(\Gamma_0(q))$ given by $(X f)(z) \coloneqq f(-\overline{z})$, the Fourier coefficients $\rho_f(n)$ and Hecke eigenvalues $\lambda_f(n)$ of $f$ satisfy 
\begin{itemize}
\item $\rho_f(1) \lambda_f(n) = \sqrt{n} \rho_f(n)$ for $n \in \N$,
\item $\rho_f(n) = \epsilon_f \rho_f(-n)$ for $n \in \Z$, where $\epsilon_f \in \{1,-1\}$ is the parity of $f$, so that $X f = \epsilon_f f$,
\item for all $m,n \in \N$, the Hecke eigenvalues satisfy the multiplicativity relations
\[\lambda_f(m) \lambda_f(n) = \sum_{\substack{d \mid (m,n) \\ (d,q) = 1}} \lambda_f\left(\frac{mn}{d^2}\right), \qquad \lambda_f(mn) = \sum_{\substack{d \mid (m,n) \\ (d,q) = 1}} \mu(d) \lambda_f\left(\frac{m}{d}\right) \lambda_f\left(\frac{n}{d}\right),\]
\item for each $p \nmid q$, there exists $\alpha_f(q) \in \C$ satisfying $p^{-\frac{7}{64}} \leq |\alpha_f(q)| \leq p^{\frac{7}{64}}$ such that for all $r \geq 1$,
\begin{equation}
\label{eqn:HeckeSatakeunram}
\lambda_f(p^r) = \sum_{m = 0}^{r} \alpha_f(q)^{m} \alpha_f^{-1}(q)^{r - m}.
\end{equation}
\item for each $p \parallel q$, there exists $\alpha_f(q) \in \{1,-1\}$ such that for all $r \geq 1$,
\begin{equation}
\label{eqn:HeckeSatakeram}
\lambda_f(p^r) = \frac{\alpha_f(q)^r}{p^{r/2}},
\end{equation}
\item for each prime $p$ for which $p^2 \mid q$, we have that $\lambda_f(p^r) = 0$ for all $r \geq 1$.
\end{itemize}
Furthermore, if $f \in \Cscr_0(\Gamma_0(q))$ is a Hecke--Maa\ss{} newform and $q$ is squarefree, then the $L^2$-norm of $f$ and the first Fourier coefficient $\rho_f(1)$ of $f$ satisfy the relation \cite[Lemma 4.6]{HK20}
\begin{equation}
\label{eqn:L2rho1}
\int_{\Gamma_0(q) \backslash \Hb} |f(z)|^2 \, d\mu(z) = \frac{2q |\rho_f(1)|^2 L(1,\ad f)}{\cosh \pi t_f}.
\end{equation}
Here we note that the measure $d\mu(z) = y^{-2} \, dx \, dy$ is such that
\begin{equation}
\label{eqn:Gamma0qvol}
\vol(\Gamma_0(q) \backslash \Hb) = \frac{\pi}{3} \nu(q), \qquad \nu(q) \coloneqq [\Gamma : \Gamma_0(q)] = q \prod_{p \mid q} \left(1 + \frac{1}{p}\right).
\end{equation}

\subsubsection{Raising and Lowering Operators}

The weight $k$ raising operator
\[R_k \coloneqq \frac{k}{2} + \left(z - \overline{z}\right) \frac{\dee}{\dee z} = \frac{k}{2} + iy \left(\frac{\dee}{\dee x} - i \frac{\dee}{\dee y}\right)\]
acts on $\Cscr_k(\Gamma_0(q))$ and raises the weight by $2$; that is, its image lies in $\Cscr_{k + 2}(\Gamma_0(q))$. Similarly, the weight $k$ lowering operator
\[L_k \coloneqq -\frac{k}{2} - \left(z - \overline{z}\right) \frac{\dee}{\dee \overline{z}} = -\frac{k}{2} - iy \left(\frac{\dee}{\dee x} + i \frac{\dee}{\dee y}\right)\]
maps $\Cscr_k(\Gamma_0(q))$ to $\Cscr_{k - 2}(\Gamma_0(q))$. From \cite[Proposition 3.9.13]{GH11}, if $f \in \Cscr_0(\Gamma_0(q))$ is a Hecke--Maa\ss{} newform of weight $0$ and level $q$ with Fourier expansion \eqref{eqn:Fourier}, then the Fourier expansion of the weight $k$ Maa\ss{} cusp form $F_k \in \Cscr_k(\Gamma_0(q))$ defined by
\begin{equation}
\label{eqn:Fk}
F_k \coloneqq \begin{dcases*}
R_{k - 2} \cdots R_0 f & if $k \in 2\N$,	\\
f & if $k = 0$,	\\
L_{k + 2} \cdots\LL_0 f& if $k \in -2\N$,
\end{dcases*}
\end{equation}
is given by
\begin{multline}
\label{eqn:Fkexppos}
F_k(z) = \frac{\Gamma\left(\frac{k + 1}{2} + it_f\right) \Gamma\left(\frac{k + 1}{2} - it_f\right)}{\Gamma\left(\frac{1}{2} + it_f\right) \Gamma\left(\frac{1}{2} - it_f\right)} \epsilon_f \rho_f(1) \sum_{n = -\infty}^{-1} \frac{\lambda_f(|n|)}{\sqrt{|n|}} W_{-\frac{k}{2},it_f}(4\pi|n|y) e(nx)	\\
+ (-1)^{\frac{k}{2}} \rho_f(1) \sum_{n = 1}^{\infty} \frac{\lambda_f(n)}{\sqrt{n}} W_{\frac{k}{2},it_f}(4\pi ny) e(nx)
\end{multline}
if $k \in 2\N \cup \{0\}$, while if $k \in -2\N$, the Fourier expansion is
\begin{multline}
\label{eqn:Fkexpneg}
F_k(z) = (-1)^{\frac{k}{2}} \epsilon_f \rho_f(1) \sum_{n = -\infty}^{-1} \frac{\lambda_f(|n|)}{\sqrt{|n|}} W_{-\frac{k}{2},it_f}(4\pi|n|y) e(nx)	\\
+ \frac{\Gamma\left(\frac{1 - k}{2} + it_f\right) \Gamma\left(\frac{1 - k}{2} - it_f\right)}{\Gamma\left(\frac{1}{2} + it_f\right) \Gamma\left(\frac{1}{2} - it_f\right)} \rho_f(1) \sum_{n = 1}^{\infty} \frac{\lambda_f(n)}{\sqrt{n}} W_{\frac{k}{2},it_f}(4\pi ny) e(nx).
\end{multline}

\subsubsection{Atkin--Lehner Operators}

Let $q$ be squarefree. For each divisor $q_1$ of $q$, so that $q = q_1 q_2$, we define an \emph{Atkin--Lehner operator} $W_{q_1} \in \SL_2(\R)$ as in \eqref{eqn:ALop}, namely
\[W_{q_1} \coloneqq \begin{pmatrix} a \sqrt{q_1} & \frac{b}{\sqrt{q_1}} \\ c q_2 \sqrt{q_1} & d\sqrt{q_1} \end{pmatrix},\]
where $a,b,c,d \in \Z$ are such that $adq_1 - bcq_2 = 1$. If $f \in \Cscr_k(\Gamma_0(q))$ is a Hecke--Maa\ss{} \emph{newform}, then it is an \emph{eigenfunction} of each Atkin--Lehner operator, in the sense that there exists some constant $\eta_f(q_1) \in \{1,-1\}$, dependent on $f$ and $q_1$ but not on $a,b,c,d$, such that
\[j_{W_{q_1}}(z)^{-k} f(W_{q_1} z) = \eta_f(q_1) f(z)\]
for all $z \in \Hb$, where $j_g(z)$ is as in \eqref{eqn:gzjgz}.

\subsection{Eisenstein Series}

We recall that the Eisenstein series $E(z,s)$ for $\Gamma \backslash \Hb$ is given for $\Re(s) > 1$ by the absolutely convergent series
\[E(z,s) \coloneqq \sum_{\gamma \in \Gamma_{\infty} \backslash \Gamma} \Im(\gamma z)^s,\]
where $\Gamma_{\infty} \coloneqq \{\pm \begin{psmallmatrix} 1 & n \\ 0 & 1 \end{psmallmatrix} : n \in \Z\}$ is the stabiliser of the cusp at infinity. The Eisenstein series extends meromorphically to $\C$ with a simple pole at $s = 1$ with residue $\frac{1}{\vol(\Gamma \backslash \Hb)}$, independently of $z$. It is an eigenfunction of the weight $0$ Laplacian with eigenvalue $s(1 - s)$ and is automorphic and of moderate growth, but it is \emph{not} cuspidal.

For $s = \frac{1}{2} + it$, the Laplacian eigenvalue of $E(z,\frac{1}{2} + it)$ is $\frac{1}{4} + t^2$, the parity is $\epsilon = 1$, while the Fourier coefficients $\rho(n,t)$ and Hecke eigenvalues $\lambda(n,t)$ satisfy
\begin{itemize}
\item $\rho(1,t) \lambda(n,t) = \sqrt{n} \rho(n,t)$ for $n \in \N$,
\item $\rho(n,t) = \rho(-n,t)$ for $n \in \Z$,
\item for all $m,n \in \N$, the Hecke eigenvalues satisfy the multiplicativity relations
\[\lambda(m,t) \lambda(n,t) = \sum_{d \mid (m,n)} \lambda\left(\frac{mn}{d^2},t\right), \qquad \lambda(mn,t) = \sum_{d \mid (m,n)} \mu(d) \lambda\left(\frac{m}{d},t\right) \lambda\left(\frac{n}{d},t\right),\]
\item for all $n \in \N$, the Hecke eigenvalues are given explicitly by
\[\lambda(n,t) = \sum_{ab = n} a^{it} b^{-it},\]
\item the first Fourier coefficient is given explicitly by $\rho(1,t) = 1/\xi(1 + 2it)$, where $\xi(s) \coloneqq \pi^{-\frac{s}{2}} \Gamma(\frac{s}{2}) \zeta(s)$.
\end{itemize}

The weight $0$ raising and lowering operators act on $E(\cdot,\frac{1}{2} + it)$ and raise and lower the weight by $2$. The Fourier expansion of $E(z,\frac{1}{2} + it)$ reads
\begin{equation}
\label{eqn:FourierEis}
\begin{split}
E\left(z,\frac{1}{2} + it\right) & = y^{\frac{1}{2} + it} + \frac{\xi(1 - 2it)}{\xi(1 + 2it)} y^{\frac{1}{2} - it} + \sum_{\substack{n = -\infty \\ n \neq 0}}^{\infty} \rho(n,t) W_{0,it_f}(4\pi|n|y) e(nx)	\\
& = y^{\frac{1}{2} + it} + \frac{\xi(1 - 2it)}{\xi(1 + 2it)} y^{\frac{1}{2} - it} + \frac{1}{\xi(1 + 2it)} \sum_{\substack{n = -\infty \\ n \neq 0}}^{\infty} \frac{\lambda(|n|,t)}{\sqrt{|n|}} W_{0,it_f}(4\pi|n|y) e(nx),
\end{split}
\end{equation}
while the Fourier expansion of the weight $2$ Eisenstein series $(R_0 E)(z,\frac{1}{2} + it)$ is given by
\begin{multline}
\label{eqn:FkexEis}
(R_0 E)\left(z,\frac{1}{2} + it\right) = \left(\frac{1}{2} + it\right) y^{\frac{1}{2} + it} + \left(\frac{1}{2} - it\right) \frac{\xi(1 - 2it)}{\xi(1 + 2it)} y^{\frac{1}{2} - it}	\\
+ \frac{\left(\frac{1}{4} + t^2\right)}{\xi(1 + 2it)} \sum_{n = -\infty}^{-1} \frac{\lambda(|n|,t)}{\sqrt{|n|}} W_{-1,it}(4\pi|n|y) e(nx)	\\
- \frac{1}{\xi(1 + 2it)} \sum_{n = 1}^{\infty} \frac{\lambda(n,t)}{\sqrt{n}} W_{1,it_f}(4\pi ny) e(nx).
\end{multline}
We additionally define
\begin{equation}
\label{eqn:tildeR0E}
(\widetilde{R_0 E}) \left(z,\frac{1}{2} + it\right) \coloneqq (R_0 E)\left(z,\frac{1}{2} + it\right) - \left(\frac{1}{2} + it\right) \Im(z)^{\frac{1}{2} + it} - \left(\frac{1}{2} - it\right) \frac{\xi(1 - 2it)}{\xi(1 + 2it)} \Im(z)^{\frac{1}{2} - it}.
\end{equation}
This is square-integrable but no longer automorphic.

Finally, we record here the following useful result concerning the behaviour of an Eisenstein series at a cusp.

\begin{lemma}
\label{lem:Eisderivdecay}
Let $q$ be squarefree. For a cusp $\bb$ of $\Gamma_0(q) \backslash \Hb$ with scaling matrix $\sigma_{\bb} \in \Gamma_0(q)$ and for $\ell \mid q$, we have that
\[\left. \frac{\dee}{\dee z}\right|_{z = x + iY} E\left(\sigma_{\bb} (\ell z),\frac{1}{2} + it\right) \ll_{q,t} \frac{1}{\sqrt{Y}}.\]
\end{lemma}

\begin{proof}
We begin by noting that $E(\ell z,\frac{1}{2} + it)$ may be written as a finite linear combination of Eisenstein series $E_{\aa}(z,\frac{1}{2} + it)$ associated to cusps $\aa$ of $\Gamma_0(q) \backslash \Hb$ \cite[Theorem 7.1]{You19}. From \cite[(6.18)]{Iwa02}, the Fourier expansion of $E_{\aa}(\sigma_{\bb} z, \frac{1}{2} + it)$ is of the form
\[E_{\aa}\left(\sigma_{\bb} z,\frac{1}{2} + it\right) = \delta_{\aa,\bb} y^{\frac{1}{2} + it} + \varphi_{\aa,\bb}\left(\frac{1}{2} + it\right) y^{\frac{1}{2} - it} + \sum_{\substack{n = -\infty \\ n \neq 0}} \rho_{\aa,\bb}(n,t) W_{0,it}(4\pi|n|y) e(nx),\]
where $\varphi_{\aa,\bb}(s)$ is an entry of the scattering matrix, while the Fourier coefficients $\rho_{\aa,\bb}(n,t)$ grow at most polynomially in $|n|$. The desired result now follows by differentiating term by term.
\end{proof}

\subsection{Ad\`{e}lic Automorphic Forms}

\subsubsection{The Ad\`{e}lic Lift of a Maa\ss{} Cusp Form}
\label{sect:adeliclift}

Following \cite[Sections 4.11 and 4.12]{GH11}, we describe the ad\`{e}lic lift of a Maa\ss{} cusp form $f \in \Cscr_k(\Gamma_0(q))$. We first lift $f \in \Cscr_k(\Gamma_0(q))$ to a function $\widetilde{f} : \GL_2^+(\R) \to \C$ defined via
\begin{equation}
\label{eqn:tildeftof}
\widetilde{f}(g) \coloneqq j_g(i)^{-k} f(gi).
\end{equation}
For all $g \in \GL_2^+(\R)$ and $\theta \in [0,2\pi)$, this satisfies
\[\widetilde{f}\left(g \begin{pmatrix} \cos \theta & \sin \theta \\ - \sin \theta & \cos \theta \end{pmatrix}\right) = e^{ik\theta} \widetilde{f}(g).\]

Next, we lift $\widetilde{f}$ to an ad\`{e}lic automorphic form $\phi = \phi_f$ on $\GL_2(\A_{\Q})$, where $\A_{\Q}$ denotes the ring of ad\`{e}les of $\Q$. To describe this lift, we first let $K_0(q) \ni k = (k_{\infty},k_2,k_3,k_5,\ldots)$ denote the congruence subgroup of $\GL_2(\A_{\Q})$ of the form
\begin{equation}
\label{eqn:K0(q)}
K_0(q) \coloneqq \left\{k \in \GL_2(\A_{\Q}) : k_{\infty} = 1_2, \ k_p = \begin{pmatrix} a_p & b_p \\ c_p & d_p \end{pmatrix} \in \GL_2(\Z_p) \text{ with } c_p \in p^r \Z_p \text{ if } p^r \parallel q\right\}.
\end{equation}
Here $1_2$ denotes the $2 \times 2$ identity matrix. We view $\GL_2^+(\R) \ni g_{\infty}$ as a subgroup of $\GL_2(\A_{\Q})$ via the embedding $g_{\infty} \mapsto (g_{\infty},1_2,1_2,\ldots)$. Finally, we view $\GL_2(\Q) \ni \gamma$ as a subgroup of $\GL_2(\A_{\Q})$ via the diagonal embedding $\gamma \mapsto (\gamma,\gamma,\gamma,\ldots)$. Then via the strong approximation theorem,
\begin{equation}
\label{eqn:strongapprox}
\GL_2(\A_{\Q}) = \GL_2(\Q) \GL_2^+(\R) K_0(q),
\end{equation}
so that every $g \in \GL_2(\A_{\Q})$ can be written (nonuniquely) as $g = \gamma g_{\infty} k$ for some $\gamma \in \GL_2(\Q)$, $g_{\infty} \in \GL_2^+(\R)$, and $k \in K_0(q)$. The ad\`{e}lic lift $\phi = \phi_f$ of a Maa\ss{} cusp form $f \in \Cscr_k(\Gamma_0(q))$ is then given by
\begin{equation}
\label{eqn:phistrongapprox}
\phi(g) = \phi(\gamma g_{\infty} k) \coloneqq \widetilde{f}(g_{\infty}).
\end{equation}
This is well-defined even though the decomposition $g = \gamma g_{\infty} k$ is not unique. In particular,
\begin{equation}
\label{eqn:phitof}
\phi(g) = f(x + iy)
\end{equation}
for $g = g_{\infty} = \begin{psmallmatrix} y & x \\ 0 & 1 \end{psmallmatrix} \in \GL_2^+(\R) \subset \GL_2(\A_{\Q})$ with $y > 0$ and $x \in \R$.

\subsubsection{The Whittaker Expansion of an Ad\`{e}lic Automorphic Form}

Let $\psi : \Q \backslash \A_{\Q} \to \C$ be the standard ad\`{e}lic additive character defined as in \cite[Definition 1.7.1]{GH11}, so that $\psi(u) = \psi_{\infty}(u_{\infty}) \prod_p \psi_p(u_p)$ for $u = (u_{\infty},u_2,u_3,\ldots) \in \A_{\Q}$ with $\psi_{\infty} : \R \to \C$ the additive character $\psi_{\infty}(u_{\infty}) \coloneqq e(u_{\infty}) \coloneqq e^{2\pi i u_{\infty}}$ and $\psi_p : \Q_p \to \C$ the standard unramified additive character defined in \cite[Definition 1.6.3]{GH11}. The Whittaker function $W_{\phi} : \GL_2(\A_{\Q}) \to \C$ of a cuspidal ad\`{e}lic automorphic form $\phi$ is
\[W_{\phi}(g) \coloneqq \int_{\Q \backslash \A_{\Q}} \phi\left(\begin{pmatrix} 1 & u \\ 0 & 1 \end{pmatrix} g\right) \overline{\psi}(u) \, du,\]
which satisfies
\begin{equation}
\label{eqn:Whittakerinvariance}
W_{\phi}\left(\begin{pmatrix} 1 & u \\ 0 & 1 \end{pmatrix} g\right) = \psi(u) W_{\phi}(g)
\end{equation}
for all $u \in \A_{\Q}$ and $g \in \GL_2(\A_{\Q})$. The automorphic form $\phi$ has the Whittaker expansion
\begin{equation}
\label{eqn:Whittakerexpansion}
\phi(g) = \sum_{\alpha \in \Q^{\times}} W_{\phi}\left(\begin{pmatrix} \alpha & 0 \\ 0 & 1 \end{pmatrix} g\right).
\end{equation}

\subsubsection{The Whittaker Expansion of an Ad\`{e}lic Lift}

Let $\phi = \phi_{F_k}$ be the ad\`{e}lic lift of a Maa\ss{} cusp form $F_k \in \Cscr_k(\Gamma_0(q))$ of weight $k$ associated to a Hecke--Maa\ss{} newform $f \in \Cscr_0(\Gamma_0(q))$ of weight $0$ as in \eqref{eqn:Fk}. Then $\phi$ is a pure tensor lying in the vector space of a cuspidal automorphic representation $\pi = \pi_f = \pi_{\infty} \otimes \bigotimes_p \pi_p$ of $\GL_2(\A_{\Q})$, where each $\pi_p$ is a generic irreducible admissible unitary representation of $\GL_2(\Q_p)$ and $\pi_{\infty}$ is a generic irreducible unitary Casselman--Wallach representation of $\GL_2(\R)$. In particular, for $g = (g_{\infty},g_2,g_3,\ldots) \in \GL_2(\A_{\Q})$, we have the factorisation
\begin{equation}
\label{eqn:Whittakerfact}
W_{\phi}(g) = c_{\phi} W_{\infty}(g_{\infty}) \prod_p W_p(g_p),
\end{equation}
where $c_{\phi}$ is a constant independent of $g$, each $W_p$ is a Whittaker function in the Whittaker model $\WW(\pi_p,\psi_p)$ of $\pi_p$, and similarly $W_{\infty} \in \WW(\pi_{\infty},\psi_{\infty})$.

We show in \hyperref[lem:nonarchimedeanWhittaker]{Lemma \ref*{lem:nonarchimedeanWhittaker}} that the Whittaker functions $W_p$ are such that for $\alpha \in \Q^{\times}$,
\[\prod_p W_p\begin{pmatrix} \alpha & 0 \\ 0 & 1 \end{pmatrix} = \begin{dcases*}
\frac{\lambda_f(|n|)}{\sqrt{|n|}} & if $\alpha = n \in \Z \setminus \{0\}$,	\\
0 & otherwise.
\end{dcases*}\]
In \hyperref[lem:archimedeanWhittaker]{Lemma \ref*{lem:archimedeanWhittaker}}, we show that for $n \in \Z \setminus \{0\}$ and $y > 0$,
\[W_{\infty}\begin{pmatrix} ny & 0 \\ 0 & 1 \end{pmatrix} = \begin{dcases*}
(-1)^{\frac{k}{2}} W_{\frac{k}{2},it}(4\pi ny) & for $n \in \N$ and $k \in 2\N \cup \{0\}$,	\\
\epsilon_f \frac{\Gamma\left(\frac{k + 1}{2} + it\right) \Gamma\left(\frac{k + 1}{2} - it\right)}{\Gamma\left(\frac{1}{2} + it\right) \Gamma\left(\frac{1}{2} - it\right)} W_{-\frac{k}{2},it}(4\pi |n|y) & for $n \in -\N$ and $k \in 2\N \cup \{0\}$,	\\
\frac{\Gamma\left(\frac{1 - k}{2} + it\right) \Gamma\left(\frac{1 - k}{2} - it\right)}{\Gamma\left(\frac{1}{2} + it\right) \Gamma\left(\frac{1}{2} - it\right)} W_{\frac{k}{2},it}(4\pi ny) & for $n \in \N$ and $k \in -2\N$,	\\
\epsilon_f (-1)^{\frac{k}{2}} W_{-\frac{k}{2},it}(4\pi |n|y) & for $n \in -\N$ and $k \in -2\N$.
\end{dcases*}\]
From this, we see that the constant $c_{\phi}$ in \eqref{eqn:Whittakerfact} is equal to $\rho_f(1)$ by taking $g = g_{\infty} = \begin{psmallmatrix} y & x \\ 0 & 1 \end{psmallmatrix} \in \GL_2^+(\R) \subset \GL_2(\A_{\Q})$ with $y > 0$ and $x \in \R$ in \eqref{eqn:Whittakerexpansion}, so that by \eqref{eqn:phitof} and \eqref{eqn:Whittakerinvariance},
\[F_k(x + iy) = \phi(g) = c_{\phi} \sum_{\alpha \in \Q^{\times}} W_{\infty}\begin{pmatrix} \alpha y & 0 \\ 0 & 1 \end{pmatrix} \prod_p W_p\begin{pmatrix} \alpha & 0 \\ 0 & 1 \end{pmatrix} e(\alpha x),\]
and comparing this ad\`{e}lic Whittaker expansion to the classical Fourier expansion at the cusp at infinity \eqref{eqn:Fkexppos} and \eqref{eqn:Fkexpneg}.

\subsubsection{Nonarchimedean Whittaker Functions}

Let $\phi = \phi_{F_k}$ be the ad\`{e}lic lift of a Maa\ss{} cusp form $F_k \in \Cscr_k(\Gamma_0(q))$ associated to a Hecke--Maa\ss{} newform $f \in \Cscr_0(\Gamma_0(q))$ as in \eqref{eqn:Fk} with $q$ squarefree. For each prime $p$, the local Whittaker functions $W_p \in \WW(\pi_p,\psi_p)$ are of a distinguished form. One can explicitly describe the values of the Whittaker function $W_p(g_p)$ for $g_p = \begin{psmallmatrix} a & 0 \\ 0 & 1 \end{psmallmatrix}$ with $a \in \Q_p^{\times}$; see \cite[Section 2.4]{Sch02}.
\begin{lemma}
\label{lem:nonarchimedeanWhittaker}
\hspace{1em}
\begin{enumerate}[leftmargin=*,label=\textup{(\arabic*)}]
\item For $p \nmid q$, the representation $\pi_p$ is a spherical principal series representation $\omega_p \boxplus \omega_p^{-1}$, where $\omega_p$ is an unramified character of $\Q_p^{\times}$ satisfying $p^{-\frac{7}{64}} \leq |\omega_p(q)|_p \leq p^{\frac{7}{64}}$ and $|\cdot|_p$ denotes the $p$-adic absolute value normalised such that $|p|_p = p^{-1}$. This character is such that $\omega_p(q)$ is equal to $\alpha_f(q)$ as in \eqref{eqn:HeckeSatakeunram}. For $a \in \Q_p^{\times}$, let $v(a) \in \Z$ be such that $|a|_p = p^{-v(a)}$. There is a distinguished Whittaker function, the spherical Whittaker function, that is right $\GL_2(\Z_p)$-invariant and satisfies
\[W_p\begin{pmatrix} a & 0 \\ 0 & 1 \end{pmatrix} = \begin{dcases*}
\sum_{m = 0}^{v(a)} \omega_p(q)^{m} \omega_p^{-1}(q)^{v(a) - m} |a|_p^{\frac{1}{2}} & if $0 < |a|_p \leq 1$, so that $v(a) \geq 0$,	\\
0 & if $|a|_p \geq p$, so that $v(a) \leq -1$.
\end{dcases*}\]
\item For $p \mid q$, the representation $\pi_p$ is a special representation $\omega_p \St_p$, where $\omega_p$ is an unramified unitary character of $\Q_p^{\times}$, so that $\omega_p(q) \in \{1,-1\}$. This character is such that $\omega_p(q)$ is equal to $\alpha_f(q)$ as in \eqref{eqn:HeckeSatakeram}. There is a distinguished Whittaker function, the Whittaker newform, that is right-invariant under the congruence subgroup of $\GL_2(\Z_p)$ consisting of elements $\begin{psmallmatrix} a & b \\ c & d \end{psmallmatrix}$ for which $c \in p\Z_p$ and satisfies
\[W_p\begin{pmatrix} a & 0 \\ 0 & 1 \end{pmatrix} = \begin{dcases*}
\omega_p(a) |a|_p & if $0 < |a|_p \leq 1$, so that $v(a) \geq 0$,	\\
0 & if $|a|_p \geq p$, so that $v(a) \leq -1$.
\end{dcases*}\]
\end{enumerate}
\end{lemma}

\subsubsection{Archimedean Whittaker Functions}

Let $\phi = \phi_{F_k}$ be the ad\`{e}lic lift of a Maa\ss{} cusp form $F_k \in \Cscr_k(\Gamma_0(q))$ associated to a Hecke--Maa\ss{} newform $f \in \Cscr_0(\Gamma_0(q))$ as in \eqref{eqn:Fk}. The local Whittaker function $W_{\infty} = W_{\infty}^{k} \in \WW(\pi_{\infty},\psi_{\infty})$ is again of a distinguished form. Since $f$ has weight $0$, the representation $\pi_{\infty}$ is a principal series representation of the form $\sgn^{\kappa} |\cdot|_{\infty}^{it} \boxplus \sgn^{\kappa} |\cdot|_{\infty}^{-it}$ with $\kappa = \kappa_f \in \{0,1\}$ and $t = t_f \in \R \cup i[-\frac{7}{64},\frac{7}{64}]$ such that $(-1)^{\kappa_f} = \epsilon_f$ is the parity of $f$ and $t_f$ is the spectral parameter of $f$, so that $\frac{1}{4} + t_f^2 = \lambda_f$ is the Laplacian eigenvalue of $f$. Here $|\cdot|_{\infty} = |\cdot|$ is the usual archimedean absolute value on $\R$.

The following claims are essentially implicit (albeit with some typographical errors) in the seminal work of Jacquet and Langlands \cite[Section 2.5]{JL70}, as further detailed by Godement \cite[Sections 2.3--2.6]{God18}; see also \cite{Pop08}. For the sake of completeness, we give explicit proofs.

\begin{lemma}
\label{lem:archimedeanWhittaker}
\hspace{1em}
\begin{enumerate}[leftmargin=*,label=\textup{(\arabic*)}]
\item For each $k \in 2\Z$, there exists a distinguished Whittaker function $W_{\infty}^{k} \in \WW(\pi_{\infty},\psi_{\infty})$, where $\pi_{\infty} = \sgn^{\kappa} |\cdot|_{\infty}^{it} \boxplus \sgn^{\kappa} |\cdot|_{\infty}^{-it}$, that is of weight $k$, so that for all $g_{\infty} \in \GL_2(\R)$ and $\theta \in [0,2\pi)$, it satisfies
\[W_{\infty}^{k}\left(g_{\infty} \begin{pmatrix} \cos \theta & \sin \theta \\ -\sin \theta & \cos \theta\end{pmatrix} \right) = e^{ik\theta} W_{\infty}^{k}(g_{\infty}).\]
\item For all $a \in \R^{\times}$, this distinguished Whittaker function satisfies
\begin{equation}
\label{eqn:Whittakeridentity}
W_{\infty}^{k}\begin{pmatrix} a & 0 \\ 0 & 1 \end{pmatrix} = \begin{dcases*}
(-1)^{\frac{k}{2}} W_{\frac{k}{2},it}(4\pi a) & for $a > 0$ and $k \in 2\N \cup \{0\}$,	\\
(-1)^{\kappa} \frac{\Gamma\left(\frac{k + 1}{2} + it\right) \Gamma\left(\frac{k + 1}{2} - it\right)}{\Gamma\left(\frac{1}{2} + it\right) \Gamma\left(\frac{1}{2} - it\right)} W_{-\frac{k}{2},it}(4\pi |a|) & for $a < 0$ and $k \in 2\N \cup \{0\}$,	\\
\frac{\Gamma\left(\frac{1 - k}{2} + it\right) \Gamma\left(\frac{1 - k}{2} - it\right)}{\Gamma\left(\frac{1}{2} + it\right) \Gamma\left(\frac{1}{2} - it\right)} W_{\frac{k}{2},it}(4\pi a) & for $a > 0$ and $k \in -2\N$,	\\
(-1)^{\kappa + \frac{k}{2}} W_{-\frac{k}{2},it}(4\pi |a|) & for $a < 0$ and $k \in -2\N$.
\end{dcases*}
\end{equation}
\item For $\kappa' \in \{0,1\}$ and $\Re(s) \geq \frac{1}{2}$, we have that
\begin{multline}
\label{eqn:W2int}
\int_{\R^{\times}} W_{\infty}^{2} \begin{pmatrix} a & 0 \\ 0 & 1 \end{pmatrix} \sgn^{\kappa'}(a) |a|^{s - \frac{1}{2}} \, d^{\times}a	\\
= \begin{dcases*}
- 2 \pi^{-s} \Gamma\left(\frac{s + 1 + it}{2}\right) \Gamma\left(\frac{s + 1 - it}{2}\right) & if $\kappa \equiv \kappa' + 1 \pmod{2}$,	\\
\left(\frac{1}{2} - s\right) \pi^{-s} \Gamma\left(\frac{s + it}{2}\right) \Gamma\left(\frac{s - it}{2}\right) & if $\kappa \equiv \kappa' \pmod{2}$,
\end{dcases*}
\end{multline}
where $d^{\times}a \coloneqq |a|^{-1} \, da$ denotes the multiplicative Haar measure on $\R^{\times}$.
\item We have that
\begin{equation}
\label{eqn:W2norm}
\int_{\R^{\times}} \left|W_{\infty}^{2} \begin{pmatrix} a & 0 \\ 0 & 1 \end{pmatrix}\right|^2 \, d^{\times}a = \left(\frac{1}{4} + t^2\right) \Gamma\left(\frac{1}{2} + it\right) \Gamma\left(\frac{1}{2} - it\right).
\end{equation}
\end{enumerate}
\end{lemma}

\begin{proof}\hspace{1em}
\begin{enumerate}[leftmargin=*,label=\textup{(\arabic*)}]
\item Let $\pi_{\infty} = \sgn^{\kappa} |\cdot|_{\infty}^{t_1} \boxplus \sgn^{\kappa} |\cdot|_{\infty}^{t_2}$ be a principal series representation with $\kappa \in \{0,1\}$ and $t_1,t_2 \in \C$. We initially assume that $\Re(t_1) > \Re(t_2)$. For each $k \in 2\Z$ and $(x_1,x_2) \in \R^2$, let
\begin{equation}
\label{eqn:Phik}
\Phi^k(x_1,x_2) \coloneqq (x_2 - \sgn(k) ix_1)^{|k|} e^{-\pi(x_1^2 + x_2^2)} = \begin{dcases*}
(x_2 - ix_1)^k e^{-\pi(x_1^2 + x_2^2)} & if $k \in 2\N \cup \{0\}$,	\\
(x_2 + ix_1)^{-k} e^{-\pi(x_1^2 + x_2^2)} & if $k \in -2\N$.
\end{dcases*}
\end{equation}
Define the Godement section $\varphi_{\infty}^{k} : \GL_2(\R) \to \C$ by
\[\varphi_{\infty}^{k}(g_{\infty}) \coloneqq \pi^{\frac{|k|}{2}} \sgn^{\kappa}(\det g_{\infty}) \left|\det g_{\infty}\right|^{t_1 + \frac{1}{2}} \int_{\R^{\times}} |y|^{t_2 - t_1 - 1} \Phi^k((0,y^{-1}) g_{\infty}) \, d^{\times}y.\]
This Godement section converges absolutely and defines an element of the induced model of $\pi_{\infty}$ of weight $k$; that is,
\begin{align}
\label{eqn:varphiinftykP}
\varphi_{\infty}^{k}\left(\begin{pmatrix} a & b \\ 0 & d \end{pmatrix} g_{\infty}\right) & = \sgn^{\kappa}(a) |a|^{t_1 + \frac{1}{2}} \sgn^{\kappa}(d) |d|^{t_2 - \frac{1}{2}} \varphi_{\infty}^{k}(g_{\infty}),	\\
\label{eqn:varphiinftykK}
\varphi_{\infty}^{k}\left(g_{\infty} \begin{pmatrix} \cos \theta & \sin \theta \\ - \sin \theta & \cos \theta \end{pmatrix}\right) & = e^{ik\theta} \varphi_{\infty}^{k}(g_{\infty})
\end{align}
for all $a,d \in \R^{\times}$, $b \in \R$, $g_{\infty} \in \GL_2(\R)$, and $\theta \in [0,2\pi)$. Taking $g_{\infty}$ to be the identity, we have the normalisation
\begin{equation}
\label{eqn:varphiinftykI}
\varphi_{\infty}^{k}\begin{pmatrix} 1 & 0 \\ 0 & 1 \end{pmatrix} = \pi^{-\frac{1 + t_1 - t_2}{2}} \Gamma\left(\frac{1 + |k| + t_1 - t_2}{2}\right).
\end{equation}
The Jacquet integral
\begin{align}
\label{eqn:Jacquetintinduced}
W_{\infty}^{k}(g_{\infty}) & \coloneqq \int_{\R} \varphi_{\infty}^{k}\left(\begin{pmatrix} 0 & -1 \\ 1 & 0 \end{pmatrix} \begin{pmatrix} 1 & u \\ 0 & 1 \end{pmatrix} g_{\infty}\right) e(-u) \, du	\\
\label{eqn:Jacquetint}
& = \pi^{\frac{|k|}{2}} \sgn^{\kappa}(\det g_{\infty}) \left|\det g_{\infty}\right|^{t_1 + \frac{1}{2}} \int_{\R^{\times}} |y|^{t_2 - t_1} \int_{\R} \Phi^k((y^{-1},u) g_{\infty}) e^{-2\pi i yu} \, du \, d^{\times}y
\end{align}
converges absolutely and defines an element of the Whittaker model $\WW(\pi_{\infty},\psi_{\infty})$ of $\pi_{\infty}$ of weight $k$. By the Iwasawa decomposition, every $g_{\infty} \in \GL_2(\R)$ can be written in the form
\[g_{\infty} = \begin{pmatrix} z & 0 \\ 0 & z \end{pmatrix} \begin{pmatrix} 1 & x \\ 0 & 1 \end{pmatrix} \begin{pmatrix} a & 0 \\ 0 & 1 \end{pmatrix} \begin{pmatrix} \cos \theta & \sin \theta \\ - \sin \theta & \cos \theta \end{pmatrix}\]
with $x \in \R$, $a,z \in \R^{\times}$, and $\theta \in [0,2\pi)$. From \eqref{eqn:varphiinftykP}, \eqref{eqn:varphiinftykK}, and the change of variables $u \mapsto u - x$, we see that for this value of $g_{\infty}$,
\[W_{\infty}^{k}(g_{\infty}) = |z|^{t_1 + t_2} e(x) e^{ik\theta} W_{\infty}^{k}\begin{pmatrix} a & 0 \\ 0 & 1 \end{pmatrix}.\]
The Whittaker function $W_{\infty}^{k}$ extends holomorphically as a function of the complex variables $t_1,t_2 \in \C$ to $(t_1,t_2) = (it,-it)$ with $t \in \R \cup i[-\frac{7}{64},\frac{7}{64}]$. This holomorphic extension defines a weight $k$ element of the Whittaker model $\WW(\pi_{\infty},\psi_{\infty})$ of $\pi_{\infty} = \sgn^{\kappa} |\cdot|_{\infty}^{it} \boxplus \sgn^{\kappa} |\cdot|_{\infty}^{-it}$.

\item We again let $\pi_{\infty} = \sgn^{\kappa} |\cdot|_{\infty}^{t_1} \boxplus \sgn^{\kappa} |\cdot|_{\infty}^{t_2}$ and initially assume that $\Re(t_1) > \Re(t_2)$. Since
\[\begin{pmatrix} 0 & -1 \\ 1 & 0 \end{pmatrix} \begin{pmatrix} 1 & u \\ 0 & 1 \end{pmatrix} \begin{pmatrix} a & 0 \\ 0 & 1 \end{pmatrix} = \begin{pmatrix} \frac{1}{\sqrt{1 + \frac{u^2}{a^2}}} & -\frac{u}{a\sqrt{1 + \frac{u^2}{a^2}}} \\ 0 & a\sqrt{1 + \frac{u^2}{a^2}} \end{pmatrix} \begin{pmatrix} \frac{u}{a\sqrt{1 + \frac{u^2}{a^2}}} & -\frac{1}{\sqrt{1 + \frac{u^2}{a^2}}} \\ \frac{1}{\sqrt{1 + \frac{u^2}{a^2}}} & \frac{u}{a\sqrt{1 + \frac{u^2}{a^2}}} \end{pmatrix},\]
we have by \eqref{eqn:varphiinftykP}, \eqref{eqn:varphiinftykK}, \eqref{eqn:varphiinftykI}, \eqref{eqn:Jacquetintinduced}, and the change of variables $u \mapsto au$ that
\begin{multline*}
W_{\infty}^{k}\begin{pmatrix} a & 0 \\ 0 & 1 \end{pmatrix} = (-1)^{\frac{k}{2}} \pi^{-\frac{1 + t_1 - t_2}{2}} \Gamma\left(\frac{1 + |k| + t_1 - t_2}{2}\right) \sgn(a)^{\kappa} |a|^{\frac{1}{2} + t_2}	\\
\times \int_{\R} (1 + iu)^{-\frac{1 - k + t_1 - t_2}{2}} (1 - iu)^{-\frac{1 + k + t_1 - t_2}{2}} e(-au) \, du.
\end{multline*}
From \cite[3.384.9]{GR15}, this integral is equal to
\[\frac{\pi^{\frac{1 + t_1 - t_2}{2}} |a|^{-\frac{1 - t_1 + t_2}{2}}}{\Gamma\left(\frac{1 + \sgn(a)k + t_1 - t_2}{2}\right)} W_{\sgn(a)\frac{k}{2},\frac{t_1 - t_2}{2}}(4\pi|a|).\]
Analytically continuing to $t_1 = it$ and $t_2 = -it$, we obtain \eqref{eqn:Whittakeridentity}.

\item For $a \in \R^{\times}$, we have that
\begin{equation}
\label{eqn:Winftyk}
W_{\infty}^{k}\begin{pmatrix} a & 0 \\ 0 & 1 \end{pmatrix} = \pi^{\frac{|k|}{2}} \sgn^{\kappa}(a) |a|^{it + \frac{1}{2}} \int_{\R^{\times}} |y|^{-2it} \int_{\R} \Phi^k(y^{-1} a,u) e^{-2\pi i yu} \, du \, d^{\times}y
\end{equation}
from \eqref{eqn:Jacquetint}. We insert the identity \eqref{eqn:Winftyk} for $W_{\infty}^{2}$ into the left-hand side of \eqref{eqn:W2int}, make the change of variables $a \mapsto ya$ and $u \mapsto u + ia$, then shift the contour of integration back to the line $\Im(u) = 0$, yielding
\[\pi \int_{\R^{\times}} \sgn^{\kappa + \kappa'}(a) |a|^{s + it} \int_{\R^{\times}} \sgn^{\kappa + \kappa'}(y) |y|^{s - it} e^{2\pi ya} \int_{\R} u^2 e^{-\pi u^2} e^{-2\pi i (y + a)u} \, du \, d^{\times}y \, d^{\times}a.\]
The innermost integral may be evaluated via integration by parts, leading to
\begin{multline*}
\frac{1}{2} \int_{\R^{\times}} \sgn^{\kappa + \kappa'}(a) |a|^{s + it} e^{-\pi a^2} \, d^{\times}a \int_{\R^{\times}} \sgn^{\kappa + \kappa'}(y) |y|^{s - it} e^{-\pi y^2} \, d^{\times}y	\\
- \pi \int_{\R^{\times}} \sgn^{\kappa + \kappa'}(a) |a|^{s + it} e^{-\pi a^2} \, d^{\times}a \int_{\R^{\times}} \sgn^{\kappa + \kappa'}(y) |y|^{s + 2 - it} e^{-\pi y^2} \, d^{\times}y	\\
- \pi \int_{\R^{\times}} \sgn^{\kappa + \kappa'}(a) |a|^{s + 2 + it} e^{-\pi a^2} \, d^{\times}a \int_{\R^{\times}} \sgn^{\kappa + \kappa'}(y) |y|^{s - it} e^{-\pi y^2} \, d^{\times}y	\\
- 2\pi \int_{\R^{\times}} \sgn^{\kappa + \kappa' + 1}(a) |a|^{s + 1 + it} e^{-\pi a^2} \, d^{\times}a \int_{\R^{\times}} \sgn^{\kappa + \kappa' + 1}(y) |y|^{s + 1 - it} e^{-\pi y^2} \, d^{\times}y.
\end{multline*}
The first three expressions vanish if $\kappa \equiv \kappa' + 1 \pmod{2}$, while the last vanishes if $\kappa \equiv \kappa' \pmod{2}$. The result then follows via the recurrence relation $\Gamma(s + 1) = s\Gamma(s)$.

\item The left-hand side of \eqref{eqn:W2norm} is
\[\int_{0}^{\infty} \left(W_{1,it}(4\pi a)^2 + \left(\frac{1}{4} + t^2\right)^2 W_{-1,it}(4\pi a)^2\right) \, \frac{da}{a}\]
from \eqref{eqn:Whittakeridentity}. The desired identity then follows from the change of variables $a \mapsto \frac{a}{4\pi}$ together with \cite[7.611.4, 8.365.1, and 8.365.8]{GR15}.
\qedhere
\end{enumerate}
\end{proof}

\section{The Boundary of \texorpdfstring{$\FF_A(q)$}{FA(q)}}
\label{sect:boundary}

We turn our attention to the boundary of the canonical fundamental domain $\FF_A(q)$ of the hyperbolic orbifold $\Gamma_A(q)\backslash \NN_A(q)$. We give an explicit description of the boundary in terms of the homology of $X_0(q)$, which relies crucially on the fact that $\PP(q)$ has a minimal number of sides. For related arguments in the case of a general fundamental polygon, see \cite[Section 3]{NT25}. Furthermore, we introduce the notion of regularised integrals of Eisenstein series along geodesics. These two concepts are key in understanding the \emph{topological terms} in the Weyl sums for our equidistribution problem.

Recall that the boundary of $\FF_A(q)$ consists of two parts: the closed geodesic $\CC_A(q)$ and certain $\Gamma_0(q)$-translates of the sides of $\PP(q)$. When considering the boundary of $\FF_A(q)$ modulo $\Gamma_0(q)$, the sides of $\PP(q)$ that are paired together are indistinguishable. With this in mind, we let
\begin{equation}
\label{eqn:specialbasis}
\CC_1,\ldots, \CC_{2g+1+e_2+e_3}\subset \partial \PP(q)
\end{equation} 
be a sequence of edges of $\PP(q)$ such that for each pair of $\Gamma_0(q)$-equivalent edges $\LL,\LL'$ of $\PP(q)$ with $\LL$ to the left of $\LL'$, we have $\CC_i=\LL$ for some $1\leq i \leq 2g+1+e_2+e_3$ and such that $\CC_1,\ldots, \CC_{2g}$ are hyperbolic sides. Note that $\CC_i$ for $1\leq i\leq 2g$ define closed curves on $X_0(q)$ since the endpoints are all $\Gamma_0(q)$-equivalent to the cusp $0$.

\subsection{Homology of Modular Curves}

Let $q$ be an odd prime. The closed modular surface $X_0(q)$ of level $q$ is a compact Riemann surface of genus $g=\tfrac{q}{12}+O(1)$. The associated integral singular homology and cohomology groups $H_1(X_0(q),\Z)$ and $H^1(X_0(q),\Z)$ are free abelian groups of rank $2g$ that sit as lattices inside $H_1(X_0(q),\R)$ and $H^1(X_0(q),\R)$ respectively. By general principles, we have the cap product pairing between homology and cohomology
\begin{equation}
\label{eqn:cappair}
\langle \cdot , \cdot \rangle_\mathrm{cap}:H_1(X_0(q),\R) \times H^1(X_0(q),\R) \to \R,
\end{equation}
which is a perfect pairing \cite[Section 3.3]{Hat02}. The cap product pairing identifies the homology group with the (linear) dual of the cohomology group and vice versa. Furthermore, the Hecke operators as well as the Atkin--Lehner operator $W_q$ are self-adjoint with respect to the cap product pairing (see e.g.\ \cite[Section 3.2.1]{Nor23} for explicit formul\ae{} for the action of Hecke and Atkin--Lehner operators on homology and cohomology).

Given a closed curve $\CC$ inside $X_0(q)$, we denote by
\[[\CC]\in H_1(X_0(q),\Z)\subset H_1(X_0(q), \R)\]
the associated homology class. Recall that $H_1(X_0(q),\Z)$ is generated as an abelian group by such classes \cite[Chapter 2]{Hat02}. If $h$ is a holomorphic cusp form of weight $2$ and level $q$, then $h(z) \, dz$ and $\overline{h(z) \, dz}$ define complex-valued harmonic $1$-forms on $X_0(q)$ and such $1$-forms span the entire space of harmonic $1$-forms. Given any harmonic $1$-form $\omega$, we get an associated (complex-valued) cohomology class:
\[
[\CC]\mapsto \int_{\CC} \omega,
\]
for all closed curves $\CC$. By theorems of de Rham and Hodge, this association yields an isomorphism between the real cohomology group $H^1(X_0(q), \R)$ and the space of real-valued harmonic $1$-forms.

We consider two bases for the homology and cohomology groups.

\subsubsection{The Hecke Basis}
\label{sect:Heckebasis}

Let $\BB_2^{\hol}(\Gamma_0(q))$ be the orthogonal basis of weight $2$ and level $q$ holomorphic cusp forms of level $q$ consisting of Hecke eigenforms normalised so that the first Fourier coefficient of each cusp form is $1$; note that these are all newforms since there are no cusp forms of weight $2$ and level $1$. For each $h\in \BB_2^{\hol}(\Gamma_0(q))$,
\begin{equation}\label{eqn:omegafdefeq}\frac{1}{2} i^{\frac{1 \mp 1}{2}} \left(h(z) \, dz\pm \overline{h(z) \, dz}\right)\end{equation}
defines a real-valued closed $1$-form. We denote by $\omega_h^\pm\in H^1(X_0(q),\R)$ the associated cohomology class. It follows by dimensional considerations that
\begin{equation}
\label{eqn:basiscoh}
\{\omega_h^{\epsilon} : h \in \BB_2^{\hol}(\Gamma_0(q)), \ \epsilon \in \{+,-\} \}
\end{equation}
is a basis for $H^1(X_0(q),\R)$. We denote by 
\[\{v_h^{\epsilon} : h \in \BB_2^{\hol}(\Gamma_0(q)), \ \epsilon \in \{+,-\} \}\]
the basis of $H_1(X_0(q),\R)$ that is dual to \eqref{eqn:basiscoh} with respect to the cap product pairing \eqref{eqn:cappair}; by this, we mean that for $h_1,h_2\in \BB_2^{\hol}(\Gamma_0(q)) $ and $\epsilon_1,\epsilon_2\in \{+,-\}$,
\[\langle v_{h_1}^{\epsilon_1}, \omega_{h_2}^{\epsilon_2} \rangle_\mathrm{cap} = \begin{dcases*}
1 & if $h_1 = h_2$ and $\epsilon_1 = \epsilon_2$,	\\
0 & otherwise.
\end{dcases*}\]

\subsubsection{The Special Basis}

By general principles (since $\Hb$ is contractible), we have a sequence of surjective maps
\begin{equation}
\label{eqn:maps}
\Gamma_0(q)\twoheadrightarrow \mathrm{Conj}(\Gamma_0(q))\twoheadrightarrow \Gamma_0(q)^\mathrm{ab}\twoheadrightarrow H_1(Y_0(q), \Z)\twoheadrightarrow H_1(X_0(q), \Z),
\end{equation} 
where the composition $\Gamma_0(q)\twoheadrightarrow H_1(X_0(q), \Z)$ is given by
\[\gamma\mapsto [\CC_\gamma],\]
with $\CC_\gamma$ any piecewise geodesic curve connecting $z$ with $\gamma z$, where $z\in \Hb\cup \Q\cup \{\infty\}$ (as the homology class is independent of these choices). The surjectivity of the first two maps is evident and the composition of all of the maps can be thought of as the map from the orbifold fundamental group $\pi_1^\mathrm{orb}(Y_0(p))\cong \Gamma_0(p)$ (using that $\Hb$ is the universal orbifold covering space of the orbifold $Y_0(p)$) to the orbifold homology group $H_1^\mathrm{orb}(Y_0(p),\Z)\cong\Gamma_0(p)^\mathrm{ab}$ composed with the surjective map from $H_1^\mathrm{orb}(Y_0(p),\Z)$ to the homology of the underlying topological space of $X_0(p)$ (see e.g.\ \cite[Section 2.2]{Car19}). Now one sees directly that the kernel of the last map contains images of the parabolic and elliptic conjugacy classes in $\Gamma_0(q)$ and furthermore by dimension considerations that these classes even generate the kernel. Thus we conclude that the images of the hyperbolic labels $\alpha_j$ of $\PP(q)$ under \eqref{eqn:maps} generate the lattice $H_1(X_0(q), \Z)$ and so in particular, span $H_1(X_0(q), \R)$.

\begin{lemma}\label{lem:basisCi}
The cohomology classes $[\CC_1],\ldots, [\CC_{2g}]$ define an integral basis for $H_1(X_0(q),\Z)$. 
\end{lemma}

\begin{proof}
We start by noticing that the homology class associated to a hyperbolic label $\alpha_j$ with $j>j^\ast$ can be written as
\begin{equation}
\label{eqn:alphaLLsum}
\sum_{j^\ast<k<j}[\LL_k],
\end{equation}
where $[\LL_k]$ denotes the homology class of the side $\LL_k$ of $\mathcal{P}(q)$. To see this, observe that the left-most vertex of the side $\LL_{j}$ is equal to $\alpha_j v$, where $v$ is the right-most vertex of $\LL_{j^\ast}$. Thus if $j^\ast<j$ we see that the concatenation of the sides
\[\LL_{j^{\ast}+1},\LL_{j^{\ast}+2}\ldots, \LL_{j-1}\]
defines a curve connecting $v$ and $\alpha_j v$, which gives the claimed expression \eqref{eqn:alphaLLsum}.

This shows that the homology class associated to a hyperbolic label $\alpha_j$ lies in the $\Z$-span of $\{[\CC_i]: 1\leq i\leq 2g+1+e_2+e_3\}$. Since the classes associated to the hyperbolic labels via the map \eqref{eqn:maps} generate the lattice $H_1(X_0(q),\Z)$ and the classes of parabolic and elliptic sides vanish, we get the desired conclusion by dimension considerations.
\end{proof}
We denote by
\begin{equation}
\label{eqn:specialCoh}
\omega_1,\ldots, \omega_{2g}\in H^1(X_0(q), \Z)\subset H^1(X_0(q),\R),
\end{equation}
the dual basis of $[\CC_1], \ldots, [\CC_{2g}]$ with respect to the cap product pairing. We refer to these bases as the \emph{special basis} of homology and cohomology respectively.

\subsection{A Homological Description of the Boundary}
\label{sect:homdes}

In studying the projection of $\FF_A(q)$ to the closed modular surface $X_0(q)$, we need to understand the boundary of this projection. Unlike the level $1$ case, this is not simply the geodesic $\CC_A(q)$, for $\CC_A(q)$ may be nontrivial in homology (which corresponds to the fact that $\Gamma_A(q)$ need not be contained in $\Gamma_0(q)$).

First of all, denote by $\partial \FF_A(q)$ the image of the oriented boundary of $\FF_A(q)$ under the projection to $X_0(q)$. Formally, this is nothing but a singular $1$-chain (as in the theory of singular homology \cite[Chapter 2]{Hat02}), which is to say an element of
\[\Z[\{ \varphi: [0,1]\to X_0(q): \text{continuous} \}].\]
We write
\[\partial \FF_A(q)= \CC_A(q)\cup \bigcup_{i=1}^{2g+1+e_2+e_3} m_i(A,q)\CC_i,\]
where $m_i(A,q)\in \Z$ is the multiplicity of the edge $\CC_i$ in $\FF_A(q)$, namely the signed number of times the boundary of $\FF_A(q)$ contains ($\Gamma_0(q)$-translates of) the two oriented edges of $\PP(q)$ corresponding to $\CC_i$. Explicitly, this means that for a compactly supported $1$-form $\omega$ on $X_0(q)$,
\[\int_{\partial \FF_A(q)} \omega= \int_{\CC_A(q)} \omega+ \sum_{i=1}^{2g+1+e_2+e_3} m_i(A,q)\int_{\CC_i} \omega.\]

\begin{lemma}[Cf.\ {\cite[Lemma 3.1]{NT25}}]
\label{lem:multiplicity}
Let $\CC_i$ be an elliptic or parabolic edge of $\PP(q)$. Then we have that
\[m_i(A,q)=0.\]
\end{lemma}

\begin{proof}
Consider the boundary of the special fundamental polygon $\PP(q)$ as a (cyclic) graph with $4g+2+2e_2+2e_3$ vertices with its natural action of $\Gamma_0(q)$, and denote by 
\[\GG(q) \coloneqq \partial\PP(q)/\Gamma_0(q)\]
the quotient graph. The graph $\GG(q)$ has $2 + e_2 + e_3$ vertices and $2g+1+e_2+e_3$ edges. The vertices corresponding to the cusp $\infty$ and the elliptic points all have a single edge connected to the vertex corresponding to the cusp $0$, and the vertex $0$ has $2g$ self-edges (i.e.\ the degree of the vertex $0$ is $4g + 1 + e_2 + e_3$). 

The oriented boundary of $\FF_A(q)$ can be described as a weighted version of $\GG(q)$; for an edge $E$ of $\GG(q)$ corresponding to $\CC_i$, we associated the weight $m_i(A,q) \in \Z$. The key observation is now that by construction these weights are all obtained by taking a cycle in the signed graph and recording the signed number of times the cycle crosses each edge. The claim now follows since the vertices of $\GG(q)$ corresponding to $\infty$ and elliptic points have degree $1$ (i.e.\ there is only one way to go to the elliptic points and to $\infty$).
\end{proof}

\hyperref[lem:multiplicity]{Lemma \ref*{lem:multiplicity}} shows that the oriented boundary of $\FF_A(q)$, when projected to $X_0(q)$, consists only of hyperbolic sides of $\PP(q)$. This means that we can rewrite the boundary solely in terms of the hyperbolic sides $\CC_1,\ldots, \CC_{2g}$ defined above:
\[\partial \FF_A(q)= \CC_A(q)\cup \bigcup_{i=1}^{2g} m_i(A,q)\CC_i.\]

\begin{lemma}
\label{lem:multiplicity2}
For $i \in \{1,\ldots,2g\}$, we have that
\[m_i(A,q) = -\langle [\CC_A(q)], \omega_i \rangle_\mathrm{cap},\]
where $\omega_1,\ldots, \omega_{2g}\in H^1(X_0(q),\Z)$ is the special basis of cohomology defined in \eqref{eqn:specialCoh}.
\end{lemma}

\begin{proof}
We observe that in homology, we have that
\[[\CC_A(q)]+\sum_{i=1}^{2g} m_i(A,q)[\CC_i]=[\partial \FF_A(q)]=0\in H_1(X_0(q),\Z),\]
since $\FF_A(q)\subset \Hb$ and $\Hb$ is contractible. Now the claim follows since the coefficient of the basis element $[\CC_i]$  in the expansion of $[\CC_A(q)]$ with respect to the special basis from \hyperref[lem:basisCi]{Lemma \ref*{lem:basisCi}} is exactly
\[\langle [\CC_A(q)], \omega_i \rangle_\mathrm{cap},\]
by the definition of the dual basis.
\end{proof}

\subsection{Regularised Integrals}

For a cusp $\bb\in \Pb^1(\Q)$ and $Y > 0$, we define the cuspidal zone
\begin{equation}
\label{eqn:cuspidalzone}
\FF_{\bb}(Y) \coloneqq \{z \in \Hb : 0 < \Re(\sigma_{\bb}^{-1} z) < 1, \ \Im(\sigma_{\bb}^{-1} z) \geq Y\},
\end{equation}
where $\sigma_{\bb} \in \Gamma$ is a scaling matrix for $\bb$, so that $\sigma_{\bb} \infty = \bb$; note that $\FF_{\bb}(Y)$ is independent of the choice of such a scaling matrix. One can readily check that for $\gamma \in \Gamma$, this satisfies
\begin{equation}
\label{eqn:cuspidalregion}
\gamma \FF_{\bb}(Y) = \FF_{\gamma \bb}(Y).
\end{equation}
Furthermore, for $\ell\in \N$ with corresponding matrix $g_{\ell} \coloneqq \begin{psmallmatrix} \sqrt{\ell} & 0 \\ 0 & \frac{1}{\sqrt{\ell}} \end{psmallmatrix} \in \SL_2(\R)$, we have that
\begin{equation}
\label{eqn:cuspidalregion2}
g_{\ell} \FF_{\bb}(Y)= \FF_{g_{\ell} \bb}\left(\tfrac{(\ell,\den(\bb))^2}{\ell}\, Y\right),
\end{equation}
where $\den(\bb)$ denotes the denominator of $\bb \in \Pb^1(\Q)$ written in reduced form, with the convention that $\den(0) \coloneqq 1$ and $\den(\infty) \coloneqq 0$. 

\begin{lemma}[{Cf.\ \cite[Appendix]{BH12}}]
\label{lem:regint}
Let $\CC\subset \Hb \cup \Pb^1(\Q)$ be a geodesic with endpoints $\bb_1,\bb_2 \in \Pb^1(\Q)$ with $\bb_1 \neq \bb_2$. Then the limit
\begin{equation}
\label{eqn:regint}
\int^\ast_{\CC} (R_0 E)\left(z,\frac{1}{2} + it\right) \, \frac{dz}{\Im(z)} \coloneqq \lim_{Y\to \infty} \int\limits_{\CC\setminus(\FF_{\bb_1}(Y)\cup \FF_{\bb_2}(Y))} (R_0 E)\left(z,\frac{1}{2} + it\right) \, \frac{dz}{\Im(z)}
\end{equation}
exists.
\end{lemma}

We may think of the limit \eqref{eqn:regint} as a regularised integral, akin to Zagier's regularisation of integrals on $\Gamma \backslash \Hb$ \cite{Zag82}.

\begin{proof}[Proof of {\hyperref[lem:regint]{Lemma \ref*{lem:regint}}}]
We let $\sigma_{\bb_1} \in \Gamma$ be a scaling matrix for $\bb_1$. Upon making the change of variables $z \mapsto \sigma_{\bb_1}^{-1} z$ and recalling \eqref{eqn:cuspidalregion}, we may reduce to the case that $\bb_1 = \infty$ and $\bb_2 = \frac{a}{b}$ with $a \in \Z$, $b \in \N$, and $(a,b)=1$. For sufficiently large $Y > 0$, we may write
\begin{multline*}
\int\limits_{\CC\setminus(\FF_{\infty}(Y)\cup \FF_{\bb_2}(Y))} (R_0 E)\left(z,\frac{1}{2} + it\right) \, \frac{dz}{\Im(z)}	\\
=\int_{\frac{a}{b} + \frac{i}{b}}^{\frac{a}{b} + iY} (R_0 E)\left(z,\frac{1}{2} + it\right) \, \frac{dz}{\Im(z)} + \int^{\frac{a}{b} + \frac{i}{b}}_{\frac{a}{b} + \frac{i}{b^2 Y}} (R_0 E)\left(z,\frac{1}{2} + it\right) \, \frac{dz}{\Im(z)}.
\end{multline*}
Let $\sigma_{\bb_2} \in \Gamma$ be a scaling matrix for $\bb_2$; in particular, we may choose $\sigma_{\bb_2} = \begin{psmallmatrix} 0 & -1 \\ 1 & 0 \end{psmallmatrix}$ if $a = 0$, while if $a \neq 1$, then there exists some $\overline{a} \in \Z$ such that $a \overline{a} \equiv 1 \pmod{b}$, in which case we may take $\sigma_{\bb_2} = \begin{psmallmatrix} a & \frac{a \overline{a} - 1}{b} & \\ b & \overline{a} \end{psmallmatrix}$. Making the change of variables $z \mapsto \sigma_{\bb_2}^{-1} z$ in the second integral above, we see that the right-hand side is equal to
\[\int_{\frac{a}{b} + \frac{i}{b}}^{\frac{a}{b} + iY} (R_0 E)\left(z,\frac{1}{2} + it\right) \, \frac{dz}{\Im(z)} - \int_{-\frac{\overline{a}}{b} + \frac{i}{b}}^{-\frac{\overline{a}}{b}+ iY} (R_0 E)\left(z,\frac{1}{2} + it\right) \, \frac{dz}{\Im(z)}\]
if $a \neq 0$, while the same holds for $a = 0$ with $\overline{a}$ replaced by $0$. Now we insert the Fourier expansion \eqref{eqn:FkexEis}. Since the constant terms cancel, we see that the limit exists as $Y$ tends to infinity due the rapid decay of the Whittaker function, namely $W_{\pm 1,it}(4\pi y) \ll y e^{-2\pi y}$ as $y$ tends to infinity.
\end{proof}

\begin{corollary}
Let $\ell\in \N$ be squarefree. Let $\CC$ be the geodesic connecting $\bb_1$ and $\bb_2$, where $\bb_1,\bb_2\in \Pb^1(\Q)$ are $\Gamma_0(\ell)$-equivalent. Then the limit
\begin{equation}
\label{eqn:regint2}
\lim_{Y\to \infty} \int\limits_{\CC\setminus(\FF_{\bb_1}(Y)\cup \FF_{\bb_2}(Y))} (R_0 E)\left(\ell z,\frac{1}{2} + it\right) \, \frac{dz}{\Im(z)}
\end{equation}
exists and is equal to
\[\int^\ast_{\CC} (R_0 E)\left(\ell z,\frac{1}{2} + it\right) \, \frac{dz}{\Im(z)}.\]
\end{corollary}

\begin{proof}
This follows directly from \hyperref[lem:regint]{Lemma \ref*{lem:regint}} combined with \eqref{eqn:cuspidalregion2} since $\frac{a_1}{b_1}, \frac{a_2}{b_2} \in \Pb^1(\Q)$ are $\Gamma_0(\ell)$-equivalent precisely when $(b_1,\ell)=(b_2,\ell)$.
\end{proof}

\section{Weyl Sums}
\label{sect:Weylsums}

\subsection{Weyl Sums for Newforms and \texorpdfstring{$L$}{L}-Functions}

Let $q$ be a positive squarefree integer for which every prime dividing $q$ splits in $E$, let $f \in \Cscr_0(\Gamma_0(q))$ be a Hecke--Maa\ss{} newform, and let $\chi$ be a narrow class character of $E$. Our goal is to relate the Weyl sum
\begin{equation}
\label{eqn:WeylMaass}
W_{\chi,f} \coloneqq \sum_{A \in \Cl_D^+} \chi(A) \int_{\FF_A(q)} f(z) \, d\mu(z)
\end{equation}
to a special value of the Rankin--Selberg $L$-function $L(s,f \otimes \Theta_{\chi})$, where $\Theta_{\chi}$ denotes the theta series associated to $\chi$, as in \cite[Appendix A.1]{HK20}, which is a newform of weight $0$, level $D$, nebentypus $\chi_D$, Laplacian eigenvalue $\lambda_{\Theta_{\chi}} = 1/4$, and parity $\epsilon_{\Theta_{\chi}} = \chi(J) \in \{1,-1\}$. The automorphic form $\Theta_{\chi}$ is a cusp form if and only if $\chi$ is complex; otherwise $\Theta_{\chi}$ is an Eisenstein series and $\chi$ is a genus character.

\begin{proposition}
\label{prop:MaassnewformWeyl}
Let $q$ be either $1$ or a prime that splits in $E$, let $f \in \Cscr_0(\Gamma_0(q))$ be a Hecke--Maa\ss{} newform normalised such that
\[\int_{\Gamma_0(q) \backslash \Hb} |f(z)|^2 \, d\mu(z) = 1,\]
and let $\chi$ be a narrow class character of $E$. Then there exist constants $c_{\chi,f} \in \C$ and $c_{\chi,h}^{\pm} \in \C$ for each $h \in \BB_2^{\hol}(\Gamma_0(q))$ such that
\begin{multline*}
W_{\chi,f} = c_{\chi,f} \frac{L\left(\frac{1}{2}, f \otimes \Theta_{\chi}\right)^{\frac{1}{2}}}{L(1,\ad f)^{\frac{1}{2}}} + \sum_{j=1}^{2g} \int_{\CC_j} (R_0 f)(z) \, \frac{dz}{\Im(z)} \sum_{h\in \BB_2^{\hol}(\Gamma_0(q))} \sum_{\pm} \langle v^\pm_{h}, \omega_j\rangle_\mathrm{cap}	\\
\times \left(c_{\chi,h}^+ L\left(\frac{1}{2}, h \otimes \Theta_{\overline{\chi}}\right)^{\frac{1}{2}} \pm c_{\chi,h}^- L\left(\frac{1}{2}, h \otimes \Theta_{\chi}\right)^{\frac{1}{2}} \right).
\end{multline*}
Here the sum over $j$ is empty if $q = 1$ (so that $g = 0$), $\CC_1,\ldots \CC_{2g}$ are the hyperbolic sides of $\PP(q)$ as in \eqref{eqn:specialbasis}, and $\omega_1,\ldots, \omega_{2g}$ is the special basis of cohomology as in \eqref{eqn:specialCoh}. Moreover, the constants $c_{\chi,f}$ and $c_{\chi,h}^{\pm}$ satisfy
\[|c_{\chi,f}|^2=\frac{(1 - \epsilon_f \chi(J)) \sqrt{D}}{q \left(\frac{1}{4} + t_f^2\right)^2} \frac{\Gamma\left(\frac{3}{4} + \frac{it_f}{2}\right)^2 \Gamma\left(\frac{3}{4} - \frac{it_f}{2}\right)^2}{\Gamma\left(\frac{1}{2} + it_f\right) \Gamma\left(\frac{1}{2} - it_f\right)},\qquad |c_{\chi,h}^\pm|^2=\frac{\sqrt{D}}{16\pi^2 \left(\frac{1}{4} + t_f^2\right)^2},\]
where $t_f \in \R \cup i[-\frac{7}{64},\frac{7}{64}]$ denotes the spectral parameter and $\epsilon_f \in \{1,-1\}$ denotes the parity of $f$.
\end{proposition}

\begin{remark}
In terms of the completed $L$-functions
\begin{align*}
\Lambda(s,f \otimes \Theta_{\chi}) & \coloneqq \pi^{-2(s + |\kappa_f - \kappa_{\chi}|)} \Gamma\left(\frac{s + |\kappa_f - \kappa_{\chi}| + it_f}{2}\right)^2 \Gamma\left(\frac{s + |\kappa_f - \kappa_{\chi}| - it_f}{2}\right)^2 L(s,f \otimes \Theta_{\chi}),	\\
\Lambda(s,\ad f) & \coloneqq \pi^{-\frac{3s}{2}} \Gamma\left(\frac{s}{2} + it_f\right) \Gamma\left(\frac{s}{2}\right) \Gamma\left(\frac{s}{2} - it_f\right) L(s,\ad f),
\end{align*}
where $\kappa_f, \kappa_{\chi} \in \{0,1\}$ are such that $(-1)^{\kappa_f} = \epsilon_f$ and $(-1)^{\kappa_{\chi}} = \epsilon_{\Theta_{\chi}} = \chi(J)$, we have that
\[|c_{\chi,f}|^2 \frac{L\left(\frac{1}{2}, f \otimes \Theta_{\chi}\right)}{L(1,\ad f)} = \frac{(1 - \epsilon_f \chi(J)) \pi^2 \sqrt{D}}{q \left(\frac{1}{4} + t_f^2\right)^2} \frac{\Lambda\left(\frac{1}{2},f \otimes \Theta_{\chi}\right)}{\Lambda(1,\ad f)}.\]
\end{remark}

\begin{remark}
Stirling's formula implies that
\begin{equation}
\label{eqn:Stirlingboundscchif}
c_{\chi,f} \ll D^{\frac{1}{4}} q^{-\frac{1}{2}} (1 + |t_f|)^{-\frac{3}{2}}, \qquad c_{\chi,h}^{\pm} \ll D^{\frac{1}{2}} (1 + |t_f|)^{-2}.
\end{equation}
\end{remark}

\begin{remark}
It is instructive to consider the case of $q = 1$ and $\chi$ a genus character associated to the pair of primitive quadratic Dirichlet characters $\chi_{D_1}$ and $\chi_{D_2}$ modulo $|D_1|$ and $|D_2|$ respectively, where $D_1$ and $D_2$ are fundamental discriminants for which $D_1 D_2 = D$. Then $\Theta_{\chi}$ is the Eisenstein newform associated to $\chi_{D_1}$ and $\chi_{D_2}$, as described in \cite{You19}, and so $L(s,f \otimes \Theta_{\chi}) = L(s,f \otimes \chi_{D_1}) L(s,f \otimes \chi_{D_2})$. Since $D > 1$, either $D_1,D_2 > 0$ or $D_1,D_2 < 0$; in the former case, we have that $\chi(J) = 1$, while $\chi(J) = -1$ in the latter case. \hyperref[prop:MaassnewformWeyl]{Proposition \ref*{prop:MaassnewformWeyl}} then gives the identity
\begin{equation}
\label{eqn:Weylsumgenus}
\left|W_{\chi,f}\right|^2 = \frac{(1 - \epsilon_f \chi(J)) \sqrt{D}}{q \left(\frac{1}{4} + t_f^2\right)^2} \frac{\Gamma\left(\frac{3}{4} + \frac{it_f}{2}\right)^2 \Gamma\left(\frac{3}{4} - \frac{it_f}{2}\right)^2}{\Gamma\left(\frac{1}{2} + it_f\right) \Gamma\left(\frac{1}{2} - it_f\right)} \frac{L\left(\frac{1}{2}, f \otimes \chi_{D_1}\right) L\left(\frac{1}{2}, f \otimes \chi_{D_2}\right)}{L(1,\ad f)}.
\end{equation}
We see that the Weyl sum $W_{\chi,f}$ vanishes if $f$ is even and $D_1,D_2 > 0$ or if $f$ is odd and $D_1,D_2 < 0$; additionally, $W_{\chi,f}$ vanishes if $f$ is odd and $D_1,D_2 > 0$, for then the root numbers of $f \otimes \chi_{D_1}$ and $f \otimes \chi_{D_2}$ are both equal to $-1$ \cite[Lemma A.2]{HK20}, and hence $L(s,f \otimes \chi_{D_1})$ and $L(s,f \otimes \chi_{D_2})$ both vanish at $s = 1/2$. These vanishing results and the identity \eqref{eqn:Weylsumgenus} when $f$ is even and $D_1,D_2 < 0$ are in exact accordance with the work of Duke, Imamo\={g}lu, and T\'{o}th \cite[Theorem 4 and (5.17)]{DIT16}.
\end{remark}

Similarly, letting $E(z,s)$ denote the Eisenstein series on $\Gamma \backslash \Hb$, which has parity $1$, we relate the Weyl sum
\begin{equation}
\label{eqn:WeylEis}
W_{\chi,t} \coloneqq \sum_{A \in \Cl_D^+} \chi(A) \int_{\FF_A} E\left(z,\frac{1}{2} + it\right) \, d\mu(z)
\end{equation}
to a special value of $L(s,\Theta_{\chi}) = L(s,\chi)$.

\begin{proposition}
\label{prop:EisnewformWeyl}
Let $\chi$ be a narrow class character of $E$. For $t \in \R$, we have that
\[W_{\chi,t} = \frac{(1 - \chi(J)) D^{\frac{1}{4} + \frac{it}{2}}}{\frac{1}{4} + t^2} \frac{\Gamma\left(\frac{3}{4} + \frac{it}{2}\right)^2}{\Gamma\left(\frac{1}{2} + it\right)} \frac{L\left(\frac{1}{2} + it,\Theta_{\chi}\right)}{\zeta(1 + 2it)}.\]
\end{proposition}

\begin{remark}
This is
\[\frac{(1 - \chi(J)) \pi D^{\frac{1}{4} + \frac{it}{2}}}{\frac{1}{4} + t^2} \frac{\Lambda\left(\frac{1}{2} + it,\Theta_{\chi}\right)}{\xi(1 + 2it)}\]
in terms of the completed $L$-functions
\[\Lambda(s,\Theta_{\chi}) \coloneqq \pi^{-s - 1} \Gamma\left(\frac{s + 1}{2}\right)^2 L(s,\Theta_{\chi}), \qquad \xi(s) \coloneqq \pi^{-\frac{s}{2}} \Gamma\left(\frac{s}{2}\right) \zeta(s).\]
\end{remark}

The proofs of \hyperref[prop:MaassnewformWeyl]{Propositions \ref*{prop:MaassnewformWeyl}} and \ref{prop:EisnewformWeyl} are given in \hyperref[sect:proofsWeyl]{Section \ref*{sect:proofsWeyl}}. Our method is to first prove identities relating certain ad\`{e}lic period integrals to $L$-functions, then show that these ad\`{e}lic period integrals are equal to integrals over closed geodesics, and finally relate these integrals over closed geodesics to integrals over hyperbolic orbifolds.

In \hyperref[sect:oldforms]{Section \ref*{sect:oldforms}}, we extend \hyperref[prop:MaassnewformWeyl]{Propositions \ref*{prop:MaassnewformWeyl}} and \ref{prop:EisnewformWeyl} to oldforms; note that if $q > 1$ is squarefree, every Eisenstein series is an oldform \cite{You19}.

\subsection{Ad\`{e}lic Period Integrals and Choices of Test Vectors}

The first step towards proving \hyperref[prop:MaassnewformWeyl]{Proposition \ref*{prop:MaassnewformWeyl}} is to apply a formula due to Martin and Whitehouse \cite{MW09}, extending work of Waldspurger \cite{Wal85}, relating certain ad\`{e}lic period integrals to a ratio of special values of $L$-functions. We now describe in some detail how the results of Martin and Whitehouse apply in our specific case.

\subsubsection{Waldspurger-Type Formul\ae{}}

Following the pioneering work of Waldspurger \cite{Wal85}, there has been considerable work in obtaining explicit formul\ae{} relating ad\`{e}lic period integrals and central values of Rankin--Selberg $L$-functions, as in work of Gross \cite{Gro88}, Zhang \cite{Zha01}, Jacquet and Chen \cite{JC01}, Martin and Whitehouse \cite{MW09}, and File, Martin, and Pitale \cite{FMP17}, among others. The setting is as follows: let $\pi$ be an automorphic representation of $\GL_2(\A_F)$ for some number field $F$, let $E$ be a quadratic extension of $F$ embedded in a quaternion algebra $\Dgp$ defined over $F$, let $\Omega: E^{\times} \backslash \A_E^\times \to \C^\times$ be a unitary Hecke character for which $\Omega|_{\A_F^\times}$ is equal to the central character of $\pi$, and let $\phi$ be a test vector in the automorphic representation $\pi^{\Dgp}$ of $\Dgp^\times (\A_F)$ corresponding to $\pi$ via the Jacquet--Langlands correspondence. We then define the ad\`{e}lic period integral
\[\Pscr_{\Omega}^{\Dgp}(\phi) \coloneqq \int\limits_{\A_{F}^{\times} E^{\times} \backslash \A_E^{\times}} \phi(x) \Omega^{-1}(x) \, dx.\]
Note that implicitly this depends on a choice of embedding $\A_E^\times \hookrightarrow \Dgp^\times (\A_F)$, which is suppressed in the notation, as well as a choice of normalisation of the measure $dx$.

A remarkable result of Waldspurger \cite{Wal85} is the formula
\begin{equation}
\label{eqn:Waldspurger}
\frac{|\Pscr_{\Omega}^{\Dgp}(\phi)|^2}{\langle \phi,\phi \rangle} = c_{\Omega,\phi} \frac{\Lambda\left(\frac{1}{2}, \pi_E \otimes \Omega\right)}{\Lambda(1,\ad \pi)},
\end{equation}
where $\phi$ is \emph{any} nonzero test vector in $\pi^{\Dgp}$ and $c_{\Omega,\phi}$ is a finite product of local factors. Here $\pi_E$ denotes the base change of $\pi$ to an automorphic representation of $\GL_2(\A_E)$; alternatively, we may write $\Lambda(s,\pi_E \otimes \Omega) = \Lambda(s,\pi \otimes \pi_{\Omega})$, where $\pi_{\Omega}$ denotes the automorphic induction of the Hecke character $\Omega$ to an automorphic representation of $\GL_2(\A_F)$.

\subsubsection{An Explicit Formula}

For applications in analytic number theory, it is essential that we have at our disposal a completely explicit form of Waldspurger's formula \eqref{eqn:Waldspurger}. Building on the work of Jacquet and Chen \cite{JC01}, Martin and Whitehouse \cite[Theorem 4.1]{MW09} provide such a formula for a \emph{specific} choice of test vector $\phi\in \pi^{\Dgp}$ (under some local assumptions, which were slightly relaxed by File, Martin, and Pitale \cite{FMP17}). For this specific choice of test vector, the local factors (whose product we denoted by $c_{\Omega,\phi}$ in \eqref{eqn:Waldspurger}) are described in \cite[Section 4.2]{MW09}.

For our application, we can restrict to the case where $F = \Q$, $E = \Q(\sqrt{D})$ with $D$ a positive fundamental discriminant, $\pi = \pi_f = \pi_{\infty} \otimes \bigotimes_p \pi_p$ is a cuspidal automorphic representation of $\GL_2(\A_{\Q})$ associated to a Hecke--Maa\ss{} newform $f \in \Cscr_0(\Gamma_0(q))$ of weight $0$, principal nebentypus, and squarefree level $q$ for which every prime diving $q$ splits in $E$, and $\Omega$ is the id\`{e}lic lift of a narrow class character $\chi$; later, we specialise to $q$ either equal to $1$ or equal to an odd prime. With this choice of data, the quaternion algebra $\Dgp$ is simply the matrix algebra $\Mat_{2 \times 2}$, so that $\pi^{\Dgp} = \pi$. We shorten notation and write $\Pscr_{\Omega}(\phi) \coloneqq \Pscr_{\Omega}^{\Dgp}(\phi)$ in this case.

The choice of test vector $\phi \in \pi$ used in \cite{MW09} is characterised by some local compatibilities with the Hecke character $\Omega$ and thus implicitly depends on the choice of embedding $\Psi_{\A_{\Q}} : \A_E \hookrightarrow \Dgp(\A_{\Q})$. The properties that characterise the local test vectors $\phi_p$ and $\phi_{\infty}$ are described in \cite[p.~172]{MW09} and are as follows.
\begin{itemize}
\item At a nonarchimedean place $p$, Martin and Whitehouse pick $\phi_p \in \pi_p$ to be nonzero and invariant under the units $R^{\times}$ of a certain order $R$ in the local quaternion algebra (which determines $\phi_p$ up to scaling). In our setting, $R$ is simply the Eichler order in $\GL_2(\Q_p)$ of reduced discriminant $p^{c(\pi_p)}$ such that
\[R \cap \Psi_p(E_p) = \Psi_p(\OO_{E_p}),\]
where $c(\pi_p)$ denotes the conductor exponent of $\pi_p$ and $E_p \coloneqq E \otimes \Q_p$.
\item At the archimedean place, we let $K_{\infty} \cong \Ogp(2)$ be a maximal compact subgroup of $\GL_2(\R)$ such that $K_{\infty} \cap \Psi_{\infty}(E_{\infty}^{\times}) \cong (\Z/2\Z)^2$ is a maximal compact subgroup of $\Psi_{\infty}(E_{\infty}^{\times}) \cong (\R^{\times})^2$, where $E_{\infty} \coloneqq E \otimes \R \cong \R^2$. Martin and Whitehouse pick $\phi_{\infty}$ such that $K_{\infty} \cap \Psi_{\infty}(E_{\infty}^{\times})$ acts (via $\pi$) on $\phi_{\infty}$ in the same way as $\Omega_{\infty} : E_{\infty}^{\times} \to \C^{\times}$ and $\phi_{\infty}$ lies in the \emph{minimal} such $K_{\infty}$-type in the sense of Popa \cite[Theorem 1]{Pop08} (which also uniquely determines $\phi_{\infty}$ up to scaling).
\end{itemize}
In our application, we slightly modify the choice of test vector $\phi_{\infty}$.

In order to obtain an explicit formula, we must now specify an embedding $\Psi_{\A_{\Q}} : \A_E \hookrightarrow \Mat_{2 \times 2}(\A_{\Q})$ and then determine which choice of local test vectors $\phi_p$ the above described conditions imply.

\subsubsection{A Specific Test Vector}
\label{sect:specifictestvector}

We construct an embedding $\Psi_{\A_{\Q}}$ using an oriented optimal embedding $\Psi : E \hookrightarrow \Mat_{2 \times 2}(\Q)$ of level $q$ as described in \eqref{eqn:Psi} associated to a Heegner form $Q = [a,b,c] \in \QQ_D(q)$ as in \eqref{eqn:Qxy}. By tensoring with $\A_\Q$, we get an embedding
\[\Psi_{\A_\Q} = (\Psi_{\infty}, \Psi_2,\Psi_3,\Psi_5,\ldots) : \A_E \hookrightarrow \Mat_{2 \times 2}(\A_{\Q}).\]

Since $\Psi$ is an \emph{optimal} embedding, the Eichler order $R$ is exactly the standard order of level $p^{c(\pi_p)}$ for each prime $p$, so that 
\[\prod_p R^\times = K_0(q),\]
where $K_0(q)\subset \GL_2(\A_\Q)$ is the congruence subgroup of level $q$ as in \eqref{eqn:K0(q)}. This means that we can choose the local component of $\phi$ at each prime to be the same as those of the ad\`{e}lisation $\phi_{F_2}$ of our Maa\ss{} cusp form $F_2 = R_0 f$.

At the archimedean place, there is the slight complication that $\Psi_{\infty} : E_{\infty} \hookrightarrow \Mat_{2 \times 2}(\R)$ is \emph{not} the diagonal embedding. If, however, we conjugate $\Psi_{\infty}$ by the matrix $\gamma_{\infty} \in \GL_2^{+}(\R)$ given by
\[\gamma_{\infty} \coloneqq \begin{dcases*}
\begin{pmatrix} -b - \sqrt{D} & b - \sqrt{D} \\ 2a & -2a \end{pmatrix} & if $a > 0$,	\\
\begin{pmatrix} b + \sqrt{D} & b - \sqrt{D} \\ -2a & -2a \end{pmatrix} & if $a < 0$,
\end{dcases*}\]
where $(a,b,c) \in \Z^3$ are associated to $\Psi$ as in \hyperref[sect:embeddings]{Section \ref*{sect:embeddings}}, then we obtain the diagonal embedding, namely
\begin{equation}
\label{eqn:diagembedding}
(\gamma_{\infty} \cdot \Psi_{\infty})(x + \sqrt{D}y,x - \sqrt{D}y) \coloneqq \gamma_{\infty}^{-1} \Psi_{\infty}(x + \sqrt{D}y,x - \sqrt{D}y) \gamma_{\infty} = \begin{pmatrix} x + \sqrt{D} y & 0 \\ 0 & x - \sqrt{D} y \end{pmatrix}.
\end{equation}
Note that $\gamma_{\infty} i = z_Q$ and $\Psi_{\infty}(\epsilon_D,\epsilon_D^{-1}) \gamma_{\infty} i = \gamma_Q z_Q$, where $z_Q$ and $\gamma_Q z_Q$ are as in \eqref{eqn:closedgeodesic}.

With this in mind, Martin and Whitehouse choose the local component of $\phi$ at the archimedean place to be
\[\begin{dcases*}
\pi_{\infty}(\gamma_{\infty}) \phi_{F_0,\infty} & if $\epsilon_f = \chi(J)$,	\\
\pi_{\infty}(\gamma_{\infty}) \phi_{F_2,\infty} - \pi_{\infty}(\gamma_{\infty}) \phi_{F_{-2},\infty} & if $\epsilon_f = -\chi(J)$,
\end{dcases*}\]
where $\phi_{F_0,\infty}$, $\phi_{F_2,\infty}$, and $\phi_{F_{-2},\infty}$ are the local components of the ad\`{e}lisations $\phi_{F_0}$, $\phi_{F_2}$, and $\phi_{F_{-2}}$ of $F_0 = f$, $F_2 = R_0 f$, and $F_{-2} =\LL_0 f$ respectively. We instead merely take the local component of $\phi$ at the archimedean place to be $\pi_{\infty}(\gamma_{\infty}) \phi_{F_2,\infty}$.

Altogether, the above implies that when using the embedding $\Psi_{\A_\Q}: \A_E \hookrightarrow \Mat_{2 \times 2}(\A_{\Q})$, our test vector is
\[\phi = \pi(\gamma_{\infty}) \phi_{F_2},\]
where we view $\gamma_{\infty} \in \GL_2^{+}(\R)$ as an element of $\GL_2(\A_{\Q})$.

\subsection{A Formula for Certain Ad\`{e}lic Period Integrals}

Let $\phi_{F_2} : \GL_2(\A_{\Q}) \to \C$ denote the ad\`{e}lic lift of $F_2 \coloneqq R_0 f \in \Cscr_2(\Gamma_0(q))$, which is an element of the cuspidal automorphic representation $\pi = \pi_f$ of $\GL_2(\A_{\Q})$ associated to $f$. Let $\Omega \in \widehat{E^{\times} \backslash \A_E^1}$ be the id\`{e}lic lift of $\chi$, so that $\Omega$ is a unitary Hecke character that is unramified at every nonarchimedean place of $E$ and has local components at the two archimedean places of $E$ of the form $(\sgn^{\kappa_{\chi}},\sgn^{\kappa_{\chi}})$ with $\kappa_{\chi} \in \{0,1\}$ such that $(-1)^{\kappa_{\chi}} = \chi(J)$. We study the ad\`{e}lic period integral
\begin{equation}
\label{eqn:Pscrdef}
\Pscr_{\Omega}(\pi(\gamma_{\infty})\phi_{F_2}) \coloneqq \int\limits_{\A_{\Q}^{\times} E^{\times} \backslash \A_E^{\times}} \phi_{F_2}(\Psi_{\A_\Q}(x)\gamma_{\infty}) \Omega^{-1}(x) \, dx.
\end{equation}
The measure $dx$ is normalised such that $\A_{\Q}^{\times} E^{\times} \backslash \A_E^{\times}$ has volume $2\Lambda(1,\chi_D) = 2L(1,\chi_D)$, where
\[\Lambda(s,\chi_D) \coloneqq \pi^{-\frac{s}{2}} \Gamma\left(\frac{s}{2}\right) L(s,\chi_D).\]

\begin{lemma}
\label{lem:MW}
We have that
\[\left|\Pscr_{\Omega}(\pi(\gamma_{\infty})\phi_{F_2})\right|^2 = \frac{1 - \epsilon_f \chi(J)}{q\sqrt{D}} \frac{\Gamma\left(\frac{3}{4} + \frac{it_f}{2}\right)^2 \Gamma\left(\frac{3}{4} - \frac{it_f}{2}\right)^2}{\Gamma\left(\frac{1}{2} + it_f\right) \Gamma\left(\frac{1}{2} - it_f\right)} \frac{L\left(\frac{1}{2}, f \otimes \Theta_{\chi}\right)}{L(1,\ad f)}.\]
\end{lemma}

\begin{proof}
We apply \cite[Theorem 4.1]{MW09} with $F = \Q$, $E = \Q(\sqrt{D})$, $\varphi = \pi(\gamma_{\infty}) \phi_{F_2}$ (so that $\pi = \pi_f$), and $\Omega$ as above. With this choice of data, we have that $S'(\pi) = S(\Omega) = \emptyset$, $\Ram(\pi) = \{p : p \mid q\}$, $\Delta_F = 1$, $\Delta_E = D$, $c(\Omega) = 1$, and $\Sigma_F^{\infty} = \{\infty\}$ in the notation of \cite[Theorem 4.1]{MW09}.

There is a slight caveat; \cite[Theorem 4.1]{MW09} does not quite apply since although the automorphic form $\phi_{F_2}$ has the same local Whittaker functions $W_p \in \WW(\pi_p,\psi_p)$ at every nonarchimedean place to that appearing in \cite[Theorem 4.1]{MW09} (compare \hyperref[lem:nonarchimedeanWhittaker]{Lemma \ref*{lem:nonarchimedeanWhittaker}} to \cite[Section 2]{MW09}), the local Whittaker function $W_{\infty}^2 \in \WW(\pi_{\infty},\psi_{\infty})$ at the archimedean place, as described in \hyperref[lem:archimedeanWhittaker]{Lemma \ref*{lem:archimedeanWhittaker}}, has a slightly different form than that appearing in \cite[Theorem 4.1]{MW09}. This issue is readily circumvented: we replace the term $C_{\infty}(E,\pi,\Omega)$ appearing in \cite[Theorem 4.1]{MW09} with its definition in \cite[Section 4.2.2]{MW09} in terms of local archimedean $L$-functions and $\widetilde{J}_{\pi_{\infty}}(f_{\infty})$, where now
\[\widetilde{J}_{\pi_{\infty}}(f_{\infty}) = \frac{\displaystyle \left|\int_{\R^{\times}} W_{\infty}^{2}\begin{pmatrix} a & 0 \\ 0 & 1 \end{pmatrix} \sgn^{\kappa_{\chi}}(a) \, d^{\times}a\right|^2}{\displaystyle \int_{\R^{\times}} \left|W_{\infty}^{2}\begin{pmatrix} a & 0 \\ 0 & 1 \end{pmatrix}\right|^2 \, d^{\times}a}.\]
This local distribution is just as in \cite[Section 3.3]{MW09} except that we have projected onto the local Whittaker function $W_{\infty}^{2} \in \WW(\pi_{\infty},\psi_{\infty})$ associated to $\phi_{F_2}$ instead of the local Whittaker function $W_{\infty} \in \WW(\pi_{\infty},\psi_{\infty})$ for which the numerator is equal to the local archimedean $L$-function and additionally satisfying $W_{\infty}\begin{psmallmatrix} a & 0 \\ 0 & 1 \end{psmallmatrix} = (-1)^{\kappa_{\Omega}} W_{\infty}\begin{psmallmatrix} -a & 0 \\ 0 & 1 \end{psmallmatrix}$ for all $a \in \R^{\times}$.

With this minor modification at the archimedean place, we deduce from \cite[Theorem 4.1]{MW09} that
\begin{equation}
\label{eqn:PscrMW}
\left|\Pscr_{\Omega}(\pi(\gamma_{\infty}) \phi_{F_2})\right|^2 = \frac{\pi}{2\sqrt{D}} \prod_{p \mid q} \frac{1}{1 - p^{-1}} \widetilde{J}_{\pi_{\infty}}(f_{\infty}) \frac{L\left(\frac{1}{2}, \pi_f \otimes \pi_{\Omega}\right)}{L(1,\ad \pi_f)} \int\limits_{\Zgp(\A_{\Q}) \GL_2(\Q) \backslash \GL_2(\A_{\Q})} |\phi_{F_2}(g)|^2 \, dg,
\end{equation}
where $\Zgp(\A_{\Q})$ denotes the centre of $\GL_2(\A_{\Q})$ and the measure $dg$ is normalised to be the Tamagawa measure multiplied by $\xi^q(2)$, where
\[\xi^q(s) \coloneqq \pi^{-\frac{s}{2}} \Gamma\left(\frac{s}{2}\right) \zeta(s) \prod_{p \mid q} \left(1 - \frac{1}{p^s}\right),\]
so that
\begin{equation}
\label{eqn:adelicvolume}
\vol(\Zgp(\A_{\Q}) \GL_2(\Q) \backslash \GL_2(\A_{\Q})) = 2 \xi^q(2) = \frac{\pi}{3} \prod_{p \mid q} \left(1 - \frac{1}{p^2}\right)
\end{equation}
since the Tamagawa number of $\PGL_2$ is $2$. From \eqref{eqn:W2int} and \eqref{eqn:W2norm} with $\kappa = \kappa_f$, $\kappa' = \kappa_{\chi}$, $t = t_f$, and $s = 1/2$,
\begin{equation}
\label{eqn:Jpif}
\widetilde{J}_{\pi_{\infty}}(f_{\infty}) = \begin{dcases*}
\frac{4}{\pi \left(\frac{1}{4} + t_f^2\right)} \frac{\Gamma\left(\frac{3}{4} + \frac{it_f}{2}\right)^2 \Gamma\left(\frac{3}{4} - \frac{it_f}{2}\right)^2}{\Gamma\left(\frac{1}{2} + it_f\right) \Gamma\left(\frac{1}{2} - it_f\right)} & if $\kappa_f \equiv \kappa_{\chi} + 1 \pmod{2}$,	\\
0 & if $\kappa_f \equiv \kappa_{\chi} \pmod{2}$.
\end{dcases*}
\end{equation}
Furthermore,
\begin{equation}
\label{eqn:L2norm}
\begin{split}
\int\limits_{\Zgp(\A_{\Q}) \GL_2(\Q) \backslash \GL_2(\A_{\Q})} |\phi_{F_2}(g)|^2 \, dg & = \frac{1}{q} \prod_{p \mid q} \left(1 - \frac{1}{p}\right) \int_{\Gamma_0(q) \backslash \Hb} |(R_0 f)(z)|^2 \, d\mu(z)	\\
& = \frac{1}{q} \prod_{p \mid q} \left(1 - \frac{1}{p}\right) \left(\frac{1}{4} + t_f^2\right) \int_{\Gamma_0(q) \backslash \Hb} |f(z)|^2 \, d\mu(z),
\end{split}
\end{equation}
where the first equality holds via the strong approximation theorem, \eqref{eqn:strongapprox}, while the second equality follows from \cite[(4.38)]{DFI02}; to check that the normalisation of measures in the first equality is correct, we replace $\phi_{F_2}$ with the constant function $1$ and recall \eqref{eqn:adelicvolume} and \eqref{eqn:Gamma0qvol}. We obtain the result upon combining \eqref{eqn:PscrMW}, \eqref{eqn:Jpif}, and \eqref{eqn:L2norm} and noting that $L(s,\pi_f \otimes \pi_{\Omega}) = L(s, f \otimes \Theta_{\chi})$ and $L(s,\ad \pi_f) = L(s,\ad f)$, and finally that
\[1 - \epsilon_f \chi(J) = \begin{dcases*}
2 & if $\kappa_f \equiv \kappa_{\chi} + 1 \pmod{2}$,	\\
0 & if $\kappa_f \equiv \kappa_{\chi} \pmod{2}$.
\end{dcases*}\qedhere\]
\end{proof}

\subsection{From Ad\`{e}lic Period Integrals to Cycle Integrals}

We relate the ad\`{e}lic period integral \eqref{eqn:Pscrdef} to a certain sum of cycle integrals over oriented geodesics in $\Gamma_0(q) \backslash \Hb$ indexed by narrow ideal classes. We first define these cycle integrals and show that they are well-defined. Given $f \in \Cscr_0(\Gamma_0(q))$ and a Heegner form $Q \in \QQ_D(q)$, we consider the cycle integral
\begin{equation}
\label{eqn:cycleint}
\int_{z_Q}^{\gamma_Q z_Q} (R_0 f)(z) \, \frac{dz}{\Im(z)},
\end{equation}
where $z_Q$ and $\gamma_Q z_Q$ are as in \eqref{eqn:closedgeodesic} and the contour of integration is the geodesic segment between these two points.

\begin{lemma}
\label{lem:cycintinv}
For all $\gamma \in \Gamma_0(q)$, the cycle integral \eqref{eqn:cycleint} is invariant under replacing $Q$ by $\gamma \cdot Q$.
\end{lemma}

For this reason, we may write \eqref{eqn:cycleint} as
\[\int_{\CC_A(q)} (R_0 f)(z) \, \frac{dz}{\Im(z)}\]
without ambiguity, where $\CC_A(q)$ denotes the oriented geodesic in $\Gamma_0(q) \backslash \Hb$ associated to a narrow ideal class $A$ corresponding to $Q$ as in \hyperref[sect:closedgeodesic]{Section \ref*{sect:closedgeodesic}}, since this cycle integral is independent of the choice of Heegner form $Q$ associated to $A$.

\begin{proof}[Proof of {\hyperref[lem:cycintinv]{Lemma \ref*{lem:cycintinv}}}]
Suppose that $Q' = \gamma \cdot Q$ for some $\gamma \in \Gamma_0(q)$. We make the change of variables $z \mapsto \gamma^{-1} z$ in \eqref{eqn:cycleint}. The integrand remains unchanged since
\[(R_0 f)(\gamma z) = j_{\gamma}(z)^2 (R_0 f)(z), \qquad \frac{d}{dz}(\gamma z) = \frac{\Im(\gamma z)}{\Im(z)} j_{\gamma}(z)^{-2}.\]
It is easily checked that $\gamma^{-1} \gamma_Q \gamma = \gamma_{Q'}$, and so the new contour of integration is the geodesic segment from $\gamma z_Q$ to $\gamma_{Q'} \gamma z_Q$ on the semicircle \eqref{eqn:semicircle} associated to $Q'$. Further changes of variables by powers of $\gamma_{Q'}$ rotate the contour of integration along this semicircle while leaving the integrand intact, and so there is an appropriate power of $\gamma_{Q'}$ for which the resulting geodesic segment intersects nontrivially with the geodesic segment from $z_{Q'}$ to $\gamma_{Q'} z_{Q'}$. We then break up the integral into two parts, and for the part that does not intersect this geodesic segment, we make one last change of variables by either $\gamma_{Q'}$ or $\gamma_{Q'}^{-1}$ as appropriate; recombining, we obtain \eqref{eqn:cycleint} with $Q'$ in place of $Q$.
\end{proof}

With this in hand, we now write the ad\`{e}lic period integral $\Pscr_{\Omega}(\pi(\gamma_{\infty}) \phi_{F_2})$ defined in \eqref{eqn:Pscrdef} as a geodesic cycle integral.

\begin{lemma}
\label{lem:adelictocycle}
We have that
\begin{equation}
\label{eqn:adelictocycle}
\Pscr_{\Omega}(\pi(\gamma_{\infty}) \phi_{F_2}) = -\frac{i \overline{\chi}(A_\Psi)}{\sqrt{D}} \sum_{A \in \Cl_D^+} \overline{\chi}(A) \int_{\CC_A(q)} (R_0 f)(z) \, \frac{dz}{\Im(z)},
\end{equation}
where $A_\Psi \in \Cl_D^+$ is the element of the narrow class group associated to the oriented optimal embedding $\Psi$. In particular, $\Pscr_{\Omega}(\pi(\gamma_{\infty}) \phi_{F_2})$ is independent of the choice of oriented optimal embedding $\Psi$ of level $q$ within an equivalence class of embeddings modulo the action of $\Gamma_0(q)$.
\end{lemma}

A related result is proven by Popa in \cite[Section 6]{Pop06}, namely the identity
\[\Pscr_{\Omega}(\pi(\gamma_{\infty}) \phi) = \frac{i^k \overline{\chi}(A_\Psi)}{D^{\frac{k - 1}{2}}} \sum_{A \in \Cl_D^+} \overline{\chi}(A) \int_{\CC_A(q)} g(z) Q(z,1)^{k - 1} \, dz.\]
Here $g$ is a \emph{holomorphic} Hecke newform of weight $k$, level $q$, and trivial nebentypus and $\phi$ is the ad\`{e}lic lift of $g$. Popa additionally proves an identity relating $|\Pscr_{\Omega}(\pi(\gamma_{\infty}) \phi)|^2$ to $L(\frac{1}{2},g \otimes \Theta_{\chi})$ akin to \hyperref[lem:MW]{Lemma \ref*{lem:MW}} via the theta correspondence \cite[Theorem 5.4.1]{Pop06}. The proof of \hyperref[lem:adelictocycle]{Lemma \ref*{lem:adelictocycle}} given below closely follows that of Popa in \cite[Section 6]{Pop06}.

\begin{proof}[Proof of {\hyperref[lem:adelictocycle]{Lemma \ref*{lem:adelictocycle}}}]
Since $\Omega$ is the id\`{e}lic lift of a narrow class character, it is trivial on both $\widehat{\OO}_E^{\times}$ and $\A_{\Q}^{\times}$. Furthermore, we have the inclusion $\Psi_{\A_{\Q}}(\widehat{\OO}_E^{\times}) \subset K_0(q)$ since $\Psi$ is an optimal embedding. As the newform $f$ has level $q$, it follows that both $\pi(\gamma_{\infty})\phi_{F_2}$ and $\Omega$ are well-defined on the double quotient $\A_{\Q}^{\times} E^{\times} \backslash \A_E^{\times} / \widehat{\OO}_E^{\times}$. Since $\widehat{\OO}_E^{\times}$ has measure $1$, we deduce that
\[\Pscr_{\Omega}(\pi(\gamma_{\infty}) \phi_{F_2}) = \int\limits_{\A_{\Q}^{\times} E^{\times} \backslash \A_E^{\times} / \widehat{\OO}_E^{\times}} \phi_{F_2}(\Psi_{\A_{\Q}}(x) \gamma_{\infty}) \Omega^{-1}(x) \, dx.\]

Via the strong approximation theorem, we have the decomposition
\[\A_{\Q}^{\times} E^{\times} \backslash \A_E^{\times} / \widehat{\OO}_E^{\times} \cong \bigsqcup_{A \in \Cl_D^+} A \cdot \epsilon_D^{\Z} \backslash E_{\infty}^1,\]
where $A = (A_v) \in \A_E^{\times}$ runs through a set of finite id\`{e}le representatives of the narrow class group $\Cl_D^+$ (so that $A_v = 1$ if $v$ is an archimedean place of $E$), which we freely identify with elements of the narrow class group, while $E_{\infty}^1 \coloneqq \{(t,t^{-1}) \in E_{\infty} : t > 0\} \cong \R_{+}^{\times}$, which we view as a subgroup of $\A_E^{\times}$ via the embedding $(t,t^{-1}) \mapsto (t,t^{-1},1,1,\ldots)$, and $\epsilon_D^{\Z} \coloneqq \{(\epsilon_D^m, \epsilon_D^{-m}) \in E_{\infty} : m \in \Z\}$. Thus every $x \in \A_{\Q}^{\times} E^{\times} \backslash \A_E^{\times} / \widehat{\OO}_E^{\times}$ can be written as $x = A (t,t^{-1},1,1,\ldots)$ with $t \in [1,\epsilon_D)$. From this, we may write
\[\Pscr_{\Omega}(\pi(\gamma_{\infty}) \phi_{F_2}) = \frac{2}{\sqrt{D}} \sum_{A \in \Cl_D^+} \overline{\chi}(A) \int_{1}^{\epsilon_D} \phi_{F_2}(\Psi_{\A_{\Q}}(A) \Psi_{\infty}(t,t^{-1}) \gamma_{\infty}) \, d^{\times}t.\]
We can check that the normalisation of measures here is correct by replacing $\pi(\gamma_{\infty}) \phi_{F_2}$ with the constant function $1$ and taking $\chi$ to be the trivial character, noting that $\vol(\A_{\Q}^{\times} E^{\times} \backslash \A_E^{\times})$ is $2L(1,\chi_D)$, whereas $h_D^+ \log \epsilon_D = \sqrt{D} L(1,\chi_D)$ by the narrow class number formula.

Let $g_A$ denote the inverse of the $\GL_2^+(\R)$ component $g_{\infty}$ of the representation $\gamma g_{\infty} k$ of the id\`{e}le $\Psi_{\A_{\Q}}(A) \in \A_{\Q}^{\times}$, as in \eqref{eqn:strongapprox}. By the definition \eqref{eqn:phistrongapprox} and \eqref{eqn:tildeftof} of the ad\`{e}lic lift together with the fact that $\det \gamma_{\infty} > 0$, we have that
\begin{equation}
\label{eqn:Pscrstrongapprox}
\Pscr_{\Omega}(\pi(\gamma_{\infty}) \phi_{F_2}) = \frac{2}{\sqrt{D}} \sum_{A \in \Cl_D^+} \overline{\chi}(A) \int_{1}^{\epsilon_D} j_{g_A^{-1} \Psi_{\infty}(t,t^{-1}) \gamma_{\infty}}(i)^{-2} (R_0 f)(g_A^{-1} \Psi_{\infty}(t,t^{-1}) \gamma_{\infty} i) \, d^{\times}t.
\end{equation}

We make the change of variables $z = g_A^{-1} \Psi_{\infty}(t,t^{-1}) \gamma_{\infty} i = g_A^{-1} \gamma_{\infty} (i t^2)$; the contour of integration in \eqref{eqn:Pscrstrongapprox} then becomes the oriented geodesic segment from $g_A^{-1} \gamma_{\infty} i = g_A^{-1} z_Q$ to $g_A^{-1} \gamma_Q z_Q$ parametrised by $g_A^{-1} \Psi_{\infty}(t,t^{-1}) z_Q$, where $z_Q$ and $\gamma_Q z_Q$ are as in \eqref{eqn:closedgeodesic} with $Q = Q_{\Psi} = [a,b,c]$ the Heegner form associated to the optimal embedding $\Psi$. We have that
\[d^{\times}t = -\frac{i}{2} j_{\gamma_{\infty}^{-1} g_A}(z)^{-2} \, \frac{dz}{\Im(z)},\]
since $t^2 = -i \gamma_{\infty}^{-1} g_A z$ and
\[\frac{d}{dz} (gz) = \frac{\Im(gz)}{\Im(z)} j_g(z)^{-2}\]
for any $g \in \GL_2(\R)$, while the cocycle relation $j_{g_1 g_2}(z) = j_{g_2}(z) j_{g_1}(g_2 z)$ implies that
\[j_{g_A^{-1} \Psi_{\infty}(t,t^{-1}) \gamma_{\infty}}(i) j_{\gamma_{\infty}^{-1} g_A}(z) = j_{\gamma_{\infty}^{-1} \Psi_{\infty}(t,t^{-1}) \gamma_{\infty}}(i) = 1,\]
where the last equality follows upon recalling \eqref{eqn:gzjgz} and \eqref{eqn:diagembedding}. We deduce that \eqref{eqn:Pscrstrongapprox} is equal to
\[-\frac{i}{\sqrt{D}} \sum_{A \in \Cl_D^+} \overline{\chi}(A) \int_{g_A^{-1} z_Q}^{g_A^{-1} \gamma_Q z_Q} (R_0 f)(z) \, \frac{dz}{\Im(z)}.\]

It is shown in \cite[Theorem 6.2.2 (i)]{Pop06} that as $A$ runs through the narrow class group $\Cl_D^+$, $g_A^{-1} \Psi g_A$ runs through a set of representatives of equivalence classes of oriented optimal embeddings of level $q$ modulo the action of $\Gamma_0(q)$. Furthermore, letting $Q'$ denote the Heegner form of level $q$ associated to the oriented optimal embedding $g_A^{-1} \Psi g_A$, we have that the contour of integration from $g_A^{-1} z_Q$ and $g_A^{-1} \gamma_Q z_Q = \gamma_{Q'} g_A^{-1} z_Q$ is a geodesic segment of length $2\log \epsilon_D$ of the semicircle \eqref{eqn:semicircle} associated to $Q'$, and hence
\[\int_{g_A^{-1} z_Q}^{g_A^{-1} \gamma_Q z_Q} (R_0 f)(z) \, \frac{dz}{\Im(z)} = \int_{\CC_{A'}(q)} (R_0 f)(z) \, \frac{dz}{\Im(z)}\]
by the proof of \hyperref[lem:cycintinv]{Lemma \ref*{lem:cycintinv}}, where $A' \in \Cl_D^+$ is the element of the narrow class group associated to $Q'$. It remains to note that by \cite[Theorem 6.2.2 (ii)]{Pop06}, the oriented optimal embedding $g_A^{-1} \Psi g_A$ is associated to $A A_{\Psi} \in \Cl_D^+$, where $A_{\Psi}$ is the element of $\Cl_D^+$ corresponding to $\Psi$.
\end{proof}

\subsection{From Cycle Integrals to Weyl Sums for Newforms}

Having shown in \hyperref[lem:adelictocycle]{Lemma \ref*{lem:adelictocycle}} that the ad\`{e}lic period integral $\Pscr_{\Omega}(\pi(\gamma_{\infty}) \phi_{F_2})$ defined in \eqref{eqn:Pscrdef} may be reexpressed as a geodesic cycle integral defined in \eqref{eqn:adelictocycle}, we now show that this integral of a Maa\ss{} cusp form over a closed geodesic $\CC_A(q)$ is related to an integral over the hyperbolic orbifold $\Gamma_A(q) \backslash \NN_A(q)$ plus a certain \emph{topological contribution}. Furthermore, we show that a similar relation holds for Eisenstein series, with a much simpler proof. It is at this point that we specialise to the setting of $q$ being either equal to $1$ or equal to an odd prime.

\begin{lemma}[Cf.\ {\cite[Lemmata 1 and 2]{DIT16}}]
\label{lem:cycletoWeyl}
Let $f \in \Cscr_0(\Gamma_0(q))$ be a Maa\ss{} cusp form of weight $0$, level $q$, where either $q = 1$ or $q$ is an odd prime that splits in $E = \Q(\sqrt{D})$, and Laplacian eigenvalue $\lambda_f = \frac{1}{4} + t_f^2$. Then for each $A \in \Cl_D^+$, we have that
\[\int_{\FF_A(q)} f(z) \, d\mu(z) = \frac{1}{\frac{1}{4} + t_f^2} \left(\int_{\CC_A(q)} (R_0 f)(z) \, \frac{dz}{\Im(z)} + \sum_{j=1}^{2g} \int_{\CC_j} (R_0 f)(z) \, \frac{dz}{\Im(z)} \langle [\CC_A(q)], \omega_j \rangle_\mathrm{cap}\right).\]
Similarly, let $E(z,\frac{1}{2} + it)$ be the Eisenstein series on $\Gamma \backslash \Hb$ with Laplacian eigenvalue $\frac{1}{4} + t^2$, where $t \in \R$. Then for $\ell \mid q$, we have that
\begin{multline*}
\int_{\FF_A(q)} E\left(\ell z, \frac{1}{2} + it\right) \, d\mu(z) = \frac{1}{\frac{1}{4} + t^2} \left( \vphantom{ \sum_{j=1}^{2g}} \int_{\CC_A(q)} (R_0 E)\left(\ell z, \frac{1}{2} + it\right) \, \frac{dz}{\Im(z)} \right. \\
\left. + \sum_{j=1}^{2g} \int^\ast_{\CC_j} (R_0 E)\left(\ell z, \frac{1}{2} + it\right) \, \frac{dz}{\Im(z)} \langle [\CC_A(q)], \omega_j \rangle_\mathrm{cap}\right),
\end{multline*}
where $\int^\ast_{\CC_j}$ denotes the regularised integral defined in \eqref{eqn:regint2}.
\end{lemma}

\begin{proof}
We show this for $q$ prime; the case of $q = 1$ follows in the same way except that $g = 0$, so that the sum over $j$ is empty. Let $\FF_A(q) \subset \Hb$ be the fundamental domain for $\Gamma_A(q) \backslash \NN_A(q)$ defined in \eqref{eqn:surface}. For each vertex $v$ of $\FF_A(q)$ and for $Y > 1$, let $\FF_{v}(Y)$ denote the cuspidal zone defined in \eqref{eqn:cuspidalzone}. Let $\phi$ be either a Hecke--Maa{\ss} cusp form or an Eisenstein series as above and denote by $\lambda_\phi$ the associated Laplace eigenvalue. By Stokes' theorem, we have that for all $Y$ sufficiently large,
\begin{multline}
\label{eqn:Stokes}
\int_{\FF_A(q) \setminus \bigcup_{v} \FF_{v}(Y)} \frac{\dee^2}{\dee \overline{z} \dee z} \phi(z) \, dz \, d\overline{z} \\
=- \int_{\dee \FF_A(q) \setminus \bigcup_{v} \FF_{v}(Y)} \frac{\dee}{\dee z} \phi(z) \, dz - \sum_{i=1}^{n+2} n_i(A,q)\int_{I_i(q)} \left. \frac{\dee}{\dee z}\right|_{z = x + iY} \phi\left(\begin{pmatrix}a_i & a_{i-1} \\ b_i & b_{i-1}\end{pmatrix} z\right) \, dx.
\end{multline}
Here $n_i(A,q)$ is the multiplicity in $\FF_A(q)$ of the cuspidal zone $\FF_{\frac{a_i}{b_i}}(Y)$ around the vertex $\frac{a_i}{b_i}$ (which is some number independent of $Y$), while $I_i(q)\subset \R$ is the horocycle interval at height $Y$ around the vertex $\frac{a_i}{b_i}$ of $\PP(q)$ defined by
\[\begin{pmatrix}a_i & a_{i-1} \\ b_i & b_{i-1} \end{pmatrix}^{-1} \PP(q) \cap \{x + iY : x \in \R\}= \{x + iY : x \in I_i(q)\}.\]

Since
\[\frac{\dee^2}{\dee \overline{z} \dee z} = \frac{1}{4} \left(\frac{\dee^2}{\dee x^2} + \frac{\dee^2}{\dee y^2}\right), \qquad dz \, d\overline{z} = -2iy^2 \, d\mu(z),\]
the left-hand side of \eqref{eqn:Stokes} is
\[\frac{i \lambda_\phi}{2} \int_{\FF_A(q) \setminus \bigcup_{v} \FF_{v}(Y)} \phi(z) \, d\mu(z).\]
As $Y$ tends to infinity, this converges to
\[\frac{i \lambda_\phi}{2} \int_{\FF_A(q)} \phi(z) \, d\mu(z).\]

The first term on the right-hand side of \eqref{eqn:Stokes} is equal to
\begin{multline*}
- \int_{\CC_A(q)} \frac{\dee}{\dee z} \phi(z) \, dz -\sum_{j=1}^{\frac{N}{2}}m_j(A,q)\int_{\CC_j\setminus \bigcup_{v} \FF_{v}(Y)}\frac{\dee}{\dee z} \phi(z) \, dz	\\
= \frac{i}{2} \int_{\CC_A(q)} (R_0 \phi)(z) \, \frac{dz}{\Im(z)}+\frac{i}{2}\sum_{j=1}^{\frac{N}{2}}m_j(A,q)\int_{\CC_j\setminus \bigcup_{v} \FF_{v}(Y)} (R_0 \phi)(z) \, \frac{dz}{\Im(z)},
\end{multline*}
where $m_j(A,q)$ denotes the multiplicity of the side $\CC_j$ of $\PP(q)$ in $\Gamma_A(q) \backslash \NN_A(q)$ as defined in \hyperref[sect:homdes]{Section \ref*{sect:homdes}}. The second term on the right-hand side of \eqref{eqn:Stokes} converges to zero as $Y$ tends to infinity; for $\phi$ cuspidal, this follows from the rapid decay at each cusp $\bb$, while for Eisenstein series, this follows from \hyperref[lem:Eisderivdecay]{Lemma \ref*{lem:Eisderivdecay}}. Thus upon taking the limit as $Y$ tends to infinity, we obtain the desired identity by \hyperref[lem:multiplicity]{Lemmata \ref*{lem:multiplicity}} and \ref{lem:multiplicity2}.
\end{proof}

\subsection{Proofs of \texorpdfstring{\hyperref[prop:MaassnewformWeyl]{Propositions \ref*{prop:MaassnewformWeyl}}}{Propositions \ref{prop:MaassnewformWeyl}} and \ref{prop:EisnewformWeyl}}
\label{sect:proofsWeyl}

We are finally in a position to prove \hyperref[prop:MaassnewformWeyl]{Proposition \ref*{prop:MaassnewformWeyl}}.

\begin{proof}[Proof of {\hyperref[prop:MaassnewformWeyl]{Proposition \ref*{prop:MaassnewformWeyl}}}]
\hyperref[lem:adelictocycle]{Lemmata \ref*{lem:adelictocycle}} and \ref{lem:cycletoWeyl} imply that the Weyl sum $W_{\chi,f}$ defined in \eqref{eqn:WeylMaass} satisfies the identity
\[W_{\chi,f} = -\frac{i \overline{\chi}(A_{\Psi}) \sqrt{D}}{\frac{1}{4} + t_f^2} \overline{\Pscr_{\Omega}(\pi(\gamma_\infty)\phi_{F_2})} + \frac{1}{\frac{1}{4} + t_f^2} \sum_{j=1}^{2g} \int_{\CC_j} (R_0 f)(z) \, \frac{dz}{\Im(z)} \sum_{A\in \Cl_D^+} \chi(A) \langle [\CC_A(q)], \omega_j \rangle_\mathrm{cap}.\]
Expressing the class $[\CC_A(q)]\in H_1(X_0(q),\R)$ in terms of the Hecke basis defined in \hyperref[sect:Heckebasis]{Section \ref*{sect:Heckebasis}} and recalling the definition \eqref{eqn:omegafdefeq} gives
\begin{multline*}
\sum_{A\in \Cl_D^+}\chi(A)\langle [\CC_A(q)], \omega_j \rangle_\mathrm{cap}	\\
= \frac{1}{2} \sum_{h\in \BB_2^{\hol}(\Gamma_0(q))} \sum_{\pm} i^{\frac{1 \mp 1}{2}} \langle v_h^\pm, \omega_j \rangle_\mathrm{cap} \sum_{A\in \Cl_D^+} \chi(A) \left(\int_{\CC_A(q)} h(z) \, dz \pm \int_{\CC_A(q)} \overline{h(z) \, dz}\right).
\end{multline*}
Combining this, the result follows from \hyperref[lem:MW]{Lemma \ref*{lem:MW}} as well as the explicit form of Waldspurger's formula for $h\in \BB_2^\hol(q)$ due to Popa \cite[Theorem 6.3.1]{Pop08}, namely
\[\left|\sum_{A\in \Cl_D^+} \chi(A) \int_{\CC_A(q)} h(z) \, dz \right|^2 = \frac{\sqrt{D}}{4\pi^2} L\left(\frac{1}{2},h\otimes \Theta_{\overline{\chi}}\right).\qedhere\]
\end{proof}

The proof of \hyperref[prop:EisnewformWeyl]{Proposition \ref*{prop:EisnewformWeyl}} is a little simpler, since we can circumvent the ad\`{e}lic formulation of this Weyl sum.

\begin{proof}[Proof of {\hyperref[prop:EisnewformWeyl]{Proposition \ref*{prop:EisnewformWeyl}}}]
From \cite[(7.3)]{DIT16}, we have that
\[\sum_{A \in \Cl_D^+} \chi(A) \int_{\CC_A} (R_0 E)(z,s) \, \frac{dz}{\Im(z)} = (1 - \chi(J)) D^{\frac{s}{2}} \frac{\Gamma\left(\frac{s + 1}{2}\right)^2}{\Gamma(s)} \frac{L(s,\chi)}{\zeta(2s)}\]
for $\Re(s) > 1$. Via analytic continuation, this identity extends to $s = \frac{1}{2} + it$. The result then follows via \hyperref[lem:cycletoWeyl]{Lemma \ref*{lem:cycletoWeyl}}.
\end{proof}

\subsection{Weyl Sums for Oldforms}
\label{sect:oldforms}

Finally, we show that Weyl sums for oldforms are essentially equal to Weyl sums for the associated newform of lower level. We define for a Hecke--Maa\ss{} newform $f$ of weight $0$ and level $1$ and a narrow class character $\chi$
\begin{equation}
\label{eqn:WeylMaassold}
W^{\ell,q}_{\chi,f} \coloneqq \sum_{A \in \Cl_D^+} \chi(A) \int_{\FF_A(q)} f(\ell z) \, d\mu(z),
\end{equation}
and similarly for the Eisenstein series $E(\cdot,\frac{1}{2} + it)$ for $\Gamma \backslash \Hb$
\begin{equation}
\label{eqn:WeylEisold}
W^{\ell,q}_{\chi,t} \coloneqq \sum_{A \in \Cl_D^+} \chi(A) \int_{\FF_A(q)} E\left(\ell z, \frac{1}{2} + it\right) \, d\mu(z).
\end{equation}
We also write $W^{1,q}_{\chi,f}=W_{\chi,f}$ for $f$ a Hecke--Maa{\ss} newform of level $q$.

\begin{lemma} 
\label{lem:oldforms}
Let $q$ be an odd prime and let $g$ denote the genus of $X_0(q)$. Let $\ell \mid q$ and let $A_{\lf} \in \Cl_D^+$ denote the narrow ideal class containing the oriented ideal
\begin{equation}
\label{eqn:lfdef}
\left[\lf; \ell, \frac{r - \sqrt{D}}{2}\right].
\end{equation}
Then for a Hecke--Maa\ss{} newform $f$ of weight $0$ and level $1$ and a narrow class character $\chi$,
\begin{multline}
\label{eqn:oldformsMaass}
W^{\ell,q}_{\chi,f} = \chi(A_{\lf}) \sum_{A \in \Cl_D^+} \chi(A) \int_{\FF_A} f(z) \, d\mu(z)\\
+ \delta_{\ell,q} \sum_{j=1}^{2g} \int_{\CC_j} (R_0 f)(\ell z) \, \frac{dz}{\Im(z)} \sum_{A \in \Cl_D^+} \chi(A) \langle [\CC_A(q)], \omega_j \rangle_\mathrm{cap}.
\end{multline}
Similarly, for the Eisenstein series $E(z,1/2+it)$ for $\Gamma \backslash \Hb$,
\begin{multline}
\label{eqn:oldformsEis}
W_{\chi,t}^{\ell,q} = \chi(A_{\lf}) \sum_{A \in \Cl_D^+} \chi(A) \int_{\FF_A} E\left(z,\frac{1}{2} + it\right) \, d\mu(z)\\
+ \delta_{\ell,q} \sum_{j=1}^{2g} \int^\ast_{\CC_j} (R_0 E)\left(\ell z,\frac{1}{2} + it\right) \, \frac{dz}{\Im(z)} \sum_{A \in \Cl_D^+} \chi(A) \langle [C_A(q)], \omega_j \rangle_\mathrm{cap}.
\end{multline}
\end{lemma}

Note that although $r$ is only defined modulo $2q$, the oriented ideal \eqref{eqn:lfdef} is well-defined since it contains $\ell$ and hence also $q$.

\begin{proof}[Proof of {\hyperref[lem:oldforms]{Lemma \ref*{lem:oldforms}}}]
We give the proof for \eqref{eqn:oldformsMaass}; the same method yields \eqref{eqn:oldformsEis}. Let $Q = [a,b,c] \in \QQ_D(q)$ be a Heegner form associated to the oriented closed geodesic $\CC_A(q)$. Since $\CC_A(q)$ is the reduction modulo $\Gamma_0(q)$ of the geodesic segment from $z_Q$ to $\gamma_Q z_Q$, as in \eqref{eqn:closedgeodesic}, \hyperref[lem:cycletoWeyl]{Lemma \ref*{lem:cycletoWeyl}} together with the change of variables $z \mapsto g_{\ell}^{-1} z$, where $g_{\ell} \coloneqq \begin{psmallmatrix} \sqrt{\ell} & 0 \\ 0 & \frac{1}{\sqrt{\ell}} \end{psmallmatrix} \in \SL_2(\R)$, imply that
\begin{multline*}
\sum_{A \in \Cl_D^+} \chi(A) \int_{\FF_A(q)} f(\ell z) \, d\mu(z) = \frac{1}{\frac{1}{4} + t_f^2} \sum_{A \in \Cl_D^+} \chi(A) \int_{g_{\ell} z_Q}^{g_{\ell} \gamma_Q z_Q} (R_0 f)(z) \, \frac{dz}{\Im(z)}\\
+ \sum_{j=1}^{2g} \int_{\CC_j} (R_0 f)(\ell z) \, \frac{dz}{\Im(z)} \sum_{A \in \Cl_D^+} \chi(A) \langle [\CC_A(q)], \omega_j \rangle_\mathrm{cap}.
\end{multline*}
The Heegner form associated to the geodesic segment from $g_{\ell} z_Q$ to $g_{\ell} \gamma_Q z_Q$, as in \hyperref[sect:closedgeodesic]{Section \ref*{sect:closedgeodesic}}, is $Q_{\ell} \coloneqq [\frac{a}{\ell},b,c\ell] \in \QQ_D(\frac{q}{\ell})$. We claim that the narrow ideal class associated to $Q_{\ell}$ is $A_{\lf}^{-1} A$, from which \eqref{eqn:oldformsMaass} follows. From \eqref{eqn:Qtoideals}, this is implied by the fact that
\[\left(\Z \frac{a}{\ell} + \Z \frac{b - \sqrt{D}}{2}\right) \left(\Z \ell + \Z \frac{r - \sqrt{D}}{2}\right) = \Z a + \Z \frac{b - \sqrt{D}}{2}\]
if $a > 0$, and similarly
\[\left(\Z \left(-\frac{a}{\ell}\sqrt{D}\right) + \Z \frac{D - b\sqrt{D}}{2}\right) \left(\Z \ell + \Z \frac{r - \sqrt{D}}{2}\right) = \Z (-a\sqrt{D}) + \Z \frac{D - b\sqrt{D}}{2}\]
if $a < 0$. Finally, if $\ell = 1$, then the fact that $\CC_j$ is fixed by an order $2$ matrix in $\Gamma$ implies that
\[\int_{\CC_j} (R_0 f)(z) \, \frac{dz}{\Im(z)} = 0.\qedhere\]
\end{proof}

\section{Sparse Equidistribution in the Level Aspect}

\subsection{Explicit Orthonormal Bases}

In order to prove \hyperref[thm:level]{Theorem \ref*{thm:level}}, we require the spectral decomposition of $L^2(\Gamma_0(q) \backslash \Hb)$. This decomposition involves three parts: the residual spectrum, consisting of the constant function; the cuspidal spectrum, spanned by Hecke--Maa\ss{} cusp forms; and the continuous spectrum, spanned by incomplete Eisenstein series.

For the cuspidal spectrum, we may choose an orthonormal basis of this subspace consisting of certain linear combinations of oldforms and newforms.

\begin{lemma}[{\cite[Lemma 3.1]{HK20}}]
\label{lem:Maasortho}
Let $q$ be squarefree. An orthonormal basis of $\Cscr_0(\Gamma_0(q))$ with respect to the inner product
\[\langle f_1,f_2\rangle_q \coloneqq \int_{\Gamma_0(q) \backslash \Hb} f_1(z) \overline{f_2(z)} \, d\mu(z)\]
is given by
\[\left\{f_{\ell} \in \Cscr_0(\Gamma_0(q)) : f \in \BB_0^{\ast}(\Gamma_0(q_1)), \ q_1 q_2 = q, \ \ell \mid q_2\right\},\]
where $\BB_0^{\ast}(\Gamma_0(q_1))$ denotes an orthonormal basis of Hecke--Maa\ss{} newforms $f$ of weight $0$ and level $q_1$ with respect to the inner product $\langle \cdot, \cdot \rangle_{q_1}$, while $f_{\ell}$ is associated to $f \in \BB_0^{\ast}(\Gamma_0(q_1))$ by
\[f_{\ell}(z) \coloneqq \left(L_{\ell}(1,\ad f) \frac{\varphi(\ell)}{\ell \nu(q_2)}\right)^{\frac{1}{2}} \sum_{vw = \ell} \frac{\nu(v)}{v} \frac{\mu(w) \lambda_f(w)}{\sqrt{w}} f(vz).\]
Here for $\ell \nmid q_1$,
\[L_{\ell}(s,\ad f) \coloneqq \prod_{p \mid \ell} \frac{1}{1 - \lambda_f(p^2) p^{-s} + \lambda_f(p^2) p^{-2s} - p^{-3s}}.\]
\end{lemma}

A similar construction is valid for the continuous spectrum by choosing an orthonormal set of Eisenstein series.

\begin{lemma}[{\cite[Section 8.4]{You19}}]
\label{lem:Eisortho}
Let $q$ be squarefree. An orthonormal basis of the vector space of Eisenstein series of weight $0$ and level $q$ is given by
\[\left\{E_{\ell}\left(\cdot,\frac{1}{2} + it\right) : \ell \mid q \right\}\]
where
\[E_{\ell}\left(z,\frac{1}{2} + it\right) \coloneqq \left(\frac{\zeta_{\ell}(1 + 2it) \zeta_{\ell}(1 - 2it)}{\nu(q)}\right)^{\frac{1}{2}} \sum_{vw = \ell} \frac{\nu(v)}{v} \frac{\mu(w) \lambda(w,t)}{\sqrt{w}} E\left(vz, \frac{1}{2} + it\right).\]
Here
\[\zeta_{\ell}(s) \coloneqq \prod_{p \mid \ell} \frac{1}{1 - p^{-s}}.\]
\end{lemma}

\subsection{Averaged Bounds for Regularised Integrals}

We require bounds in terms of $q$ for the integrals of both Maa{\ss} cusp forms and Eisenstein series over the sides of $\PP(q)$. In order to bound the regularised integrals that show up in the Weyl sums, we start by showing a useful lemma: all the hyperbolic sides of $\PP(q)$ can be brought to a standard form under the action of the Atkin--Lehner operator $W_q = \begin{psmallmatrix} 0 & -\frac{1}{\sqrt{q}} \\ \sqrt{q} & 0 \end{psmallmatrix}$, as in \eqref{eqn:ALop}, and $\Gamma_0(q)$.

\begin{lemma}
\label{lem:canonicalform}
Let $\LL$ be a hyperbolic side of $\PP(q)$ with endpoints $\frac{a_i}{b_i}$ and $\frac{a_{i + 1}}{b_{i + 1}}$. There exists a unique element $\gamma_\LL \in \Gamma_0(q)$ such that
\[\gamma_\LL W_q \LL = \left\{z \in \Hb : \Re(z) = \frac{v}{q}\right\}\]
for $v \in \{1,\ldots,q - 1\}$ satisfying $vb_{i+1}\equiv b_{i} \pmod{q}$. Moreover, $v$ uniquely determines $\LL$.
\end{lemma}

\begin{proof}
Let $v \in \{1,\ldots,q - 1\}$ be such that $vb_{i+1}\equiv b_{i} \pmod{q}$ (recalling that $(b_{i+1},q)=1$) and define
\[\gamma_{\LL} \coloneqq \begin{pmatrix} a_{i+1}v-a_{i} & \tfrac{b_{i+1}v-b_{i}}{q} \\ a_{i+1}q & b_{i+1} \end{pmatrix}\in \Gamma_0(q).\]
Then the fact that $a_{i+1} b_i - a_i b_{i+1} = 1$ implies that
\[\gamma_{\LL} W_q \frac{a_{i+1}}{b_{i+1}} = \infty, \quad \gamma_{\LL} W_q \frac{a_{i}}{b_{i}} = \frac{v}{q},\]
as desired.

To show uniqueness, we suppose that $\gamma_{\LL}' W_q$ also maps $a_{i + 1}/b_{i + 1}$ to $\infty$ and $\frac{a_i}{b_i}$ to $\frac{v}{q}$. Then $\gamma_{\LL}' \gamma_{\LL}^{-1}\infty=\infty$, so that $\gamma_{\LL}' \gamma_{\LL}^{-1} \in \Gamma_{\infty}$, and hence there exists some $n \in \Z$ such that $\gamma_{\LL}' \gamma_{\LL}^{-1} = \begin{psmallmatrix} 1 & n \\ 0 & 1 \end{psmallmatrix}$. Since $\gamma_{\LL}' \gamma_{\LL}^{-1} \frac{v}{q} \in (0,1)$, we must have that $n = 0$, so that $\gamma_{\LL}' = \gamma_{\LL}$.

Finally, if we have two hyperbolic sides $\LL$ and $\LL'$ such that
\[\gamma_{\LL} W_q \LL = \gamma_{\LL}' W_q \LL',\]
then by applying $W_q^{-1}$ and using that $W_q$ normalises $\Gamma_0(q)$, we conclude that $\LL$ and $\LL'$ are $\Gamma_0(q)$-equivalent. This implies that $b_{i}b_{i^\ast+1}\equiv b_{i+1}b_{i^\ast} \pmod{q}$. But we also have $b_{i}b_{i^\ast}\equiv -b_{i+1}b_{i^\ast+1} \pmod{q}$, which implies $b_{i}^2+b_{i+1}^2\equiv 0 \pmod{q}$, and thus $\LL=\LL'$ is an even side, contrary to the assumption. 
 \end{proof}

We start by bounding the cuspidal case.

\begin{lemma}
\label{lem:bndintCusp}
Let $q$ be an odd prime and let $\LL$ be a hyperbolic side of $\PP(q)$. Then for $q_1 q_2 = q$, we have that
\begin{equation}
\label{eqn:sideL2boundMaass}
\frac{1}{q_2} \sum_{\substack{f \in \BB_0^{\ast}(\Gamma_0(q_1)) \\ |t_f|\leq T}} \left|\int_{\LL} (R_0 f)(q_2 z) \, \frac{dz}{\Im(z)}\right|^2 \ll_{\e} q^{\e} T^{3 + \e}.
\end{equation}
\end{lemma}
 
\begin{proof}
We make the change of variables $z \mapsto W_q^{-1} \gamma_{\LL}^{-1} z$, where $\gamma_{\LL} \in \Gamma_0(q)$ is as in \hyperref[lem:canonicalform]{Lemma \ref*{lem:canonicalform}}. As $g_{q_2} W_q^{-1} W_{q_1} \in \Gamma_0(q_1)$, we have via \hyperref[lem:canonicalform]{Lemma \ref*{lem:canonicalform}} that
\[\int_{\LL} (R_0 f)(q_2 z) \, \frac{dz}{\Im(z)} = \eta_f(q_1) \int_{\frac{v}{q}}^{i\infty} (R_0 f)(z) \, \frac{dz}{\Im(z)},\]
where $\eta_f(q_1) \in \{1,-1\}$ is the Atkin--Lehner eigenvalue of $f$ under the action of $W_{q_1}$. So upon defining
\[\Lambda\left(s,f,\frac{v}{q}\right) \coloneqq q^{s - \frac{1}{2}} \int_{\frac{v}{q}}^{i\infty} (R_0 f)(z) \Im(z)^{s - \frac{1}{2}} \, \frac{dz}{\Im(z)}\]
for $s \in \C$ and recalling \hyperref[lem:Maasortho]{Lemma \ref*{lem:Maasortho}}, we deduce that the left-hand side of \eqref{eqn:sideL2boundMaass} is equal to
\begin{equation}
\label{eqn:sideL2bound2}
\frac{1}{q_2} \sum_{\substack{f \in \BB_0^{\ast}(\Gamma_0(q_1)) \\ |t_f|\leq T}} \left|\Lambda\left(\frac{1}{2},f,\frac{v}{q}\right)\right|^2.
\end{equation}

Let $\overline{v} \in \{1,\ldots,q - 1\}$ be such that $v\overline{v} \equiv 1 \pmod{q}$. Then $(R_0 f)(\frac{v}{q} + \frac{iy}{q}) = - (R_0 f)(-\frac{\overline{v}}{q} + \frac{i}{qy})$ via the automorphy of $R_0 f$, namely $(R_0 f)(\gamma z) = j_{\gamma}(z)^2 (R_0 f)(z)$ with $\gamma = \begin{psmallmatrix} v & \frac{v \overline{v} - 1}{q} \\ q & \overline{v} \end{psmallmatrix} \in \Gamma_0(q)$ and $z = -\frac{\overline{v}}{q} + \frac{i}{qy}$. This yields the functional equation
\begin{equation}
\label{eqn:funceq}
\begin{split}
\Lambda\left(s,f,\frac{v}{q}\right) & = i \int_{1}^{\infty} \left((R_0 f)\left(\frac{v}{q} + \frac{iy}{q}\right) y^{s - \frac{1}{2}} - (R_0 f)\left(-\frac{\overline{v}}{q} + \frac{iy}{q}\right) y^{\frac{1}{2} - s}\right) \, \frac{dy}{y}	\\
& = -\Lambda\left(1 - s,f,-\frac{\overline{v}}{q}\right),
\end{split}
\end{equation}
noting that the integral represents an entire function. Moreover, for $\Re(s) > 1$, we may insert the Fourier expansion \eqref{eqn:Fkexppos} for $R_0 f$ in order to see that
\begin{multline}
\label{eqn:LambdaDirseries}
\Lambda\left(s,f,\frac{v}{q}\right) = - i \rho_f(1) q^{s - \frac{1}{2}} \sum_{n = 1}^{\infty} \frac{\lambda_f(n)}{n^s} e\left(\frac{nv}{q}\right) \int_{0}^{\infty} W_{1,it_f}(4\pi y) y^{s - \frac{1}{2}} \, \frac{dy}{y}	\\
+ i \epsilon_f \rho_f(1) q^{s - \frac{1}{2}} \sum_{n = 1}^{\infty} \frac{\lambda_f(n)}{n^s} e\left(-\frac{nv}{q}\right) \int_{0}^{\infty} \left(\frac{1}{4} + t_f^2\right) W_{-1,it_f}(4\pi y) y^{s - \frac{1}{2}} \, \frac{dy}{y}.
\end{multline}
From \eqref{eqn:Whittakeridentity} and \eqref{eqn:W2int}, the first integral above is
\begin{equation}
\label{eqn:firstint}
\pi^{-s} \Gamma\left(\frac{s + 1 + it_f}{2}\right) \Gamma\left(\frac{s + 1 - it_f}{2}\right) + \frac{1}{2} \left(s - \frac{1}{2}\right) \pi^{-s} \Gamma\left(\frac{s + it_f}{2}\right) \Gamma\left(\frac{s - it_f}{2}\right),
\end{equation}
while the second is
\begin{equation}
\label{eqn:secondint}
\pi^{-s} \Gamma\left(\frac{s + 1 + it_f}{2}\right) \Gamma\left(\frac{s + 1 - it_f}{2}\right) - \frac{1}{2} \left(s - \frac{1}{2}\right) \pi^{-s} \Gamma\left(\frac{s + it_f}{2}\right) \Gamma\left(\frac{s - it_f}{2}\right).
\end{equation}

With this in hand, we derive the approximate functional equation for $\Lambda(\frac{1}{2},f,\frac{v}{q})$, namely the identity
\begin{equation}
\label{eqn:approxfunceq}
\Lambda\left(\frac{1}{2},f,\frac{v}{q}\right) = \frac{i \rho_f(1)}{\sqrt{\pi}} \sum_{\pm} \mp \epsilon_f^{\frac{1 \mp 1}{2}} \sum_{n = 1}^{\infty} \frac{\lambda_f(n)}{\sqrt{n}} \left(e\left(\pm \frac{nv}{q}\right) + e\left(\mp \frac{n\overline{v}}{q}\right)\right) V_{\pm}\left(\frac{n}{q},t_f\right),
\end{equation}
where for $x > 0$, $t \in \R \cup i[-\frac{7}{64},\frac{7}{64}]$, and $\sigma > 0$,
\begin{multline}
\label{eqn:Vpmxt}
V_{\pm}(x,t) \coloneqq \frac{1}{2\pi i} \int_{\sigma - i\infty}^{\sigma + i\infty} x^{-s} e^{s^2}	\\
\times \left(\Gamma\left(\frac{s}{2} + \frac{3}{4} + \frac{it}{2}\right) \Gamma\left(\frac{s}{2} + \frac{3}{4} - \frac{it}{2}\right) \pm \frac{s}{2} \Gamma\left(\frac{s}{2} + \frac{1}{4} + \frac{it}{2}\right) \Gamma\left(\frac{s}{2} + \frac{1}{4} - \frac{it}{2}\right)\right) \, \frac{ds}{s}.
\end{multline}
To see this, consider
\begin{equation}
\label{eqn:approxfunceqdifference}
\frac{1}{2\pi i} \int_{\sigma - i\infty}^{\sigma + i\infty} \Lambda\left(\frac{1}{2} + s,f,\frac{v}{q}\right) e^{s^2} \, \frac{ds}{s} - \frac{1}{2\pi i} \int_{-\sigma - i\infty}^{-\sigma + i\infty} \Lambda\left(\frac{1}{2} + s,f,\frac{v}{q}\right) e^{s^2} \, \frac{ds}{s}
\end{equation}
with $\sigma > \frac{1}{2}$. This is equal to the left-hand side of \eqref{eqn:approxfunceq} simply by shifting the first contour of integration from $\Re(s) = \sigma$ to $\Re(s) = -\sigma$, which picks up a residue at $s = 0$ equal to $\Lambda(\frac{1}{2},f,\frac{v}{q})$. On the other hand, this is equal to the right-hand side of \eqref{eqn:approxfunceq} by first using the functional equation \eqref{eqn:funceq} for the second integral and making the change of variables $s \mapsto -s$, and then inserting the expression \eqref{eqn:LambdaDirseries} for $\Lambda(\frac{1}{2} + s,f,\frac{v}{q})$ and $\Lambda(\frac{1}{2} + s,f,-\frac{\overline{v}}{q})$, interchanging the order of summation and integration, and recalling \eqref{eqn:firstint} and \eqref{eqn:secondint}.

Next, we note that by \eqref{eqn:L2rho1} and Stirling's formula, we have that for any $A > 0$,
\[\rho_f(1) V_{\pm}(x,t_f) \ll_{\e} \begin{dcases*}
q_1^{-\frac{1}{2} + \e} (1 + |t_f|)^{\frac{1}{2} + \e} & for $x \leq 1 + |t_f|$,	\\
q_1^{-\frac{1}{2} + \e} (1 + |t_f|)^{A + \frac{1}{2} + \e} x^{-A} & for $x \geq 1 + |t_f|$.
\end{dcases*}\]
Here we have used the fact that $L(1,\ad f) \gg_{\e} (q_1 (1 + |t_f|))^{-\e}$; additionally, for $x \leq 1 + |t_f|$, we have shifted the contour to $\Re(s) = -\frac{25}{64} + \e$, while we have shifted the contour to $\Re(s) = A$ for $x \geq 1 + |t_f|$. This bound shows that the portion of the sum over $n \in \N$ in \eqref{eqn:approxfunceq} for which $n \geq (q (1 + |t_f|))^{1 + \e}$ is negligibly small.

The result now follows by inserting the approximate functional equation \eqref{eqn:approxfunceq} into \eqref{eqn:sideL2bound2}, dividing the sum over $n \in \N$ in \eqref{eqn:approxfunceq} into dyadic ranges, and applying the spectral large sieve, namely the well-known bound
\[\sum_{\substack{f \in \BB_0^{\ast}(\Gamma_0(q_1)) \\ |t_f| \leq T}} \left|\sum_{n \leq N} a_n \lambda_f(n)\right|^2 \ll_{\e} (q_1 T N)^{\e} (q_1 T^2 + N) \sum_{n \leq N} |a_n|^2.\qedhere\]
\end{proof}

An analogous result holds for Eisenstein series.

\begin{lemma}
\label{lem:bndintEis}
Let $q$ be an odd prime and let $\LL$ be a hyperbolic side of $\PP(q)$. Then we have that
\[\frac{1}{q} \int_{-T}^{T} \left|\int^\ast_{\LL} (R_0 E)\left(qz,\frac{1}{2} + it\right) \, \frac{dz}{\Im(z)}\right|^2\,dt \ll_{\e} q^{\e} T^{3 + \e}.\]
\end{lemma}

\begin{proof}
For $\Re(s) > 1$, we define
\[\Lambda\left(s,E_t,\frac{v}{q}\right) \coloneqq q^{s - \frac{1}{2}} \int_{\frac{v}{q}}^{i\infty} (\widetilde{R_0 E})\left(z,\frac{1}{2} + it\right) \Im(z)^{s - \frac{1}{2}} \, \frac{dz}{\Im(z)},\]
where $(\widetilde{R_0 E})(z,\frac{1}{2} + it)$ is as in \eqref{eqn:tildeR0E}. The automorphy of $(R_0 E)(z,\frac{1}{2} + it)$ implies the functional equation
\begin{align*}
\Lambda\left(s,E_t,\frac{v}{q}\right) & = i \int_{1}^{\infty} \left((\widetilde{R_0 E})\left(\frac{v}{q} + \frac{iy}{q},\frac{1}{2} + it\right) y^{s - \frac{1}{2}} - (\widetilde{R_0 E})\left(-\frac{\overline{v}}{q} + \frac{iy}{q},\frac{1}{2} + it\right) y^{\frac{1}{2} - s}\right) \, \frac{dy}{y}		\\
& \qquad + i \sum_{\pm} \left(\frac{1}{2} \pm it\right) \frac{\xi(1 \pm 2it)}{\xi(1 + 2it)} q^{-\frac{1}{2} \mp it} \left(\frac{1}{s \pm it} - \frac{1}{1 - s \pm it}\right)	\\
& = -\Lambda\left(1 - s,E_t,-\frac{\overline{v}}{q}\right).
\end{align*}
Note that the first integral is entire, while the second term, which arises from the constant terms in the Fourier expansion of $(R_0 E)(z,\frac{1}{2} + it)$, is meromorphic with simple poles at $s = 1 \pm it$ and $s = \mp it$ and a zero at $s = 1/2$. Furthermore, for $\Re(s) > 1$, we have that
\begin{multline*}
\Lambda\left(s,E_t,\frac{v}{q}\right) = - i \frac{q^{s - \frac{1}{2}}}{\xi(1 + 2it)} \sum_{n = 1}^{\infty} \frac{\lambda(n,t)}{n^s} e\left(\frac{nv}{q}\right) \int_{0}^{\infty} W_{1,it}(4\pi y) y^{s - \frac{1}{2}} \, \frac{dy}{y}	\\
+ i \frac{q^{s - \frac{1}{2}}}{\xi(1 + 2it)} \sum_{n = 1}^{\infty} \frac{\lambda(n,t)}{n^s} e\left(-\frac{nv}{q}\right) \int_{0}^{\infty} \left(\frac{1}{4} + t^2\right) W_{-1,it}(4\pi y) y^{s - \frac{1}{2}} \, \frac{dy}{y};
\end{multline*}
these integrals are explicitly calculated in \eqref{eqn:firstint} and \eqref{eqn:secondint}. Finally, the approximate functional equation in this setting reads
\[\Lambda\left(\frac{1}{2},E_t,\frac{v}{q}\right) = \frac{i}{\sqrt{\pi} \xi(1 + 2it)} \sum_{\pm} \mp \sum_{n = 1}^{\infty} \frac{\lambda(n,t)}{\sqrt{n}} \left(e\left(\pm \frac{nv}{q}\right) + e\left(\mp \frac{n\overline{v}}{q}\right)\right) V_{\pm}\left(\frac{n}{q},t\right),\]
where $V_{\pm}(x,t)$ is as in \eqref{eqn:Vpmxt}. This follows by the same method as for Hecke--Maa\ss{} newforms, namely considering the integral \eqref{eqn:approxfunceqdifference} with $f$ replaced by $E_t$; although there are additional poles of the integrand at $s = 1/2 \pm it$ and $s = -1/2 \mp it$, these ultimately cancel out.

The same approach as for the Hecke--Maa\ss{} form case shows that the left-hand side of \eqref{eqn:sideL2boundMaass} is equal to
\[\frac{1}{q} \int_{-T}^{T} \left|\Lambda\left(\frac{1}{2},E_t,\frac{v}{q}\right)\right|^2\, dt.\]
The result once more follows from the spectral large sieve, which states that
\[\int_{-T}^{T} \left|\sum_{n \leq N} a_n \lambda(n,t)\right|^2 \, dt \ll_{\e} (TN)^{\e} (T^2 + N) \sum_{n \leq N} |a_n|^2.\qedhere\]
\end{proof}

\subsection{Proof of \texorpdfstring{\hyperref[thm:level]{Theorem \ref*{thm:level}}}{Theorem \ref{thm:level}}}

\begin{proof}[Proof of {\hyperref[thm:level]{Theorem \ref*{thm:level}}}]
Let $u : \Hb \to \R$ be a smooth compactly supported function such that no two points of the support of $u$ are $\Gamma$-equivalent. This gives rise to a smooth compactly supported function $U(z) : \Gamma \backslash \Hb \to \R$ given by
\[U(z) \coloneqq \sum_{\gamma \in \Gamma} u(\gamma z).\]
For each odd prime $q$, choose $\omega_q \in \Gamma / \Gamma_0(q)$ and define
\[U_q(z) \coloneqq \sum_{\gamma \in \Gamma_0(q)} u(\omega_q \gamma z) = \begin{dcases*}
U(z) & if $z \in \omega_q^{-1} \Gamma \backslash \Hb$,	\\
0 & otherwise.
\end{dcases*}\]
Via \hyperref[lem:Maasortho]{Lemmata \ref*{lem:Maasortho}} and \ref{lem:Eisortho}, the spectral expansion of $U_q(z)$ reads
\begin{multline*}
U_q(z) = \frac{1}{\vol(\Gamma_0(q) \backslash \Hb)} \int_{\Gamma \backslash \Hb} U(z) \, d\mu(z) + \sum_{q_1 q_2 = q} \sum_{\ell \mid q_2} \sum_{f \in \BB_0^{\ast}(\Gamma_0(q_1))} \left\langle U_q, f_{\ell}\right\rangle_q f_{\ell}(z)	\\
+ \frac{1}{4\pi} \sum_{\ell \mid q} \int_{-\infty}^{\infty} \left\langle U_q, E_{\ell}\left(\cdot,\frac{1}{2} + it\right)\right\rangle_q E_{\ell}\left(z,\frac{1}{2} + it\right) \, dt.
\end{multline*}
Then by character orthogonality and \hyperref[lem:oldforms]{Lemma \ref*{lem:oldforms}}, the quantity
\[\frac{\vol(\Gamma_0(q) \backslash \Hb)}{\sum_{A \in G_D} \vol(\FF_A(q))} \sum_{A \in G_D} \int_{\FF_A(q)} U_q(z) \, d\mu(z),\]
where $G_D = C (\Cl_D^+)^2 \in \Gen_D$ is a genus with $C \in \Cl_D^+$, is equal to the sum of the contribution from the residual spectrum,
\[\int_{\Gamma \backslash \Hb} U(z) \, d\mu(z),\]
the contribution from the cuspidal spectrum,
\begin{multline*}
\frac{\vol(\Gamma_0(q) \backslash \Hb)}{\sum_{A \in G_D} \vol(\FF_A(q))} \frac{1}{2^{\omega(D) - 1}} \sum_{\chi \in \widehat{\Gen_D}} \chi(C) \sum_{q_1 q_2 = q} \sum_{\ell_1\ell_2 \mid q_2} \frac{\sqrt{\varphi(\ell_1\ell_2)} \nu(\ell_1) \mu(\ell_2) \lambda_f(\ell_2)}{\ell_1^{3/2} \ell_2 \sqrt{\nu(q_2)}} \\
\times \sum_{f \in \BB_0^{\ast}(\Gamma_0(q_1))} \sqrt{L_{\ell_1\ell_2}(1,\ad f)} \left\langle U_q, f_{\ell_1\ell_2}\right\rangle_q W^{\ell_1,q}_{\chi,f},
\end{multline*}
and the contribution from the continuous spectrum,
\begin{multline*}
\frac{\vol(\Gamma_0(q) \backslash \Hb)}{ \sum_{A \in G_D} \vol(\FF_A(q))} \frac{1}{2^{\omega(D) - 1}}\sum_{\chi \in \widehat{\Gen_D}} \chi(C) \sum_{\ell_1\ell_2 \mid q} \frac{\nu(\ell_1) \mu(\ell_2) \lambda(\ell_2,t)}{\ell_1 \sqrt{\ell_2 \nu(q)}} \\
\times \frac{1}{4\pi} \int_{-\infty}^{\infty} \left|\zeta_{\ell_1\ell_2}(1 + 2it)\right| \left\langle U_q, E_{\ell_1\ell_2}\left(\cdot,\frac{1}{2} + it\right)\right\rangle_q W^{\ell_1,q}_{\chi,t} \, dt.
\end{multline*}
Here $\widehat{\Gen_D}$ is the group of genus characters, namely characters that are trivial on $(\Cl_D^+)^2$, while the Weyl sums $W^{\ell,q}_{\chi,f}$ and $W^{\ell,q}_{\chi,t}$ are as in \eqref{eqn:WeylMaassold} and \eqref{eqn:WeylEisold}.

We fix $\e > 0$ and break up the contributions from the cuspidal and continuous spectra into two parts: the contributions for which $|t_f|,|t| \leq D^{\e}$ and their complements. These complementary contributions are negligibly small, since for any nonnegative integer $A$, we have that
\begin{equation}
\label{eqn:intbypartsbounds}
\left\langle U_q, f_{\ell}\right\rangle_q \ll_{U,A} (1 + |t_f|)^{-A}, \qquad \left\langle U_q, E_{\ell}\left(\cdot,\frac{1}{2} + it\right)\right\rangle_q \ll_{U,A} (1 + |t|)^{-A},
\end{equation}
which follows by repeated integration by parts (see, for example \cite[(6.1)]{LMY13}). Thus from the Cauchy--Schwarz inequality,
\[\left|\frac{\vol(\Gamma_0(q) \backslash \Hb)}{\sum_{A \in G_D} \vol(\FF_A(q))} \sum_{A \in G_D} \int_{\FF_A(q)} U_q(z) \, d\mu(z) - \int_{\Gamma \backslash \Hb} U(z) \, d\mu(z)\right|^2\]
is bounded by the sum of two terms plus a negligibly small term. These two terms are those for which $|t_f| \leq D^{\e}$ and $|t| \leq D^{\e}$. By Bessel's inequality, these two terms are
\begin{equation}
\label{eqn:firstterm}
\ll_{\e} D^{\e} q^{\e} \left|\left\langle U, U\right\rangle_1\right|^2 \left(\frac{\vol(\Gamma_0(q) \backslash \Hb)}{\sum_{A \in G_D} \vol(\FF_A(q))}\right)^2 \sum_{\chi \in \widehat{\Gen_D}} \sum_{q_1 \ell \mid q} \sum_{\substack{f \in \BB_0^{\ast}(\Gamma_0(q_1)) \\ |t_f| \leq D^{\e}}} \frac{1}{\ell} \left|W^{\ell,q}_{\chi,f}\right|^2
\end{equation}
and
\begin{equation}
\label{eqn:secondterm}
\ll_{\e} D^{\e} q^{\e} \left|\left\langle U, U\right\rangle_1\right|^2 \left(\frac{\vol(\Gamma_0(q) \backslash \Hb)}{\sum_{A \in G_D} \vol(\FF_A(q))}\right)^2 \sum_{\chi \in \widehat{\Gen_D}} \sum_{\ell \mid q} \frac{1}{\ell} \int_{-D^{\e}}^{D^{\e}} \left|W^{\ell,q}_{\chi,t}\right|^2 \, dt
\end{equation}
respectively.

We recall that the set of genus characters $\chi \in \widehat{\Gen_D}$ is indexed by ordered pairs of fundamental discriminants $(D_1,D_2)$ such that $D_1 D_2 = D$; each genus character $\chi$ is then associated to the pair of primitive quadratic Dirichlet characters $\chi_{D_1}$ and $\chi_{D_2}$ modulo $|D_1|$ and $|D_2|$ respectively. Note that $\Theta_{\chi}$ is an Eisenstein series, so that $L(s,f \otimes \Theta_{\chi}) = L(s,f \otimes \chi_{D_1}) L(s,f \otimes \chi_{D_2})$ and $L(s,\Theta_{\chi}) = L(s,\chi_{D_1}) L(s,\chi_{D_2})$.

For $\ell q_1 \mid q$ and $f \in \BB_0^{\ast}(\Gamma_0(q_1))$, we deduce from \hyperref[prop:MaassnewformWeyl]{Proposition \ref*{prop:MaassnewformWeyl}} and Stirling's formula \eqref{eqn:Stirlingboundscchif} applied to the geodesic term, as well as the Cauchy--Schwarz inequality applied to the topological terms, that
\begin{multline}
\label{eqn:Wellqchifupperbound}
\left|W^{\ell,q}_{\chi,f}\right|^2 \ll_\e \frac{\sqrt{D}}{q_1 (1 + |t_f|)^3} \frac{L\left(\frac{1}{2},f \otimes \chi_{D_1}\right) L\left(\frac{1}{2}, f \otimes \chi_{D_2}\right)}{L(1,\ad f)}	\\
+\delta_{q_1\ell,q} \frac{\sqrt{D}}{(1 + |t_f|)^4} \sum_{j=1}^{2g} \left| \int_{\CC_j} (R_0 f)(\ell z) \, \frac{dz}{\Im (z)}\right|^2 \sum_{j=1}^{2g} \sum_{h\in \BB_2^{\hol}(\Gamma_0(q))} \sum_{\pm} |\langle v^\pm_{h}, \omega_j\rangle_\mathrm{cap}|^2 \\
\times \sum_{h\in \BB_2^{\hol}(\Gamma_0(q))} L\left(\frac{1}{2}, h \otimes \chi_{D_1}\right) L\left(\frac{1}{2}, h \otimes \chi_{D_2}\right),
\end{multline}
and similarly
\begin{multline}
\label{eqn:Wellqchitupperbound}
\left|W^{\ell,q}_{\chi,t}\right|^2 \ll_\e \frac{\sqrt{D}}{(1 + |t|)^3} \left|\frac{L\left(\frac{1}{2} + it,\chi_{D_1}\right) L\left(\frac{1}{2} + it,\chi_{D_2}\right)}{\zeta(1 + 2it)}\right|^2 \\
+ \delta_{\ell,q} \sqrt{D} \sum_{j=1}^{2g} \left| \int_{\CC_j} (R_0 E)\left(\ell z,\frac{1}{2} + it\right) \, \frac{dz}{\Im (z)}\right|^2 \sum_{j=1}^{2g} \sum_{h\in \BB_2^{\hol}(\Gamma_0(q))} \sum_{\pm} \left|\langle v^\pm_{h}, \omega_j\rangle_\mathrm{cap}\right|^2 \\
\times \sum_{h\in \BB_2^{\hol}(\Gamma_0(q))} L\left(\frac{1}{2}, h \otimes \chi_{D_1}\right) L\left(\frac{1}{2}, h \otimes \chi_{D_2}\right).
\end{multline}

We first bound the portions of \eqref{eqn:firstterm} and \eqref{eqn:secondterm} coming from the first terms on the right-hand sides of \eqref{eqn:Wellqchifupperbound} and \eqref{eqn:Wellqchitupperbound}, namely the geodesic terms. We use H\"{o}lder's inequality with exponents $(3,3,3)$ together with the Weyl law upper bound
\[\sum_{\substack{f \in \BB_0^{\ast}(\Gamma_0(q_1)) \\ |t_f| \leq T}} 1 \ll T^2 q_1\]
as well as the third moment bounds
\begin{align}
\label{eqn:cubicMaass}
\sum_{\substack{f \in \BB_0^{\ast}(\Gamma_0(q_1)) \\ T - 1 \leq |t_f| \leq T}} \frac{L\left(\frac{1}{2}, f \otimes \chi_D\right)^3}{L(1,\ad f)} \ll_{\e} D^{1 + \e} q_1^{1 + \e} T^{1 + \e},	\\
\label{eqn:cubicEis}
\int\limits_{T - 1 \leq |t| \leq T} \left|\frac{L\left(\frac{1}{2} + it,\chi_D\right)^3}{\zeta(1 + 2it)}\right|^2 \, dt \ll_{\e} D^{1 + \e} T^{1 + \e}
\end{align}
The sixth moment bound \eqref{eqn:cubicEis} is \cite[Theorem 1]{You17}. The third moment bound \eqref{eqn:cubicMaass} with $T^{1 + \e}$ replaced by $T^A$ for some unspecified large constant $A$ was proven for \emph{holomorphic} cusp forms by Petrow and Young \cite[Theorem 1]{PY19}; it was then improved to the form \eqref{eqn:cubicMaass} in \cite[Theorem 4.1]{AW23} and in a more general form in \cite[Theorem 11.1]{GHLN24}, which both also recover the bound \eqref{eqn:cubicEis}. Combining this with the fact that $|\widehat{\Gen_D}| = 2^{\omega(D) - 1} \ll_{\e} D^{\e}$, and that the lower bound for the volume $\vol(\FF_A(q))$ coming from \eqref{eqn:volbound}, we obtain the bounds $O_{\e}(D^{-\frac{1}{6} + \e} q^{2 + \e})$ for these terms.

We next bound the portions of \eqref{eqn:firstterm} and \eqref{eqn:secondterm} coming from the second terms on the right-hand sides of \eqref{eqn:Wellqchifupperbound} and \eqref{eqn:Wellqchitupperbound}, namely the topological terms. For the innermost sum over $h \in \BB_2^{\hol}(\Gamma_0(q))$, we use H\"{o}lder's inequality with exponents $(3,3,3)$ together with the dimension upper bound $|\BB_2^{\hol}(\Gamma_0(q))| \ll q$ and the third moment bound
\[\sum_{h \in \BB_2^{\hol}(\Gamma_0(q))} L\left(\frac{1}{2}, h \otimes \chi_D\right)^3 \ll_{\e} D^{1 + \e} q^{1 + \e}\]
due to Petrow and Young \cite[Theorem 1]{PY19} (see also \cite[Theorem 11.1]{GHLN24}). Next, we want to apply the $L^2$-bound for the cap product pairing derived by the second author \cite[Theorem 6.1]{Nor23} which applies to (dual bases of) \emph{basic bases} of homology in the terminology of \cite{Nor23}. These are bases of the vector space $H_1(Y_0(q), \R)$ consisting of classes containing the geodesic connecting $z$ and $\gamma z$ where $\gamma\in \Gamma_0(q)$ and $\gamma \infty=\tfrac{v}{q}$ with $0<v<q$. \hyperref[lem:canonicalform]{Lemma \ref*{lem:canonicalform}} implies that indeed
\[W_q[\CC_1],\ldots , W_q[\CC_{2g}]\]
is the projection to the homology $H_1(X_0(q),\R)$ of a basic basis, where $W_q$ denotes the Atkin--Lehner involution acting on homology. Since $W_q$ is self-adjoint with respect to the cap product pairing and $W_qh(z)=\eta_h(q) h(z)$, where $\eta_h(q) \in \{1,-1\}$ is the Atkin--Lehner eigenvalue of $h\in \mathcal{B}_2^\mathrm{hol}(\Gamma_0(q))$, we get
\[|\langle v_h^\pm, \omega_j \rangle_\mathrm{cap}|^2=|\langle v_h^\pm, W_q \omega_j \rangle_\mathrm{cap}|^2.\]
Since $W_q$ commutes with taking dual bases (using that $W_q$ is self-adjoint), we conclude by \cite[Theorem 6.1]{Nor23} the bound
\[\sum_{j = 1}^{2g} \sum_{h\in \BB_2^{\hol}(\Gamma_0(q))} \sum_{\pm} \left|\langle v^\pm_{h}, \omega_j\rangle_\mathrm{cap}\right|^2 \ll_{\e} q^{2 + \e}.\]
As discussed in \cite[Section 6]{Nor23}, this should be viewed as a homological version of the sup-norm problem.

Finally, we apply \hyperref[lem:bndintCusp]{Lemmata \ref*{lem:bndintCusp}} and \ref{lem:bndintEis} to bound the remaining second moment of integrals along the hyperbolic sides $\CC_j$. Combined, this yields the bounds $O_{\e}(D^{-\frac{1}{6} + \e} q^{6 + \e})$ for these terms.

We deduce that
\begin{equation}
\label{eqn:Uqintequidistribution}
\frac{\vol(\Gamma_0(q) \backslash \Hb)}{\sum_{A \in G_D} \vol(\FF_A(q))} \sum_{A \in G_D} \int_{\FF_A(q)} U_q(z) \, d\mu(z) = \int_{\Gamma \backslash \Hb} U(z) \, d\mu(z) + O_{U,\e}\left(D^{-\frac{1}{12} + \e} q^{3 + \e}\right).
\end{equation}
We improve the error term in \eqref{eqn:Uqintequidistribution} to $O_{\e}(D^{-\frac{1}{4} + \e} q^{3 + \e})$ under the assumption of the generalised Lindel\"{o}f hypothesis by invoking the conditional bounds
\begin{align*}
\frac{L\left(\frac{1}{2},f \otimes \chi_{D_1}\right) L\left(\frac{1}{2}, f \otimes \chi_{D_2}\right)}{L(1,\ad f)} & \ll_{\e} (D q_1 (1 + |t_f|))^{\e},	\\
L\left(\frac{1}{2},h \otimes \chi_{D_1}\right) L\left(\frac{1}{2}, h \otimes \chi_{D_2}\right) & \ll_{\e} (Dq)^{\e},	\\
\left|\frac{L\left(\frac{1}{2} + it,\chi_{D_1}\right) L\left(\frac{1}{2} + it,\chi_{D_2}\right)}{\zeta(1 + 2it)}\right|^2 & \ll_{\e} (D (1 + |t|))^{\e},
\end{align*}
and using the Weyl law instead of applying H\"{o}lder's inequality.

We are almost able to deduce \hyperref[thm:level]{Theorem \ref*{thm:level}} upon taking $u(z)$ to closely approximate a fundamental domain of $\Gamma \backslash \Hb$. There is one remaining obstacle, namely that we are assuming that $u$ is \emph{compactly supported}, and so we have not precluded the possibility of the escape of mass in this equidistribution problem.

We circumvent this by performing a dyadic partition of unity and choosing $u$ to be a function of the form $u(z) \coloneqq \delta_{\Re(z) \in [0,1]} \Psi(\frac{\Im(z)}{Y})$, with $Y > 1$ a large parameter and $\Psi : (0,\infty) \to \R$ a smooth bump function equal to $1$ on $[1,2]$ and vanishing outside $[\frac{1}{2},\frac{5}{2}]$. We use the same framework as above except that the bounds \eqref{eqn:intbypartsbounds} are replaced by
\[\left\langle U_q, f_{\ell}\right\rangle_q = 0, \qquad \left\langle U_q, E_{\ell}\left(\cdot,\frac{1}{2} + it\right)\right\rangle_q \ll_{\Psi} Y^{-\frac{1}{2}} (1 + |t|)^{-A}\]
for any $A > 0$. These follow by inserting the Fourier expansions \eqref{eqn:Fourier} and \eqref{eqn:FourierEis}: the cusp form term vanishes since there is no constant term in the Fourier expansion, while for the Eisenstein case, we make the change of variables $y \mapsto Yy$ and then repeatedly integrate by parts. Thus there is no contribution from the cuspidal spectrum, while we bound the contribution from the continuous spectrum using the same bounds as above. In this way, we find that
\[\frac{\vol(\Gamma_0(q) \backslash \Hb)}{\sum_{A \in G_D} \vol(\FF_A(q))} \sum_{A \in G_D} \int_{\FF_A(q)} U_q(z) \, d\mu(z) = \int_{\Gamma \backslash \Hb} U(z) \, d\mu(z) + O_{U,\e}\left(Y^{-1} D^{-\frac{1}{12} + \e} q^{\frac{3}{2} + \e}\right).\]
This yields the desired result.
\end{proof}

\begin{remark}
Instead of proceeding via H\"{o}lder's inequality together with bounds for the third moments of $L$-functions, it is tempting to try to directly bound the moments
\begin{gather*}
\sum_{\substack{f \in \BB_0^{\ast}(\Gamma_0(q_1)) \\ |t_f| \leq D^{\e}}} \frac{L\left(\frac{1}{2},f \otimes \chi_{D_1}\right) L\left(\frac{1}{2}, f \otimes \chi_{D_2}\right)}{L(1,\ad f)},	\\
\sum_{h \in \BB_2^{\hol}(\Gamma_0(q))} L\left(\frac{1}{2},h \otimes \chi_{D_1}\right) L\left(\frac{1}{2},h \otimes \chi_{D_2}\right),	\\
\int_{-D^{\e}}^{D^{\e}} \left|\frac{L\left(\frac{1}{2} + it,\chi_{D_1}\right) L\left(\frac{1}{2} + it,\chi_{D_2}\right)}{\zeta(1 + 2it)}\right|^2 \, dt,
\end{gather*}
via approximate functional equations, the Kuznetsov and Petersson formul\ae{}, and the Vorono\u{\i} summation formula (or equivalently two applications of the Poisson summation formula). Work of Holowinsky and Templier \cite[Theorem 1]{HT14} suggests that it should be possible via this approach to bound each of these by $O_{\e}(q^{1 + \e} + D^{\frac{1}{2} + \e})$ (cf.\ \cite[Remark 6.2]{HK20}). Unfortunately, this falls just shy of being sufficient for our purposes; we instead require a bound of the form $O(D^{\frac{1}{2} - \delta})$ for some $\delta > 0$. This is a level-aspect analogue of the obstacle discussed in \cite[Section 8]{HR22} connecting small-scale equidistribution to sub-Weyl subconvexity.
\end{remark}

\begin{remark}
Liu, Masri, and Young \cite[Theorem 1.4]{LMY13} have proven the equidistribution of level $q$ Heegner points in translates $\omega_q^{-1} \Gamma \backslash \Hb$ for odd negative fundamental discriminants $D$ and for primes $q$ that split in $\Q(\sqrt{D})$ as $q(-D)$ tends to infinity under the proviso that $q \leq (-D)^{\delta}$ for some fixed $\delta < \frac{1}{20}$. The method of proof of \hyperref[thm:level]{Theorem \ref*{thm:level}} may be used to strengthen this result to relax the condition on $q$ to be squarefree, $D$ to be any negative fundamental discriminant (not necessarily odd), and $q \leq (-D)^{\delta}$ for some fixed $\delta < \frac{1}{12}$. Moreover, an analogous result can also be shown to hold for the equidistribution of level $q$ closed geodesics. The input for this improvement is the usage of the Weyl-strength third moment bounds \eqref{eqn:cubicMaass} and \eqref{eqn:cubicEis} in place of Burgess-strength pointwise bounds
\[L\left(\frac{1}{2}, f \otimes \chi_D\right) \ll_{t_f,\e} D^{\frac{3}{8} + \e} q^{\frac{1}{2} + \e}, \qquad L\left(\frac{1}{2} + it,\chi_D\right) \ll_{t,\e} D^{\frac{3}{16} + \e}.\]
\end{remark}

\section{Sparse Equidistribution in the Subgroup Aspect}
\label{sect:sparseproof}

\begin{proof}[Proof of {\hyperref[thm:subgroup]{Theorem \ref*{thm:subgroup}}}]
Via the Weyl equidistribution criterion together with the lower bound \eqref{eqn:volbound} for $\vol(\FF_A(q))$ and the orthonormal bases for the cuspidal and continuous spectra given in \hyperref[lem:Maasortho]{Lemmata \ref*{lem:Maasortho}} and \ref{lem:Eisortho}, it suffices to show that for each $q_1 \mid q$, $\ell \mid \frac{q}{q_1}$, and $f \in \BB_0^{\ast}(\Gamma_0(q_1))$, and for each $t \in \R$,
\begin{align*}
\sum_{A \in CH} \int_{\FF_A(q)} f(\ell z) \, d\mu(z) & = o_{q,t_f}\left(\frac{|H|}{h_D^+} \frac{\sqrt{D} L(1,\chi_D)}{\log D}\right),	\\
\sum_{A \in CH} \int_{\FF_A(q)} E\left(\ell z,\frac{1}{2} + it\right) \, d\mu(z) & = o_{q,t}\left(\frac{|H|}{h_D^+} \frac{\sqrt{D} L(1,\chi_D)}{\log D}\right),
\end{align*}
where we have used the narrow class number formula $h_D^+ \log \epsilon_D = \sqrt{D} L(1,\chi_D)$. Via character orthogonality, these expressions are respectively equal to
\[\frac{|H|}{h_D^+} \sum_{\chi \in H^{\perp}} \overline{\chi}(C) W^{\ell,q}_{\chi,f}, \qquad \frac{|H|}{h_D^+} \sum_{\chi \in H^{\perp}} \overline{\chi}(C) W^{\ell,q}_{\chi,t},\]
where the Weyl sums $W^{\ell,q}_{\chi,f}$ and $W^{\ell,q}_{\chi,t}$ are as in \eqref{eqn:WeylMaassold} and \eqref{eqn:WeylEisold}. From \hyperref[lem:oldforms]{Lemma \ref*{lem:oldforms}} and \hyperref[prop:MaassnewformWeyl]{Propositions \ref*{prop:MaassnewformWeyl}} and \ref{prop:EisnewformWeyl}, we have that
\begin{align*}
\left|W^{\ell,q}_{\chi,f}\right|^2 & \ll_{q,t_f} \sqrt{D}\left( L\left(\frac{1}{2}, f \otimes \Theta_{\chi}\right) + \sum_{h\in \BB_2^\hol(q)} \left(L\left(\frac{1}{2}, h \otimes \Theta_{\chi}\right) + L\left(\frac{1}{2}, h \otimes \Theta_{\overline{\chi}}\right)\right)\right),	\\
\left|W^{\ell,q}_{\chi,t}\right|^2 & \ll_{q,t} \sqrt{D}\left(\left|L\left(\frac{1}{2} + it, \Theta_{\chi}\right)\right|^2 + \sum_{h\in \BB_2^\hol(q)} \left(L\left(\frac{1}{2}, h \otimes \Theta_{\chi}\right) + L\left(\frac{1}{2}, h \otimes \Theta_{\overline{\chi}}\right)\right)\right).
\end{align*}
\hyperref[thm:subgroup]{Theorem \ref*{thm:subgroup}} thereby holds unconditionally for $\delta < \frac{1}{2826}$ due to the subconvexity estimates
\begin{align}
\label{eqn:HMsubconvex}
L\left(\frac{1}{2},f \otimes \Theta_{\chi}\right) & \ll_{q_1,t_f,\e} D^{\frac{1}{2} - \frac{1}{1413} + \e},	\\
\label{eqn:Msubconvex}
L\left(\frac{1}{2},h \otimes \Theta_{\chi}\right) & \ll_{q} D^{\frac{1}{2} - \frac{1}{1057}}\\
\label{eqn:BHMsubconvex}
L\left(\frac{1}{2} + it, \Theta_{\chi}\right) & \ll_t D^{\frac{1}{4} - \frac{1}{1889}},
\end{align}
combined with the fact that $|H^{\perp}| = \frac{h_D^+}{|H|}$, and the (ineffective) Siegel bound $L(1,\chi_D) \gg_{\e} D^{-\e}$. When $\chi$ is not a genus character, so that $\Theta_{\chi}$ is a cusp form, the first bound above is due to Harcos and Michel \cite[Theorem 1]{HM06} (see additionally \cite[Theorem 1.3]{Har11}), the second is due to Michel \cite[Theorem 2]{Mic04}, while the third is due to Blomer, Harcos, and Michel \cite[Theorem 2]{BHM07}. When $\chi$ is a genus character, then stronger bounds are known via the work of Petrow and Young \cite[Theorem 1]{PY19}.

We obtain \hyperref[thm:subgroup]{Theorem \ref*{thm:subgroup}} for $\delta < \frac{1}{4}$ under the assumption of the generalised Lindel\"{o}f hypothesis, since this implies the stronger bounds $|W_{\chi,f}^{\ell,q}|^2 \ll_{q,t_f,\e} \ll D^{\frac{1}{2} + \e}$ and $|W_{\chi,t}^{\ell,q}|^2 \ll_{q,t,\e} D^{\frac{1}{2} + \e}$.
\end{proof}

\begin{remark}
Improvements of the bound \eqref{eqn:BHMsubconvex} exist in the literature (see, for example, \cite[Theorem 1]{BlKh19} for $D$ prime); the obstacles in unconditionally enlarging the range of $\delta$ in \hyperref[thm:subgroup]{Theorem \ref*{thm:subgroup}} are improvements of the bounds \eqref{eqn:HMsubconvex} and \eqref{eqn:Msubconvex}.
\end{remark}

In \cite[Section 4]{DIT16}, it is observed that \hyperref[thm:subgroup]{Theorem \ref*{thm:subgroup}} is \emph{trivial} when $q = 1$ and $H = \Cl_D^+$ is the whole narrow class group, for then for every fixed continuity set $B \subset \Gamma_0(q) \backslash \Hb$, we have the \emph{equality}
\begin{equation}
\label{eqn:equality}
\frac{\sum_{A \in CH} \vol(\FF_A(q) \cap \Gamma_0(q) B)}{\sum_{A \in CH} \vol(\FF_A(q))} = \frac{\vol(B)}{\vol(\Gamma_0(q) \backslash \Hb)}
\end{equation}
with \emph{no} error term. Moreover, if $q = 1$ and $H = (\Cl_D^+)^2$, so that $CH \in \Gen_D$ is a genus, then we also have the equality \eqref{eqn:equality} if $J$ lies in the principal genus $(\Cl_D^+)^2 \in \Gen_D$, where $J \in \Cl_D^+$ is the narrow ideal class containing the different $\df \coloneqq (\sqrt{D})$, so that $J^2$ is the principal narrow ideal class $I$; this occurs if and only if $D$ is not divisible by a prime congruent to $3$ modulo $4$.

The reason for the trivial equality \eqref{eqn:equality} is simply due to the fact that the pair $\Gamma_A \backslash \NN_A$ and $\Gamma_{JA^{-1}}(1) \backslash \NN_{JA^{-1}}(1)$ are \emph{complementary}, in the sense that their union covers $\Gamma \backslash \Hb$ evenly and the images of their boundary geodesics are the same as sets but with opposite orientations. As a consequence, we have the following result.

\begin{corollary}
\label{cor:trivial}
Let $q = 1$ and let $D$ be a positive fundamental discriminant. Let $CH$ be a coset of $\Cl_D^+$ with $H$ a subgroup of $\Cl_D^+$ and $C \in \Cl_D^+$. Then we have the equality \eqref{eqn:equality} for every fixed continuity set $B \subset \Gamma \backslash \Hb$ if $C^2 J \in H$.
\end{corollary}

\begin{proof}
The oriented closed geodesics $\CC_A(1)$ and $\CC_{JA^{-1}}(1)$ are the same curve with opposite orientations, which means that $\Gamma_A(1) \backslash \NN_A(1)$ and $\Gamma_{JA^{-1}}(1) \backslash \NN_{JA^{-1}}(1)$ cover $\Gamma_0(1) \backslash \Hb$ evenly. Thus \eqref{eqn:equality} holds if $JA^{-1} \in CH$ for every $A \in CH$. This condition is met precisely when $C^2 J \in H$.
\end{proof}

For $q \neq 1$, on the other hand, it is no longer the case that $\Gamma_A(q) \backslash \NN_A(q)$ and $\Gamma_{JA^{-1}}(q) \backslash \NN_{JA^{-1}}(q)$ are complementary. Indeed, if a narrow ideal class $A$ is associated to a Heegner form $Q = [a,b,c] \in \QQ_D(q)$, where $q > 1$ and $b \equiv r \pmod{2q}$ for some fixed residue class $r$ modulo $2q$ such that $r^2 \equiv D \pmod{4q}$, then there does \emph{not} exist a Heegner form $Q' = [a',b',c'] \in \QQ_D(q)$ with $b' \equiv r \pmod{2q}$ such that $\Gamma_Q(q) \backslash \NN_Q(q)$ and $\Gamma_{Q'}(q) \backslash \NN_{Q'}(q)$ are complementary. Instead, the complementary Heegner form is $[-a,-b,-c] \in \QQ_D(q)$, which is such that $-b \equiv -r \pmod{2q}$, and hence does not correspond to a narrow ideal class $A \in CH$ appearing in the sums in \eqref{eqn:equality}. As a notable consequence, \hyperref[thm:subgroup]{Theorem \ref*{thm:subgroup}} is nontrivial if $q > 1$ and $H = \Cl_D^+$.

\section{Small Scale Equidistribution and Discrepancy Bounds}

\subsection{Automorphic Kernels and Selberg--Harish-Chandra Transforms}

For $z,w \in \Hb$, set
\[\rho(z,w) \coloneqq \log \frac{\left|z - \overline{w}\right| + |z - w|}{\left|z - \overline{w}\right| - |z - w|}, \qquad u(z,w) \coloneqq \frac{|z - w|^2}{4 \Im(z) \Im(w)} = \sinh^2 \frac{\rho(z,w)}{2}.\]
The function $u : \Hb \times \Hb \to [0,\infty)$ is a point-pair invariant for the symmetric space $\Hb \cong \SL_2(\R) / \SO(2)$; that is, $u(g z, g w) = u(z,w)$ for all $g \in \SL_2(\R)$ and $z,w \in \Hb$. From this, a function $k : [0,\infty) \to \C$ gives rise to a point-pair invariant $k(u(z,w))$ on $\Hb$.

We take $k(u(z,w)) = k_R(u(z,w))$ to be equal to the indicator function of a ball of radius $R$ centred at a point $w$,
\[B_R(w) \coloneqq \{z \in \Hb : \rho(z,w) \leq R\} = \left\{z \in \Hb : u(z,w) \leq \sinh^2 \frac{R}{2}\right\},\]
normalised by the volume of this ball,
\[\vol(B_R) = 4\pi \sinh^2 \frac{R}{2},\]
namely
\[k_R(u(z,w)) \coloneqq \begin{dcases*}
\dfrac{1}{\vol(B_R)} & if $u(z,w) \leq \sinh^2 \dfrac{R}{2}$,	\\
0 & otherwise.
\end{dcases*}\]
We additionally consider the convolution
\[k \ast k'(u(z,w)) = \int_{\Hb} k(u(z,\zeta)) k'(u(\zeta,w)) \, d\mu(\zeta)\]
of two point-pair invariants $k,k'$ on $\Hb$.

\begin{lemma}
\label{lem:kconv}
For $0 < \rho < R$, the convolution $k_R \ast k_{\rho}(u(z,w))$ is nonnegative, bounded by $\frac{1}{\vol(B_R)}$, and satisfies
\[k_R \ast k_{\rho}(u(z,w)) = \begin{dcases*}
\frac{1}{\vol(B_R)} & if $u(z,w) \leq \sinh^2 \frac{R - \rho}{2}$,	\\
0 & if $u(z,w) \geq \sinh^2 \frac{R + \rho}{2}$,
\end{dcases*}\]
so that
\[\frac{\vol(B_{R - \rho})}{\vol(B_R)} k_{R - \rho} \ast k_{\rho}(u(z,w)) \leq k_R(u(z,w)) \leq \frac{\vol(B_{R + \rho})}{\vol(B_R)} k_{R + \rho} \ast k_{\rho}(u(z,w))\]
for all $z,w \in \Hb$.
\end{lemma}

\begin{proof}
This follows from the triangle inequality for the distance function $\rho(z,w)$.
\end{proof}

Given $k : [0,\infty) \to \C$, we define the automorphic kernel $K : \Gamma \backslash \Hb \times \Gamma \backslash \Hb \to \C$ by
\[K(z,w) \coloneqq \sum_{\gamma \in \Gamma} k(u(\gamma z,w)).\]
We write
\[K_R(z,w) \coloneqq \sum_{\gamma \in \Gamma} k_R(u(\gamma z,w)).\]
This is the indicator function of an injective geodesic ball in $\Gamma \backslash \Hb$ provided that $k_R(u(\gamma_1 z,w)) = k_R(u(\gamma_2 z,w)) = 1$ if and only if $\gamma_1 = \gamma_2$. A necessary condition for this is the requirement that $2 \Ht(w) \sinh R \leq 1$, where $\Ht(w) \coloneqq \max_{\gamma \in \Gamma} \Im(\gamma w)$. For $0 < \rho < R$, we consider the convolution kernel
\[K_R \ast K_{\rho}(z,w) \coloneqq \sum_{\gamma \in \Gamma} k_R \ast k_{\rho}(u(\gamma z,w)).\]

\begin{corollary}
\label{cor:kconv}
For $0 < \rho < R$, the convolution kernel $K_R \ast K_{\rho}(z,w)$ is nonnegative and satisfies
\begin{equation}
\label{eqn:kconv}
\frac{\vol(B_{R - \rho})}{\vol(B_R)} K_{R - \rho} \ast K_{\rho}(z,w) \leq K_R(z,w) \leq \frac{\vol(B_{R + \rho})}{\vol(B_R)} K_{R + \rho} \ast K_{\rho}(z,w)
\end{equation}
for all $z,w \in \Gamma \backslash \Hb$ satisfying $2 \Ht(w) \sinh R \leq 1$.
\end{corollary}

The spectral expansion for the convolution kernel $K_{R \pm \rho} \ast K_{\rho}$ involves a sum over an orthonormal basis $\BB_0(\Gamma)$ of the space of Maa\ss{} cusp forms (which we may choose to consist of Hecke--Maa\ss{} eigenforms), where the inner product is
\[\langle f,g\rangle \coloneqq \int_{\Gamma \backslash \Hb} f(z) \overline{g(z)} \, d\mu(z),\]
and an integral over $t \in \R$ indexing the Eisenstein series $E(z,\frac{1}{2} + it)$. It also involves the Selberg--Harish-Chandra transforms $h_{R \pm \rho}$ of $k_{R \pm \rho}$ and $h_{\rho}$ of $k_{\rho}$. The Selberg--Harish-Chandra transform takes sufficiently well-behaved functions $k : [0,\infty) \to \C$ to functions $h : \R \cup i[-\frac{1}{2},\frac{1}{2}] \to \C$ via
\[h(t) \coloneqq 2\pi \int_{0}^{\infty} P_{-\frac{1}{2} + it}(\cosh \rho) k\left(\sinh^2 \frac{\rho}{2}\right) \sinh \rho \, d\rho,\]
where $P_{\lambda}^{\mu}(z)$ denotes the associated Legendre function. In particular,
\[h_R(t) = \frac{2\pi}{\vol(B_R)} \int_{0}^{R} P_{-\frac{1}{2} + it}(\cosh \rho) \sinh \rho \, d\rho.\]
The size of this transform is given by the following well-known bounds.

\begin{lemma}[{\cite[Lemma 2.33]{HR22}}]
\label{lem:hRupperbounds}
Suppose that $0 < R \leq 1/\e$ for some fixed $\e > 0$. For $t \in \R$,
\begin{equation}
\label{eqn:hRupperbounds}
h_{R}(t) \ll \begin{dcases*}
1 & for $|t| \leq \dfrac{1}{R}$,	\\
R^{-\frac{3}{2}} |t|^{-\frac{3}{2}} & for $|t| \geq \dfrac{1}{R}$.
\end{dcases*}
\end{equation}
\end{lemma}

The advantage of convolving is that it smooths the point-pair invariant and improves the decay of the Selberg--Harish-Chandra transform, since the Selberg--Harish-Chandra transform of the convolution $k_1 \ast k_2$ is the product $h_1(t) h_2(t)$ of the individual Selberg--Harish-Chandra transforms. This ensures that the convolution kernel has a spectral expansion on $L^2(\Gamma \backslash \Hb)$ that not only converges in $L^2$ but uniformly.

\begin{lemma}[{\cite[Theorems 1.14 and 7.4]{Iwa02}}]
\label{lem:kerneluniform}
The automorphic kernel $K_R$ satisfies
\begin{align*}
\int_{\Gamma \backslash \Hb} K_R(z,w) \, d\mu(z) & = h_R\left(\frac{i}{2}\right) = 1,	\\
\int_{\Gamma \backslash \Hb} f(z) K_R(z,w) \, d\mu(z) & = h_R(t_f) f(w),	\\
\int_{\Gamma \backslash \Hb} E\left(z,\frac{1}{2} + it\right) K_R(z,w) \, d\mu(z) & = h_R(t) E\left(w,\frac{1}{2} + it\right)
\end{align*}
for every $f \in \BB_0(\Gamma)$, $t \in \R$, and $w \in \Gamma \backslash \Hb$. Moreover, the convolved kernel $K_{R \pm \rho} \ast K_{\rho}$ has the spectral expansion
\begin{multline}
\label{eqn:KastKspectral}
K_{R \pm \rho} \ast K_{\rho}(z,w) = \frac{1}{\vol(\Gamma \backslash \Hb)} + \sum_{f \in \BB_0(\Gamma)} h_{R \pm \rho}(t_f) h_{\rho}(t_f) f(z) \overline{f(w)}	\\
+ \frac{1}{4\pi} \int_{-\infty}^{\infty} h_{R \pm \rho}(t) h_{\rho}(t) E\left(z,\frac{1}{2} + it\right) \overline{E\left(w,\frac{1}{2} + it\right)} \, dt,
\end{multline}
which converges absolutely and uniformly.
\end{lemma}

\subsection{Bounds for Fractional Moments of \texorpdfstring{$L$}{L}-Functions}

The proofs of \hyperref[thm:smallscale]{Theorems \ref*{thm:smallscale}} and \ref{thm:discrepancy} require the following bounds for fractional moments of certain $L$-functions.

\begin{lemma}
\label{lem:fractionalmoment}
Let $D_1$ and $D_2$ be odd squarefree fundamental discriminants and let $\chi_{D_1}$ and $\chi_{D_2}$ be the primitive quadratic Dirichlet characters modulo $|D_1|$ and $|D_2|$. Then for $T \geq 1$ and $w \in \Gamma \backslash \Hb$, we have that
\begin{multline}
\label{eqn:fractionalmoment}
\begin{rcases*}
\displaystyle \sum_{\substack{f \in \BB_0(\Gamma) \\ T \leq t_f \leq 2T}} \sqrt{\frac{L\left(\frac{1}{2},f \otimes \chi_{D_1}\right) L\left(\frac{1}{2},f \otimes \chi_{D_2}\right)}{L(1,\ad f)}} |f(w)|	\\
\displaystyle \int\limits_{T \leq |t| \leq 2T} \left|\frac{L\left(\frac{1}{2} + it, \chi_{D_1}\right) L\left(\frac{1}{2} + it, \chi_{D_2}\right)}{\zeta(1 + 2it)}\right| \left|E\left(w,\frac{1}{2} + it\right)\right| \, dt
\end{rcases*}	\\
\ll_\e \begin{dcases*}
T^{\frac{3}{2}} \left(T^{\frac{1}{2}} + \Ht(w)^{\frac{1}{2}}\right) D^{\frac{1}{6} + \e} & if $T \leq D^{\frac{1}{12}}$,	\\
T^{\frac{1}{2}} \left(T^{\frac{1}{2}} + \Ht(w)^{\frac{1}{2}}\right) D^{\frac{1}{4} + \e} & if $D^{\frac{1}{12}} \leq T \leq D^{\frac{1}{4}}$,	\\
T^{\frac{3}{2} + \e} \left(T^{\frac{1}{2}} + \Ht(w)^{\frac{1}{2}}\right) & if $T \geq D^{\frac{1}{4}}$.
\end{dcases*}
\end{multline}
Assuming the generalised Lindel\"{o}f hypothesis, we have the improved bounds
\begin{equation}
\label{eqn:fractionalmomentGLH}
\begin{rcases*}
\displaystyle \sum_{\substack{f \in \BB_0(\Gamma) \\ T \leq t_f \leq 2T}} \sqrt{\frac{L\left(\frac{1}{2},f \otimes \chi_{D_1}\right) L\left(\frac{1}{2},f \otimes \chi_{D_2}\right)}{L(1,\ad f)}} |f(w)|	\\
\displaystyle \int\limits_{T \leq |t| \leq 2T} \left|\frac{L\left(\frac{1}{2} + it, \chi_{D_1}\right) L\left(\frac{1}{2} + it, \chi_{D_2}\right)}{\zeta(1 + 2it)}\right| \left|E\left(w,\frac{1}{2} + it\right)\right| \, dt
\end{rcases*} \ll_{\e} T^{\frac{3}{2} + \e} \left(T^{\frac{1}{2}} + \Ht(w)^{\frac{1}{2}}\right) D^{\e}
\end{equation}
for all $T \geq 1$.
\end{lemma}

\begin{proof}
We apply the Cauchy--Schwarz inequality and use the local Weyl law \cite[Proposition 7.2]{Iwa02}
\[\begin{rcases*}
\displaystyle \sum_{t_f \leq T} |f(w)|^2	\\
\displaystyle \int_{-T}^{T} \left|E\left(w,\frac{1}{2} + it\right)\right|^2 \, dt
\end{rcases*}
\ll T^2 + T \Ht(w)\]
and the bound
\begin{equation}
\label{eqn:HRbounds}
\begin{rcases*}
\displaystyle \sum_{\substack{f \in \BB_0(\Gamma) \\ T \leq t_f \leq 2T}} \frac{L\left(\frac{1}{2},f \otimes \chi_{D_1}\right) L\left(\frac{1}{2},f \otimes \chi_{D_2}\right)}{L(1,\ad f)}	\\
\displaystyle \int\limits_{T \leq |t| \leq 2T} \left|\frac{L\left(\frac{1}{2} + it,\chi_{D_1}\right)^2 L\left(\frac{1}{2} + it,\chi_{D_2}\right)^2}{\zeta(1 + 2it)}\right|^2 \, dt
\end{rcases*} \ll_{\e} \begin{dcases*}
T^2 D^{\frac{1}{3} + \e} & if $T \leq D^{\frac{1}{12}}$,	\\
D^{\frac{1}{2} + \e} & if $D^{\frac{1}{12}} \leq T \leq D^{\frac{1}{4}}$,	\\
T^{2 + \e} & if $T \geq D^{\frac{1}{4}}$.
\end{dcases*}
\end{equation}
This is proven for $D_1 = 1$ and $D_2 = D$ in \cite[Proposition 2.14]{HR22}. The same method of proof yields the above bound, with the only notable difference being a slightly different application of the Vorono\u{\i} summation formula for
\[\sum_{m = 1}^{\infty} \frac{\lambda_{\chi_{D_1},\chi_{D_2}}(m,0) e\left(\frac{md}{c}\right)}{m^s},\]
where $\lambda_{\chi_{D_1,D_2}}(m,0) \coloneqq \sum_{ab = m} \chi_{D_1}(a) \chi_{D_2}(b)$, which is given in \cite[Appendix A]{LT05}. Assuming the generalised Lindel\"{o}f hypothesis, the Weyl law implies the improved bounds $O_{\e}(T^{2 + \e} D^{\e})$ for all $T \geq 1$ for \eqref{eqn:HRbounds}.
\end{proof}

\subsection{Proofs of \texorpdfstring{\hyperref[thm:smallscale]{Theorems \ref*{thm:smallscale}}}{Theorems \ref{thm:smallscale}} and \ref{thm:discrepancy}}

\begin{proof}[Proof of {\hyperref[thm:smallscale]{Theorem \ref*{thm:smallscale}}}]
For $w \in \Gamma \backslash \Hb$, the difference
\begin{equation}
\label{eqn:smallscalesqueeze}
\frac{\vol(\Gamma \backslash \Hb)}{\vol(B_R)} \frac{\sum_{A \in G_D} \vol(\FF_A \cap \Gamma B_R(w))}{\sum_{A \in G_D} \vol(\FF_A)} - 1
\end{equation}
is equal to
\[\frac{\vol(\Gamma \backslash \Hb)}{\sum_{A \in G_D} \vol(\FF_A)} \int_{\FF_A} K_R(z,w) \, d\mu(z) - 1.\]
By the upper bound \eqref{eqn:kconv}, the spectral expansion \eqref{eqn:KastKspectral}, and character orthogonality, we see that for any $0 < \rho < R$, \eqref{eqn:smallscalesqueeze} is bounded from above by the sum of the three terms
\begin{gather}
\label{eqn:firstterm2}
\frac{\vol(B_{R + \rho})}{\vol(B_R)} - 1,	\\
\label{eqn:secondterm2}
\frac{\vol(B_{R + \rho}) \vol(\Gamma \backslash \Hb)}{\vol(B_R) \sum_{A \in G_D} \vol(\FF_A)} \frac{1}{2^{\omega(D) - 1}} \sum_{\chi \in \widehat{\Gen_D}} \chi(G_D) \sum_{f \in \BB_0(\Gamma)} h_{R + \rho}(t_f) h_{\rho}(t_f) \overline{f(w)} W_{\chi,f},	\\
\label{eqn:thirdterm}
\frac{\vol(B_{R + \rho}) \vol(\Gamma \backslash \Hb)}{\vol(B_R) \sum_{A \in G_D} \vol(\FF_A)} \frac{1}{2^{\omega(D) - 1}} \sum_{\chi \in \widehat{\Gen_D}} \chi(G_D) \frac{1}{4\pi} \int_{-\infty}^{\infty} h_{R + \rho}(t) h_{\rho}(t) \overline{E\left(w,\frac{1}{2} + it\right)} W_{\chi,t} \, dt.
\end{gather}
Here the Weyl sums $W_{\chi,f}$ and $W_{\chi,t}$ are as in \eqref{eqn:WeylMaass} and \eqref{eqn:WeylEis}. Similarly, \eqref{eqn:smallscalesqueeze} is bounded from below by the sum of the same three terms except with $R + \rho$ replaced by $R - \rho$ at each instance. Thus to prove \hyperref[thm:smallscale]{Theorem \ref*{thm:smallscale}}, it suffices to show that for $R > D^{-\delta}$ with $\delta < \frac{1}{2}$, there exists some choice of $\rho \in (0,R)$ for which each of the three terms \eqref{eqn:firstterm2}, \eqref{eqn:secondterm2}, and \eqref{eqn:thirdterm} is $o(1)$.

The term \eqref{eqn:firstterm2} is $O(\rho R^{-1})$ since $\vol(B_R) \asymp R^2$. For the terms \eqref{eqn:secondterm2} and \eqref{eqn:thirdterm}, we begin by recalling that the set of genus characters $\chi \in \widehat{\Gen_D}$ is indexed by ordered pairs of fundamental discriminants $(D_1,D_2)$ such that $D_1 D_2 = D$; each genus character $\chi$ is then associated to the pair of primitive quadratic Dirichlet characters $\chi_{D_1}$ and $\chi_{D_2}$ modulo $|D_1|$ and $|D_2|$ respectively. From \hyperref[prop:MaassnewformWeyl]{Proposition \ref*{prop:MaassnewformWeyl}} and Stirling's formula, we deduce that
\begin{align*}
\left|W_{\chi,f}\right|^2 & \ll \frac{\sqrt{D}}{(1 + |t_f|)^3} \frac{L\left(\frac{1}{2},f \otimes \chi_{D_1}\right) L\left(\frac{1}{2}, f \otimes \chi_{D_2}\right)}{L(1,\ad f)},	\\
\left|W_{\chi,t}\right|^2 & \ll \frac{\sqrt{D}}{(1 + |t|)^3} \frac{L\left(\frac{1}{2} + it,\chi_{D_1}\right) L\left(\frac{1}{2} - it,\chi_{D_1}\right) L\left(\frac{1}{2} + it,\chi_{D_2}\right) L\left(\frac{1}{2} - it,\chi_{D_2}\right)}{\zeta(1 + 2it) \zeta(1 - 2it)}.
\end{align*}
We also use the bound $\sum_{A \in G_D} \vol(\FF_A) \gg_{\e} D^{\frac{1}{2} - \e}$ \cite[Proposition 1]{DIT16}. We then dyadically divide up the sum over $f \in \BB_0(\Gamma)$ and the integral over $t \in \R$ based on the size of $t_f,|t| \in [T,2T]$ and employ the bounds \eqref{eqn:hRupperbounds} for $h_{R \pm \rho}$ and $h_{\rho}$ and the bounds \eqref{eqn:fractionalmoment} for the ensuing fractional moment of $L$-functions. It remains to choose
\[\rho = \begin{dcases*}
R^{\frac{1}{2}} D^{-\frac{1}{12}} & if $R \geq D^{-\frac{1}{12}}$,	\\
R D^{-\frac{1}{24}} & if $D^{-\frac{5}{12}} \leq R \leq D^{-\frac{1}{12}}$,	\\
\frac{1}{2} R^{\frac{1}{2}} D^{-\frac{1}{4}} & if $D^{-\frac{1}{2}} \leq R \leq D^{-\frac{5}{12}}$,	\\
\frac{1}{2} R & if $0 < R \leq D^{-\frac{1}{2}}$.
\end{dcases*}\]
In this way, we find that
\begin{equation}
\label{eqn:smallscalebound}
\left|\frac{\vol(\Gamma \backslash \Hb)}{\vol(B_R)} \frac{\sum_{A \in G_D} \vol(\FF_A \cap \Gamma B_R(w))}{\sum_{A \in G_D} \vol(\FF_A)} - 1\right| \ll_{\e} \begin{dcases*}
R^{-\frac{1}{2}} D^{-\frac{1}{12} + \e} & if $D^{-\frac{1}{12}} \leq R \leq 1$,	\\
D^{-\frac{1}{24} + \e} & if $D^{-\frac{5}{12}} \leq R \leq D^{-\frac{1}{12}}$,	\\
R^{-\frac{1}{2}} D^{-\frac{1}{4} + \e} & if $R \leq D^{-\frac{5}{12}}$.
\end{dcases*}
\end{equation}
This is $o(1)$ provided that $R \geq D^{-\delta}$ for some $\delta < \frac{1}{2}$.
\end{proof}

\begin{proof}[Proof of {\hyperref[thm:discrepancy]{Theorem \ref*{thm:discrepancy}}}]
From \eqref{eqn:smallscalebound} and the fact that $\vol(B_r) \asymp R^2$ for $R \leq 1$, we have that for any $w \in \Gamma \backslash \Hb$,
\[\left|\frac{\sum_{A \in G_D} \vol(\FF_A \cap \Gamma B_R(w))}{\sum_{A \in G_D} \vol(\FF_A)} - \frac{\vol(B_R)}{\vol(\Gamma \backslash \Hb)} \right| \ll_{\e} \begin{dcases*}
R^{\frac{3}{2}} D^{-\frac{1}{12} + \e} & if $D^{-\frac{1}{12}} \leq R \leq 1$,	\\
R^2 D^{-\frac{1}{24} + \e} & if $D^{-\frac{5}{12}} \leq R \leq D^{-\frac{1}{12}}$,	\\
R^{\frac{3}{2}} D^{-\frac{1}{4} + \e} & if $R \leq D^{-\frac{5}{12}}$.
\end{dcases*}\]
In particular,
\[\sup_{B_R(w) \subset \Gamma \backslash \Hb} \left|\frac{\sum_{A \in G_D} \vol(\FF_A \cap \Gamma B_R(w))}{\sum_{A \in G_D} \vol(\FF_A)} - \frac{\vol(B_R)}{\vol(\Gamma \backslash \Hb)} \right| \ll_{\e} D^{-\frac{1}{12} + \e}.\]

This result may be strengthened under the assumption of the generalised Lindel\"{o}f hypothesis by conditionally improving the bounds \eqref{eqn:smallscalebound}. Instead of using the unconditional bounds \eqref{eqn:fractionalmoment}, we use the conditional bounds \eqref{eqn:fractionalmomentGLH}. Choosing
\[\rho = \begin{dcases*}
\frac{1}{2} R^{\frac{1}{2}} D^{-\frac{1}{4}} & if $D^{-\frac{1}{2}} \leq R \leq 1$,	\\
\frac{1}{2} R & if $R \leq D^{-\frac{1}{2}}$,
\end{dcases*}\]
we find that
\[\left|\frac{\sum_{A \in G_D} \vol(\FF_A \cap \Gamma B_R(w))}{\sum_{A \in G_D} \vol(\FF_A)} - \frac{\vol(B_R)}{\vol(\Gamma \backslash \Hb)} \right| \ll_{\e} R^{\frac{3}{2}} D^{-\frac{1}{4} + \e}\]
for all $R \leq 1$, which yields the conditional bound $O_{\e}(D^{-1/4 + \e})$ for the discrepancy.
\end{proof}

\section{Questions}

We end by raising some natural follow-up questions beyond the results contained in this paper.

\begin{question}
Can one prove \hyperref[thm:subgroup]{Theorem \ref*{thm:subgroup}} for $\delta = 0$ using techniques purely from arithmetic ergodic theory?
\end{question}

This question was raised by Duke, Imamo\={g}lu, and T\'{o}th \cite[Section 4]{DIT16}. Indeed, Duke's theorems on the equidistribution of closed geodesics and of Heegner points on the modular surface and additionally of lattice points on the sphere \cite{Duk88} were first proved under the additional hypothesis of a splitting condition by Linnik \cite{Lin68} using techniques from ergodic theory. Einsiedler, Lindenstrauss, Michel, and Venkatesh \cite{ELMV12} showed that in the case of closed geodesics, Linnik's ergodic method is valid even without the splitting condition hypothesis.

\begin{question}\hspace{1em}

\begin{enumerate}[leftmargin=*,label=\textup{(\arabic*)}]
\item For any fixed $\delta > 0$, does there exist a sequence of positive fundamental discriminants $D$ and \emph{subsets} $H$ of $\Cl_D^+$ that are not necessarily cosets with $\frac{|H|}{h_D^+} \ll D^{-\delta}$ for which the equidistribution result \eqref{eqn:subgroupequidistribution} in \hyperref[thm:subgroup]{Theorem \ref*{thm:subgroup}} \emph{fails} to hold?
\item Does the equidistribution result \eqref{eqn:subgroupequidistribution} in \hyperref[thm:subgroup]{Theorem \ref*{thm:subgroup}} hold for arbitrary \emph{subsets} $H$ of $\Cl_D^+$ that are not necessarily cosets provided that $\frac{|H|}{h_D^+} \log D$ tends to infinity with $D$?
\end{enumerate}
\end{question}

The condition that $CH$ be a coset of $\Cl_D^+$ in \hyperref[thm:subgroup]{Theorem \ref*{thm:subgroup}} may be thought of as imposing the requirement that we restrict to a subset of $\Cl_D^+$ with an \emph{algebraic} structure. By comparing to related results on the sparse equidistribution of closed geodesics, namely \cite[Theorem 4.1]{AE16} and \cite[Theorem 1.8]{BoKo17}, we expect that such an algebraic condition is necessary in order for equidistribution to hold if the cardinality of this subset is $O(D^{-\delta} h_D^+)$ for some $\delta > 0$. On the other hand, analogous results for closed geodesics, namely \cite[Theorem 1.2]{AE16}, lead us to expect that \hyperref[thm:subgroup]{Theorem \ref*{thm:subgroup}} holds for subsets $H$ of $\Cl_D^+$ \emph{without} any algebraic structure provided that $\frac{|H|}{h_D^+} \log D$ tends to infinity with $D$.

\begin{question}
Does the \emph{joint equidistribution} of hyperbolic orbifolds hold?
\end{question}

Here by joint equidistribution, we mean the following. For each positive fundamental discriminant $D$, choose a genus $G_D$ in the group of genera $\Gen_D$ and choose $C \in \Cl_D^+$ such that the minimal norm of any integral ideal representing $C$ tends to infinity as $D$ tends to infinity. Joint equidistribution of hyperbolic orbifolds is the statement that for each fixed continuity set $B \subset \Gamma \backslash \Hb \times \Gamma \backslash \Hb$,
\[\frac{\sum_{A \in G_D} \vol(\FF_A \times \FF_{AC} \cap (\Gamma \times \Gamma) B)}{\sum_{A \in G_D} \vol(\FF_A \times \FF_{AC})} = \frac{\vol(B)}{\vol(\Gamma \backslash \Hb \times \Gamma \backslash \Hb)} + o_{B}(1)\]
as $D$ tends to infinity. This is a natural analogue of the mixing conjecture of Michel and Venkatesh \cite{MV06} for Heegner points, which has been resolved by Khayutin \cite[Theorem 1.3]{Kha19} under the assumptions of a splitting hypothesis and of the nonexistence of Landau--Siegel zeroes. The first author resolved a toy model of this problem in the much simpler setting of modular inverses on the torus \cite[Theorem 1.6]{Hum22}.

Finally, we briefly mention that one can pose a closely related variant of joint equidistribution, namely \emph{simultaneous equidistribution}. Associated to each narrow ideal class of a quadratic field $\Q(\sqrt{D})$ are geometric invariants for \emph{each} quaternion algebra that is unramified at each place that splits in $E$. Michel and Venkatesh \cite{MV06} pose the question of whether such geometric invariants equidistribute \emph{simultaneously} in \emph{multiple} different quaternion algebras. For progress on this problem and variants thereof, we direct the reader to \cite{AES16,ALMW22,BB24,BBK22,EL19}.

\end{document}